\documentclass[11pt]{article}
\usepackage{amsmath,amssymb,amsthm,bm}
\usepackage{mathtools}
\usepackage{dsfont}
\usepackage{verbatim}
\usepackage{graphicx,rotating } 
\usepackage{xcolor}
\usepackage{natbib}
\usepackage{ulem} % crossing out text  
\usepackage{enumitem} % shift itemize  environment
\usepackage{pstricks,pst-node,pst-plot}
\usepackage{xr}
\usepackage[colorlinks,citecolor=blue,urlcolor=blue]{hyperref}
\usepackage{multirow}
\usepackage{placeins} 
\usepackage{comment}
\usepackage{tikz}
\usepackage[framemethod=tikz]{mdframed}
\newmdenv[tikzsetting={draw=black, line width=0.5pt, dash pattern=on 1pt off 1pt, dashed}, linecolor=white, outerlinewidth=1pt]{examplebox}

\newcommand{\cA}{\mathcal{A}}
\newcommand{\cB}{\mathcal{B}}
\newcommand{\cC}{\mathcal{C}}

\newcommand{\cE}{\mathcal{E}}
\newcommand{\cF}{\mathcal{F}}

\newcommand{\cH}{\mathcal{H}}

\newcommand{\cL}{\mathcal{L}}

\newcommand{\cN}{\mathcal{N}}

\newcommand{\cP}{\mathcal{P}}

\newcommand{\cT}{\mathcal{T}}
\newcommand{\cU}{\mathcal{U}}

\newtheorem{theorem}{Theorem}%[section]
\newtheorem{lemma}{Lemma}%[section]
\newtheorem{proposition}{Proposition}%[section]
\newtheorem{corollary}{Corollary}%[section]

\theoremstyle{definition}
\newtheorem{definition}{Definition}%[section]
\theoremstyle{plain}

\theoremstyle{remark}

\newtheorem{remark}{Remark}

\newcommand{\Prb}{\mathbb{P}}
 
\newcommand{\figurepath}{}

\newcommand*{\myov}[1]{\overbracket[0.8pt][0pt]{#1}}
\newcommand*{\myund}[1]{\underbracket[0.8pt][0pt]{#1}}
\let\limsup\undefined
\let\liminf\undefined
\DeclareMathOperator*\liminf{\myund{\lim}}
\DeclareMathOperator*\limsup{\myov{\lim}}

\begin{document}

\title{SCoRE sets: A versatile framework for simultaneous inference}
\author{Fabian J.E. Telschow$^1$, Junting Ren$^2$ and
	Armin Schwartzman$^{2,3}$ \\[5mm]
	$^1$Department of Mathematics, Humboldt University of Berlin \\
	$^2$Division of Biostatistics, University of California, San Diego \\
	$^3$Hal{\i}c{\i}o\u{g}lu Data Science Institute, University of California, San Diego \\[5mm]
}
\date{\today}
\maketitle
 
%========================================================================
% Abstract
%========================================================================
\begin{abstract}
	We study asymptotic statistical inference in the space of bounded
functions endowed with the supremums norm over an arbitrary metric
space $S$ using a novel concept: Simultaneous COnfidence
Region of Excursion (SCoRE) Sets.
They simultaneously quantify the uncertainty of several lower and upper
excursion sets of a target function. We investigate their connection to
multiple hypothesis tests controlling the familywise error rate in the strong
sense and show that they grant a unifying perspective on several statistical
inference tools such as simultaneous confidence bands, quantification of
uncertainties in level set estimation, for example, CoPE sets, and multiple
hypothesis testing over $S$, for example,
finding relevant differences or regions of equivalence within $S$.
In particular, our abstract setting allows us to refine and reduce the assumptions in recent
articles on CoPE sets and relevance and equivalence testing using the supremums norm.
\end{abstract}

%========================================================================
% Main document
%========================================================================´

\section{Introduction}
Historically there has been a large body of work connecting hypothesis tests
and confidence sets. The earliest work on this topic is \cite{Neyman:1937confidence} which
developed the well known \textit{duality between families of point hypothesis tests and
confidence sets}, compare \cite[Thm 3.5.1]{Lehmann:2005testing} for a modern treatment.
Later in \cite{Aitchison:1964confidence,Aitchison:1965likelihood,Gabriel:1969simultaneous}
simultaneous confidence sets have been used to derive tests for more complex hypotheses,
especially in multiple testing. The opposite direction starting from stagewise multiple testing
procedures and deriving simultaneous confidence sets has been studied, among others, in
\cite{Stefansson:1988, Hayter:1994relationship, Holm:1999, Guilbaud:2008, Magirr:2013}.

Most of the above works have in common that they treat statistical hypothesis testing as the fundamental
paradigm and view confidence sets as a derived concept.
However, there is an intuitive appeal of confidence intervals over hypothesis testing
which is nicely expressed in R. Little's comment to the ASA statement on the $p$-value
\cite{Wasserstein:2016_ASA}:
\textit{``[...] I teach a basic course in biostatistics to public health students.
Confidence intervals are no problem--ideas like margin of error have even entered the vernacular.
The difficulties begin with hypothesis testing. [...]"}.
Implicitly, this intuition appeared as well in the
works on equivalence testing where the null hypothesis is
that a parameter (for example a population mean) is not contained
in a known interval because the first equivalence tests were based on confidence intervals
\cite{Westlake:1972, Schuirmann:1981}. Only later tests have been derived
from the intersection-union principle \cite{Schuirmann:1987,
Hauck:1992types}. A thoughtful discussion of the connection between confidence intervals
and equivalence tests and possible pitfalls is presented in \cite{Berger:1996}.

\paragraph{Our Contributions.}
In this article we aim to strengthen the confidence set viewpoint.
We assume that the following mathematical objects are given: an unknown target function $\mu:S\rightarrow \mathbb{R}$,
an estimator $\hat\mu$ of $\mu$ that satisfies a weakened version of a uniform limit theorem (ULT)
in the space of bounded functions $\ell^\infty(S)$ endowed with the supremums norm and
two sets $\cA$ and $\cB$ of real-valued functions over $S$.
Our main theorem (Theorem \ref{thm:MainSCoPES}) then provides the limit distribution
that the canonical estimates $\hat\cU_a$ and $\hat\cL_b$ derived from the ULT
of the upper $\cU_a$ and lower $\cL_b$ excursion sets above $a$ and below $b$ of $\mu$, respectively, simultaneously satisfy
$\hat\cU_a \subseteq \cU_a$ and $\hat\cL_b \subseteq \cL_b$ for all $a\in\cA$ and $b\in\cB$. The concept is visualized in Figure \ref{Fig:ScoPESconcept} in the Appendix.
Our Theorem \ref{thm:MainSCoPES} and its corollaries generalize coverage probability excursion (CoPE) sets \cite{Sommerfeld:2018CoPE,Bowring:2019,Bowring:2021} and yield a more general version of the view on simultaneous confidence bands that we introduced in \cite{Ren:2022}.
In particular, we resolve some shortcomings in \cite{Sommerfeld:2018CoPE} which
led to unnecessary assumptions on $\mu$.

Our second main contribution is to clarify the connection between confidence regions for
excursion sets and multiple hypotheses tests that control the familywise error rate (FWER)
in the strong sense. Especially, we give a useful duality in Proposition \ref{prop:duality}.
From this perspective we develop another interpretation of Theorem \ref{thm:MainSCoPES} as
providing the \textit{oracle limit distribution} of a class of multiple testing strategies
that asymptotically control the FWER in the strong sense, compare Section \ref{sec:duality}.
Furthermore, we show in Theorem \ref{thm:Dette_Generalized} and \ref{thm:Equiv_2sides}
that many recently proposed asymptotic global relevance and equivalence tests for functional
data based on the supremums norm, among others,
\cite{Dette2021FunctionOnFunction,Dette:2022cov,Dette:2021bio}, can be derived
and generalized from the viewpoint of our Theorem \ref{thm:MainSCoPES}.
Even better, our established duality allows us to construct from Theorem \ref{thm:MainSCoPES}
local relevance and local equivalence tests that control the FWER
in the strong sense, compare Theorem \ref{thm:RelTubeTest} and \ref{thm:lequivTest}.
To the best of our knowledge the latter have not been discussed in the literature yet.

A third contribution of our article is discussing a slight shift in the interpretation of statistical
inference on excursion sets from the viewpoint of our duality between confidence sets for excursion
sets and multiple hypothesis testing which we exemplify in
a simple multiple linear regression example in Section \ref{scn:Scheffe}.

Finally, we want to point out that we aware of the fact that the assumption that
an estimator $\hat\mu$ satisfies a ULT might be too strong. We introduce this assumption mainly
to compare to the current literature and derive explicit limit distributions.
The core of our results, however, is the SCoRE set Metatheorem, Theorem \ref{thm:SCoPESMetatheorem},
which requires considerably weaker assumptions.

%\subsection{Connections to the Literature}
\paragraph{Connections to the Literature.}
The first work we are aware of which quantifies the
uncertainty of the lower and the upper excursion set above zero
of a function $\mu$ defined on $\mathbb{R}^D$ from a statistical viewpoint is \cite{Mammen:2013}.
Using an estimator $\hat\mu_N$ of $\mu$ they construct from data a lower and an upper
excursion set $\hat{\mathcal{L}}$ and $\hat{\mathcal{U}}$ above zero
respectively and prove in their Lemma 2.1. nonasymptotic bounds such that both
inclusions $ \hat{\mathcal{L}} \subseteq \mathcal{L}_0$
and $ \hat{\mathcal{U}} \subseteq \mathcal{U}_0$ hold true. This can be viewed
as a special case of our Proposition \ref{prop:NonAsymMetaTheorem}
from Appendix \ref{scn:MetaTheorem}.
In their Theorem 3.1 they apply these sets to quantify the uncertainty of the excursion sets
of a kernel density estimator above a single $c\in\mathbb{R}$ asymptotically.
They show that under certain assumptions on $\mu$, for example, a non-degeneracy condition
of the gradient of $\mu$ the above inclusion is asymptotically valid, if the parameter for constructing
$\hat{\mathcal{L}}$ and $\hat{\mathcal{U}}$ is estimated using the bootstrap.
Noteworthy this result does not explicitly extract the asymptotic distribution.
Another more recent work which applies similar ideas to kernel density estimators is \cite{Qiao:2019}.
In particular, they derive assumptions and rates for having asymptotically
nominal coverage and give a broad
overview on applications of level and excursion set estimation.

The first article that explicitly derives limit distributions for confidence regions of
excursion sets for general estimators $\hat\mu$ of $\mu$ satisfying a central limit theorem
in $C(S)$ is \cite{Sommerfeld:2018CoPE}. This paper has some shortcomings,
for example, $\mu$ is not allowed to be tangential to the level $c$,
\cite[Assumption 2.1.(a) and Lemma 1]{Sommerfeld:2018CoPE} which is a similar condition as the
Assumption (A.ii) required in Theorem 3.1. of \cite{Mammen:2013}.
However, in \cite{Sommerfeld:2018CoPE} this condition is only needed because of an imprecise definition,
which also prevents that their theory can be connected easily to multiple hypothesis testing controlling
the FWER.
We explain this in more detail in Section \ref{scn:selectionSCoPES} and Appendix
\ref{scn:ChoiceInclusion}. Note that this definition persists in the applications of their work
to geoscience \cite{French:2017assessing} and neuroimaging \cite{Bowring:2019,Bowring:2021}
and has also been used in the innovative work \cite{Maullin:2022} which generalizes \cite{Sommerfeld:2018CoPE} to
intersections and unions of excursion sets of several functions
$\mu^1,\ldots, \mu^K\in a(s)$, $K\in\mathbb{N}$, above a single $c\in\mathbb{R}$.
Our Corollary \ref{cor:MaxImproved} generalizes the main theorem from \cite{Sommerfeld:2018CoPE}
as it removes their Assumption 2.1.(a), allows for non-constant thresholds $c\in\ell^\infty(S)$ and
can be interpreted in terms of a multiple hypothesis test.

% Junting
To date only \cite{Ren:2022} is dealing with several
excursion sets at the same time.
Their main theorem shows that properly thresholding SCBs yields
simultaneous confidence regions for all lower and upper excursion sets over $c\in\mathbb{R}$. 
We generalize their main result in our Proposition \ref{prop:CoPEasSCB} and embed it into
the testing literature.
Moreover, our Corollary \ref{cor:SCB_CoPE} can be viewed as its asymptotic generalization as it
connects asymptotic $(1-\alpha)$-SCBs (among others, 
\cite{Degras:2011SCB, Telschow:2022SCB}) and simultaneous confidence regions for excursion sets.
In principle even the fast and fair SCBs \cite{Liebl:2019fast} can be fitted into our framework.
Here we need to fall back to the SCoRE set Metatheorem, since their key innovation is that the
quantile parameter $q$ is a function, which enables them to adapt the width of the SCBs not only
to the variance, yet also to the local correlation. The benefit of this is that
invalidation of the coverage can be spread fairly over a partition of $S = [0,1]$.
We do not include this result here, since we restrict ourselves to constant $q$ for the sake of simplicity. 

% Hypothesis testing
Last but not least there is a less obvious connection to the recently proposed
relevance tests \cite{Dette:2020functional, Dette2021FunctionOnFunction, Dette:2022cov}
and equivalence tests \cite{Dette:2021bio, Dette2018:equivalenceReg}
in the space of continuous functions over $S = [0,1]$.
These articles use the test statistic $\Vert \hat\mu_N \Vert_\infty = \sup_{s\in S} \vert \hat\mu_N(s) \vert$
and derive using a CLT of $\hat\mu$ in $C(S)$ its limiting distributions
under the null and alternative hypotheses which depend on the set of extreme points of $\mu$.
We explain in Section \ref{scn:testing} how these test procedures can be derived from
Corollary \ref{cor:Extraction} and clarify that they only control the FWER in the
weak sense at level $\alpha$.
In our Theorems \ref{thm:RelTubeTest} and \ref{thm:lequivTest} we improve these testing strategies
by modifying them to control the FWER in the strong sense at level $\alpha$.
%Furthermore, we discuss novel, yet similar hypothesis tests such as local relevance and
%equivalence tests which to the best of our knowledge have not yet been proposed in the literature.

%\subsection{Organization of the Article}
\paragraph{Organization of the Article.}
In Section \ref{scn:notations} we introduce notations and definitions required
to understand our main results.
In Section \ref{scn:SimCoPE} we explain our main theorem,
its corollaries and the required assumptions. It also contains a general strategy
to consistently estimate generalized preimages --a concept that appears in our main theorem--
and outlines a general strategy to estimate the required quantile parameter.
The connection between SCoRE sets and statistical hypothesis testing are explained
in Section \ref{scn:testing}. In particular, we state our duality to certain multiple hypothesis
tests in Section \ref{sec:duality} and the oracle limit distribution in Section \ref{sec:oracle}.
In Section \ref{scn:Scheffe} we explain the interpretation of statistical inference based on
SCoRE sets using a simple example from linear regression. The article finishes with
a discussion, Section \ref{scn:discussion}, of some consequences of our results and ideas for future work.

%-------------------------------------------------------------------------------------------------------
\section{Notations and Definitions}\label{scn:notations}
%-------------------------------------------------------------------------------------------------------
In this article $ ( S, d ) $ denotes a metric space.
%For any $B\subset S$ we write $B^\complement$ for the complement
$S \setminus B$, ${\rm cl}B$ for the topological closure,
${\rm int}B$ for the interior and $\partial B = {\rm cl}B\setminus{\rm int}B$
for the topological boundary of $B$.
The set $ \mathcal{F}( S ) $ denotes the set of functions
$ f: S  \rightarrow \mathbb{R}\cup \{ \pm \infty \} $.
The set $\ell^\infty(S) \subset \mathcal{F}( S )$ is the set of all bounded functions
$ f \in \mathcal{F}( S )$, i.e., $\Vert f \Vert_\infty = \sup_{s\in S} \vert f(s) \vert < \infty$
and $ C( S ) \subseteq \mathcal{F}( S ) $ is the subset of continuous functions
with respect to the topology generated by the metric $d$.
If $ f \in \mathcal{F}(S) $ and $ r \in \mathbb{R} \cup \{ \pm \infty \} $, we write $ f \equiv r $,
if $ f $ is the constant function with value $ r $ and if no confusion is possible we identify
$r$ with the constant function with value $r$.
%-------------------------------------------------------------------------------------------------------
For any $f \in \mathcal{F}(S)$ we define, as usual,
\begin{equation}\label{eq:supinf_conv}
	\sup_{s \in \emptyset } f(s) = -\infty ~~\text{ and }~~
	\inf_{s \in \emptyset } f(s) = \infty\,.
\end{equation}

%-------------------------------------------------------------------------------------------------------

Let $( \Omega, \mathfrak{P},\Prb) $ be a probability space.
Our results are based on the J. Hoffmann-J{\o}rgensen theory of weak convergence
as elaborated in the first chapters of \citet{Vaart:1996weak}.
Recall that the \textit{inner probability} of a set $V\subset \Omega$ is given by
$\mathbb{P}_*( V ) = \sup\{\, \mathbb{P}(W) ~\vert~ W \subseteq V\,, W \in \mathfrak{P} \,\}$,
and its \textit{outer probability} by
 $\mathbb{P}^*( V ) = \inf\{\, \mathbb{P}(W) ~\vert~ W \supseteq V\,, W \in \mathfrak{P} \,\}$.
We use repeatedly the statements $(i)$ and $(ii)$ of the Portmanteau Theorem
\citep[Theorem 1.3.4]{Vaart:1996weak}. Hence we introduce the following shorter notation
to simplify bounds on the the limes inferior ($\liminf$) of inner probabilities and
limes superior ($\limsup$) of outer probabilities.
\begin{definition}
	Let $\Omega_N \subseteq \Omega$ be a sequence of sets, $X:\Omega\rightarrow \mathbb{R}$
	a random variable and $q \in \mathbb{R}$. We write
	\begin{equation*}
		\lim_{N\rightarrow\infty} \mathbb{P}_*\big[ \Omega_N \big] = \mathbb{P}\big[ X \prec q \big]\,,
	\end{equation*}
	under Assumptions {\rm\textbf{(A)}} / {\rm\textbf{(B)}} if the following two statements hold:
	\begin{equation*}
	\begin{split}
	&\text{Under Assumptions {\rm\textbf{(A)}}}:~ ~ ~
		\liminf_{N \rightarrow \infty}
			\mathbb{P}_*\big[ \Omega_N \big]
		 	\geq \mathbb{P}\big[ X < q \big]\\
	&\text{Under Assumptions {\rm\textbf{(B)}}}:~ ~ ~
	\limsup_{N \rightarrow \infty}
		\mathbb{P}^*\big[ \Omega_N \big]
		 \leq \mathbb{P}\big[ X \leq q \big].
	\end{split}
	\end{equation*}
\end{definition}

%-------------------------------------------------------------------------------------------------------
To compactly state our main result, Theorem \ref{thm:MainSCoPES}, we require generalized notions of
preimages and graphs of functions. We introduce these concepts hereafter. A visualization of these
concepts can be found in Figure \ref{fig:fig0}.
Recall that the graph of $ f \in \mathcal{F}(S) $ is the set
$
	\Gamma(f) = \big\{(s,r) \in S \times \mathbb{R}~\vert~ r = f(s) \big\}\,.
$
\begin{definition}[\textbf{Graph of a Set of Functions}]\label{def:graph}
	For any  $ \cH \subseteq \mathcal{F}(S) $ define the \textit{graph of $\cH$} to be the union of
	the graphs of its elements, i.e.,
	\begin{equation*}
		\Gamma(\cH) = \bigcup_{ h \in \cH} \Gamma(h)\,.
	\end{equation*}
	As a further convention we will write $h \in \Gamma(\cH)$ if $ \Gamma(h) \subseteq \Gamma(\cH) $
	for $h \in \mathcal{F}(S)$.
\end{definition}

%-------------------------------------------------------------------------------------------------------
\begin{definition}[\textbf{Preimage of a Set of Functions}]\label{def:Cpreimage}
	For any $\cH \subseteq \mathcal{F}(S) $ and $ g \in \mathcal{F}(S) $
	define the \textit{preimage of $\cH$ under $g$} by
$$
	g^{-1}_{\cH}
			= \big\{\, s \in S ~\vert~ \exists h \in \cH:~ g(s) = h(s) \,\big\}\,.
$$
	If $ \cH = \{ h \} $, $ h \in \mathcal{F}(S) $, the abbreviation
	$ g^{-1}_{\{ h \}} = g^{-1}_h $ is used.
\end{definition}
%-------------------------------------------------------------------------------------------------------
If $\cH \not\subset \cC(S)$ or $ g \notin \cC(S) $ we need a
generalized concept of a preimage of a set of functions in the sense that we 
add all ``touching points" of $\Gamma(g)$ and $\Gamma(\cH)$ to the preimage.
Making this idea mathematically precise requires the thickening
of a set $\cH\subseteq\cF(S)$.
%-------------------------------------------------------------------------------------------------------
%\begin{definition}[\textbf{Thickening of a Set of Functions}]\label{def:thickening}
%	For $ \cH \subseteq \mathcal{F}(S) $ and $\eta > 0$ we call
%	$
%	\cH_\eta = \big\{\, f + \varepsilon \in \mathcal{F}(S)~\vert~
%								f \in \cH,~ \vert \varepsilon \vert \leq \eta
%						\,\big\}
%	$ the \textit{$\eta$-thickening of $\cH$}.
%\end{definition}
\begin{definition}[\textbf{Thickenings of a Set of Functions}]\label{def:thickening}
	For $ \cH \subseteq \cF(S) $ and $\eta > 0$ the set
	$
	\cH_{\pm \eta} = \big\{\, h \pm \varepsilon \in \cF(S)~\vert~
								h \in \cH,~ 0 \leq \varepsilon \leq \eta
						\,\big\}
	$ is the $\pm\eta$-thickening of $\cH$, while $
	\cH_{\eta} = \cH_{-\eta} \cup \cH_{+\eta}$ is called
	the	$\eta$-thickening of $\cH$.
\end{definition}

\begin{definition}[\textbf{Generalized Preimage of a Set of Functions}]\label{def:genPreimage}
	For $\cH \subseteq \mathcal{F}(S)$ and $ g \in \mathcal{F}(S) $ we call the set
	\begin{equation*}
		\mathfrak{g}^{\pm}_\cH
		= \bigcap_{\eta > 0} {\rm cl}\,g^{-1}_{\cH_{\pm\eta}}
	\end{equation*}
	the \textit{upper(+)/lower(-) generalized preimage of $\cH$ under $g$}.
%	 and the set
%	$\mathfrak{g}^{-1}_\cH = \mathfrak{g}^{+}_\cH \cup \mathfrak{g}^{-}_\cH$
%	the generalized preimage of $\cH$ under $g$.
\end{definition}
\begin{remark}\label{rmk:SetConvergence}
	The terminology generalized preimage is reasonable as
	$ g^{-1}_\cH \subseteq \mathfrak{g}^{\pm}_\cH \subseteq \mathfrak{g}^{-1}_\cH$.
	The main reason why the generalized preimage appears later on is that for $S$ being compact and
	$\mathfrak{g}^{\pm}_\cH \neq \emptyset$, it holds that
	\begin{equation}\label{eq:HausdorffConvergence}
		\lim_{N\rightarrow\infty}
		d_H\Big( {\rm cl}\,g^{-1}_{\cH_{\pm\eta_N}}, \mathfrak{g}^{\pm}_\cH \Big) = 0\,,
	\end{equation}
	for any positive sequence
	$ (\eta_N)_{N\in\mathbb{N}} \subset\mathbb{R} $
	converging to zero as $ g^{-1}_{\cH_{\pm\eta}} \subseteq g^{-1}_{\cH_{\pm\eta'}} $
	for $ \eta' \geq \eta $, see \cite[Corollary 5.30]{Tuzhilin:2020}. Moreover, $\mathfrak{g}^{\pm}_\cH$
	is the unique set with this property since it is closed. If there would be another $A \subseteq S$
	satisfying \eqref{eq:HausdorffConvergence}, it holds by the triangle inequality that
	\begin{equation*}
		d_H\left( \mathfrak{g}^{\pm}_\cH, A \right)
			\leq \lim_{N\rightarrow 0} d_H\left( {\rm cl}\,g^{-1}_{\cH_{\pm\eta_N}} , \mathfrak{g}^{\pm}_\cH \right) +
				d_H\left( {\rm cl}\,g^{-1}_{\cH_{\pm\eta_N}} , A \right)
			= 0
	\end{equation*}
	and therefore $\mathfrak{g}^{\pm}_\cH = {\rm cl}A$, compare \cite[Problem 5.1.(3)]{Tuzhilin:2020}.
\end{remark}
%-------------------------------------------------------------------------------------------------------

\begin{figure}[h]\centering
			\includegraphics[width=0.49\textwidth]{\figurepath 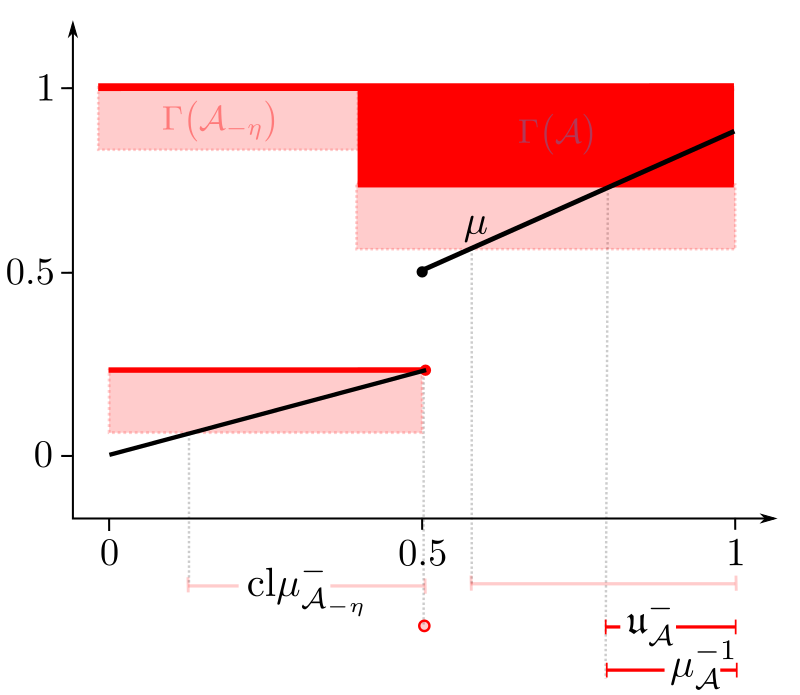}
		\caption{ Visualization of Definitions \ref{def:graph}-\ref{def:genPreimage}.
				  Black: the target function $\mu$.
				  Red: the graph $\Gamma(\cA)$ of a set of functions $\cA$.
				  Light red: the graph $\Gamma(\cA_{-\eta})$ of lower $\eta$-thickening $\cA_{-\eta}$ of $\cA$.
				  Moreover, the preimages and lower generalized preimages of $\cA$ under $\mu$ are shown below the
				  $x$-axis.
			    }\label{fig:fig0}
\end{figure}
\section{An Asymptotic SCoRE Set Theorem}\label{scn:SimCoPE}
%-------------------------------------------------------------------------------------------------------

In this section we state and discuss our main theorems. The main assumption will be that
the estimators $\hat\mu_N:~ \Omega \rightarrow \ell^\infty(S) $ of $\mu\in \ell^\infty(S)$
satisfy a uniform limit theorem (ULT) in $\ell^\infty( U )$
for some appropriately chosen $U\subseteq S$. Although we call $\hat\mu_N$ an estimator
we do not in general assume that it is measurable.
Our main objects of interest are excursion sets:
\begin{definition}
	The \textit{lower and upper excursion sets of $\mu$ over $f\in \mathcal{F}( S )$} are
	\begin{equation*}
	\begin{split}
		&\cL_f = \big\{ s\in  S ~\vert~ \mu(s) < f(s) \big\}\,,~ ~ ~ ~ 
		\cU_f  = \big\{ s\in  S ~\vert~ \mu(s) > f(s) \big\}\,.
	\end{split}
	\end{equation*}
	If $\mu$ is replaced by $\hat\mu_N$, we call the originating set-valued functions
	$\hat{\cL}_f$ and $\hat{\cU}_f$ the lower and upper excursion sets of $\hat\mu_N$
	over $f$.	
\end{definition}

The notation used in our main theorem and its corollaries can be simplified, if we define
\begin{equation}\label{eq:LimitForm}
\begin{split}
	T_{\cA, \cB}( f )
	&= \max\Bigg\{ \sup_{s \in \mu^{-1}_\cA } f(s)
				,~ \sup_{s \in \mu^{-1}_\cB} -f(s) \Bigg\}\\
	\mathfrak{T}_{\cA, \cB}( f )
	&= \max\Bigg\{ \sup_{s \in \mathfrak{u}^{-}_\cA} f(s),
				~ \sup_{s \in \mathfrak{u}^{+}_\cB} -f(s) \Bigg\}
\end{split}
\end{equation}
for $\cA, \cB \subseteq \mathcal{F}(S)$ and $\mu,f\in\ell^\infty(S)$.
Here $\mathfrak{u}^{\pm}_\cH$ is the upper/lower generalized preimage of $\cH\subseteq\mathcal{F}(S)$ under $\mu$.
We abbreviate $T_{\cH, \cH}( f ) = T_{\cH}( f )$ and
 $\mathfrak{T}_{\cH, \cH}( f ) = \mathfrak{T}_{\cH}( f )$.

%-------------------------------------------------------------------------------------------------------
\subsection{Assumptions}
%-------------------------------------------------------------------------------------------------------
The next definition specifies what kind of uniform limit theorem our main result requires.
%-------------------------------------------------------------------------------------------------------
\begin{definition}\label{def:SuffWeakConv}
	Let $ \sigma \in \ell^\infty(S)$ be such that $0 < \mathfrak{o} < \sigma(s) < \mathfrak{O} < \infty$
	for all $s\in S$ and $(\tau_N)_{N\in\mathbb{N}} \subset \mathbb{R}$ a positive sequence converging to zero.
	We denote with $(G_N)_{N \in \mathbb{N}}$ the sequence
	\begin{equation}\label{eq:FiniteG}
		G_N = \frac{\hat\mu_N - \mu}{\tau_N\sigma}\,.
	\end{equation}
	Let $ \cA, \cB \subseteq \mathcal{F}(S) $, $\tilde\eta>0$ and $\cA_{-\tilde\eta}$ and
	$\cB_{+\tilde\eta}$ be the one-sided thickenings of $\cA$ and $\cB$
	respectively as defined in Definition \ref{def:thickening}.
	We say that an estimator $ \hat\mu_N$ of $\mu$ with values in $\ell^\infty(S)$ fulfills a
	\textit{uniform limit theorem on $\cA_{-\tilde\eta}\cup \cB_{+\tilde\eta}$}
	or short a $\cA_{-\tilde\eta}$-$\cB_{+\tilde\eta}$-ULT, if the following conditions hold:
	\begin{enumerate}
		\item[(i)] There exist a tight, Borel measurable
			  $G:\Omega\rightarrow\ell^\infty\big( {\rm cl}\,\mu^{-1}_{\cA_{-\tilde{\eta}}} \cup {\rm cl}\,\mu^{-1}_{\cB_{+\tilde{\eta}}} \big)$
			  such that
			  $
				G_N \rightsquigarrow G
			  $
			  weakly in
			  $\ell^\infty\big( {\rm cl}\,\mu^{-1}_{\cA_{-\tilde{\eta}}} \cup {\rm cl}\,\mu^{-1}_{\cB_{+\tilde{\eta}}} \big)$ in
			  the sense of \citep[Definition 1.3.3]{Vaart:1996weak}.
			  Here we silently identified $G_N$ and its restriction to
			  ${\rm cl}\,\mu^{-1}_{\cA_{-\tilde{\eta}}} \cup {\rm cl}\,\mu^{-1}_{\cB_{+\tilde{\eta}}}$.
		\item[(ii)] There is a constant $ K > 0 $ and a sequence $ (Z_N)_{N\in\cN} $
					of functions $\Omega\rightarrow \ell^\infty(S)$ with
					$ \inf_{s \in S } \tau_N^{-1}Z_N(s) $ being asymptotically
					tight\footnote{In the sense of  \citep[p.21]{Vaart:1996weak}}
					such that $\Prb_*$ almost surely for all $ a \in \cA $ and all $ b \in \cB $
					it holds that
			    	\begin{equation}\label{eq:ass_farzone}
						\big(\, \hat\mu_N(s) - a(s) \,\big)
			    		\cdot{\rm sgn}\big(\, \mu(s) - a(s) \,\big)
						\geq \sigma(s) \big(\, K  + Z_N(s) \,\big)
					\end{equation}
    				for all $ s \in S\, \setminus {\rm cl}\,\mu^{-1}_{\cA_{-\tilde{\eta}}} $
					and
					\begin{equation*}
						\big(\, \hat\mu_N(s) - a(s) \,\big)
			    		\cdot{\rm sgn}\big(\, \mu(s) - a(s) \,\big)
						\geq \sigma(s) \big(\, K  + Z_N(s) \,\big)
					\end{equation*}
    				for all $ s \in S\, \setminus {\rm cl}\,\mu^{-1}_{\cB_{+\tilde{\eta}}} $.
	\end{enumerate}		
\end{definition}
\begin{remark}\label{rmk:SignInequal}
	Condition $(ii)$ is an adaptation of Assumption 2.1.(c) from \cite{Sommerfeld:2018CoPE}
	and weakens the assumption of having a ULT as in $(i)$ on all of $S$
	Essentially it assumes that the sequence of estimators $\hat\mu_N$
	have the correct sign for all $s \in S$ sufficiently far away from the
	generalized preimage $\mathfrak{u}^{-1}_\cA$ and $\mathfrak{u}^{-1}_\cB$.
	This follows for $\mathfrak{u}^{-1}_\cA$ from 
	\begin{equation*}
	\begin{split}
		\limsup_{N\rightarrow\infty}& \,\mathbb{P}^*\Big[\,
							\exists s \in S \setminus \mu^{-1}_{\cA_{\tilde{\eta}}}:~
							{\rm sgn} \big(\, \mu(s) - a(s) \,\big)
							= - {\rm sgn} \big(\, \hat\mu_N(s) - a(s) \,\big)
					\,\Big]\\
		 &\leq\limsup_{N\rightarrow\infty}\mathbb{P}_*\Big[\,
		 					\exists s \in S \setminus \mu^{-1}_{\cA_{\tilde{\eta}}}:~
		 					- \big\vert\, \hat\mu_N(s) - a(s) \,\big\vert
		 						\geq \sigma(s)\big( K + Z_N(s) \big)
		 				\,\Big]\\
		 &\leq \limsup_{N\rightarrow\infty}\mathbb{P}^*\Big[\,  \inf_{s \in S} Z_N(s) \leq -K \,\Big]\\
		 &\leq 1 - \liminf_{N\rightarrow\infty}\mathbb{P}_*\Big[\, \inf_{s \in S} Z_N(s) \in [-K, K] \,\Big]
		 = 0\,,
	\end{split}
	\end{equation*}
	where we used \eqref{eq:ass_farzone} in the first inequality and $ \inf_{s \in S }\tau_N^{-1} Z_N(s)$
	being asymptotically tight in the last equality.
\end{remark}
%-------------------------------------------------------------------------------------------------------
Let $\cA,\cB \in \cF(S)$ and $\tilde\eta> 0$. The following assumptions are used in our main result:
\begin{itemize}[leftmargin=1.4cm]
	\item[\textbf{(A1)}] The estimator $\hat\mu_N$ of $\mu$ satisfies an
						 $\cA_{-\tilde\eta}$-$\cB_{+\tilde\eta}$-ULT.
	\item[\textbf{(A2)}] $G$ and $G_N$ restricted to
						 $\big( {\rm cl}\,\mu^{-1}_{\cA_{-\tilde{\eta}}} \setminus {\rm int}\,\mathfrak{u}^{-}_\cA
						 \big) \cup \big( {\rm cl}\,\mu^{-1}_{\cB_{+\tilde{\eta}}} \setminus {\rm int}\,\mathfrak{u}^{+}_\cB \big)$
						 have almost surely
						 continuous sample paths.
	\item[\textbf{(A3)}] ${\rm cl}\,\mu^{-1}_{\cA_{-\tilde{\eta}}}$ and ${\rm cl}\,\mu^{-1}_{\cB_{+\tilde{\eta}}}$
						 are compact.
	\item[\textbf{(A4)}] ${\rm cl}\, \mu^{-1}_\cA = \mathfrak{u}^{-}_\cA$ and
						 ${\rm cl}\, \mu^{-1}_\cB = \mathfrak{u}^{+}_\cB$.
%	 MARKER
%	
%	there is an open neighborhood
%						$V \supset \partial \mu^{-1}_\cC \setminus \mu^{-1}_\cC$ such that $G_N(s)$ restricted to
%						$V \cap \mu^{-1}_\cC$ has almost surely continuous sample paths.
\end{itemize}
%-------------------------------------------------------------------------------------------------------
\begin{remark}
	Assumptions \textbf{(A1)}-\textbf{(A3)} are required to prove the lower bounds on the confidence statements
	in our main theorem.
	They ensure that for $\cA, \cB \subseteq \mathcal{F}(S)$ and
	$(\eta_N)_{N \in \mathbb{N}}$ converging to zero, that
	\begin{equation*}
		T_{ {\rm cl}\,\cA_{\eta_N}, {\rm cl}\,\cB_{\eta_N} }( G_N )
	 		\rightsquigarrow \mathfrak{T}_{ \cA, \cB }( G ) \text{ weakly in }\mathbb{R}\,,
	 \end{equation*}
	 by Lemma \ref{lem:Tweak} in Appendix \ref{App:WeakMaxConvergence}.
\end{remark}
%-------------------------------------------------------------------------------------------------------
\begin{remark}\label{rmk:A3}
	The compactness assumption in \textbf{(A3)} can be relaxed. The requirement in the proof is that
	for any sequence $(\eta_N)_{N\in\mathbb{N}}$ converging to zero we have that
	${\rm cl}\,\mu^{-1}_{\cA_{-\eta_N}}$ and
	${\rm cl}\,\mu^{-1}_{\cB_{+\eta_N}}$ are converging in Hausdorff distance to $\mathfrak{u}^{-}_\cA$
	and $\mathfrak{u}^{+}_\cB$ respectively,
	which is implied by \textbf{(A3)} as discussed in Remark \ref{rmk:SetConvergence}.
\end{remark}
%-------------------------------------------------------------------------------------------------------
\begin{remark}
	Assumption \textbf{(A4)} yields that the lower and upper bound of the confidence statement
	in our main theorem	are derived from the same limiting process.
	In particular, \textbf{(A4)} is trivially satisfied if $\cA,\cB\subseteq C(S)$ and $\mu\in C(S)$,
	compare Appendix \ref{App:frakSeqS}.
	However, if \textbf{(A4)} is not satisfied, there exists a slightly different limiting process
	which yields an upper bound similar to the upper bound in Theorem \ref{thm:MainSCoPES},
	compare Remark \ref{rmk:WeakerUpperBound}.
\end{remark}

%-------------------------------------------------------------------------------------------------------

\begin{figure}[h]
			\includegraphics[width=0.49\textwidth]{\figurepath 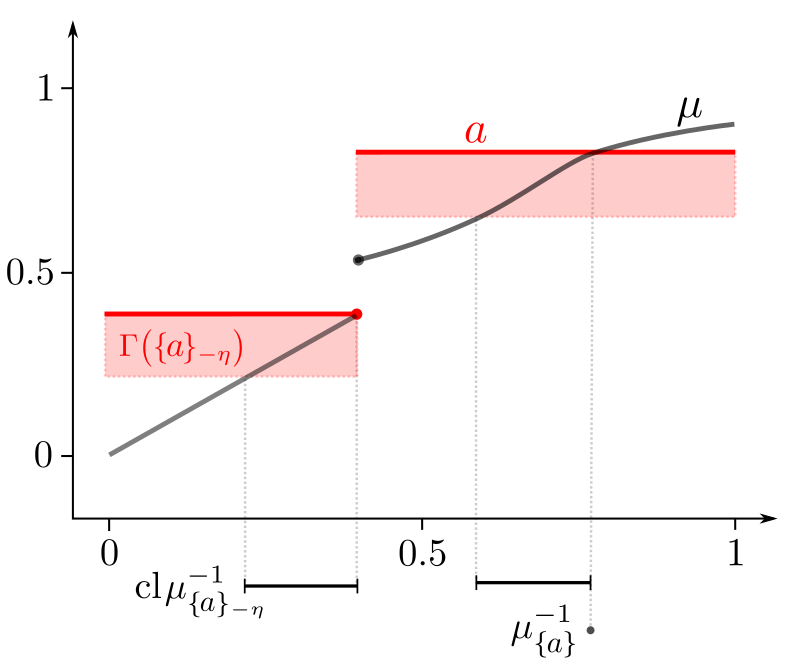}
			\includegraphics[width=0.49\textwidth]{\figurepath 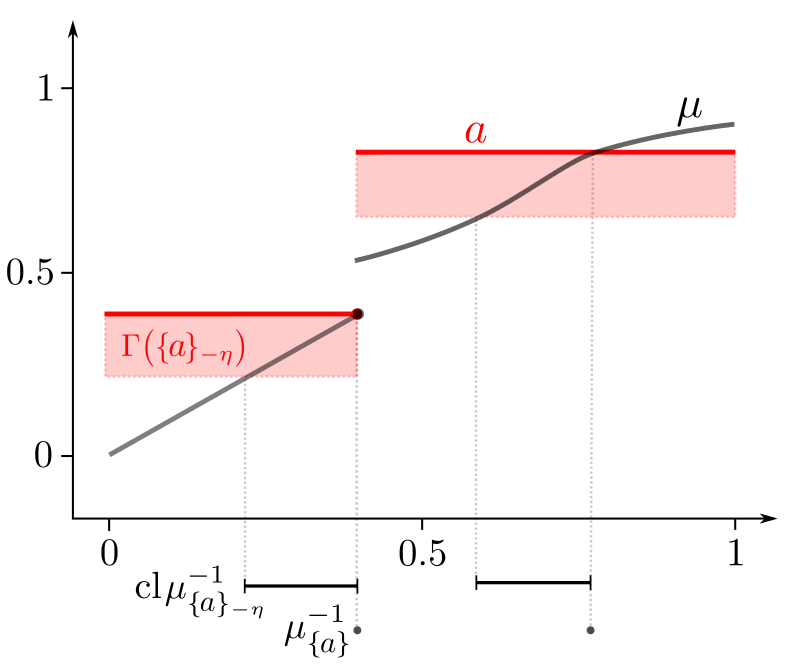}
		\caption{ Illustration that $\mu^{-1}_{\{a\}}$ is not necessarily the
				 Hausdorff-limit of $\mu^{-1}_{ \{a\}_{-\eta} }$.
			     The problem are "touching points" of $\Gamma(\mu)$ and
				 $\Gamma(a)$ such as $s_0\approx 0.4$, if $\mu$ or $a$ are discontinuous.
				\textit{Left:} 
					As $\mu(s_0) \neq a(s_0)$ it holds
			    	$s_0\notin \mu^{-1}_{\{c\}}$. However,
			    	for each $\eta >0$ there is an $s_\eta \in \mu^{-1}_{{\{a\}}_\eta}$
			    	arbitrary close to $s_0$. Thus, the Hausdorff
			    	convergence $\mu^{-1}_{{\{a\}}_{-\eta}}$ to $\mu^{-1}_{{\{a\}}}$ fails.
			    \textit{Right:}	    
			    	As $b(s_0) = \mu(s_0)$ it holds that $\mu^{-1}_{{\{a\}}_{-\eta}}$ converges in Hausdorff distance
			    	to $\mu^{-1}_{\{a\}}$ as $\eta\rightarrow 0$.
			    }\label{fig:ContMuProblem}
\end{figure}
%-------------------------------------------------------------------------------------------------------

%------------------------------------------------------------------------------------------------------- 

%-------------------------------------------------------------------------------------------------------
\subsection{An Asymptotic SCoRE Set Theorem}\label{scn:selectionSCoPES}
%-------------------------------------------------------------------------------------------------------
The next result is our main theorem. It can be viewed as a Corollary of the
SCoRE set Metatheorem for random variables satisfying a ULT, see Appendix \ref{scn:MetaTheorem}.
%-------------------------------------------------------------------------------------------------------
\begin{theorem}\label{thm:MainSCoPES}
	Let $ \cA, \cB \subseteq \mathcal{F}(S) $ and $q\in\mathbb{R}$. Assume {\rm\textbf{(A1)}}.
	Then
\begin{equation*}
	\lim_{N \rightarrow \infty}
		\mathbb{P}_* \Big[ \forall a \in \cA\,\forall b\in \cB:
									~ \hat{\cU}_{a + q\tau_N\sigma} \subseteq \cU_{a}
											~\wedge~
									  \hat{\cL}_{b - q\tau_N\sigma} \subseteq \cL_{b}
							   \,\Big]
		 = \mathbb{P}\big[\, \mathfrak{T}_{\cA, \cB}( G ) \prec q \,\big]
\end{equation*}
under {\rm\textbf{(A2)}},{\rm\textbf{(A3)}} / {\rm\textbf{(A4)}}.
\end{theorem}
%-------------------------------------------------------------------------------------------------------
\begin{remark}
	The above theorem is valid even if $\mathfrak{u}^{-}_{\cA} \cup \mathfrak{u}^{+}_{\cB} = \emptyset$.
	In this case the inner probability that the inclusions on the l.h.s. hold true simultaneously
	converges to $1$ for all $q\in \mathbb{R}$,
	which is consistent with our notation as
	\begin{equation*}
		\mathbb{P}\big[\,  \mathfrak{T}_{\emptyset, \emptyset}( G ) < q \,\big]
		= \mathbb{P}\left[  -\infty < q \,\right] = 1\,.
	\end{equation*}
\end{remark}
%
%-------------------------------------------------------------------------------------------------------
\begin{remark}
	Since $\mathfrak{u}^{-1}_{\cC} = \mathfrak{u}^{-1}_{\cC'}$
	for $\cC' = \big\{ c \in \mathcal{F}(S)~\vert~ c \in \Gamma(\cC) \big\}$ we could replace
	$\cA,\cB$ by $\cA', \cB'$ in the above theorem without changing the r.h.s..
	Neither does the r.h.s. change if we  add any of the inclusions
	$S \setminus \hat{\cL}_{a + q\tau_N\sigma} \subseteq S \setminus \cL_{a}$, $a\in\cA$, or
	$S \setminus \hat{\cU}_{b - q\tau_N\sigma} \subseteq S \setminus \cU_{b}$, $b\in\cB$, to the probability
	statement on the l.h.s., compare Lemma \ref{lem:SCoPESlemma}.
\end{remark}
%-------------------------------------------------------------------------------------------------------
\begin{remark}
	Assumption \textbf{(A4)} cannot be weakened to include any
	$s_0 \in \mathfrak{u}^{-1}_{\cA} \setminus {\rm cl}\,\mu^{-1}_{\cA}$ or $s_0 \in 
	\mathfrak{u}^{-1}_{\cB} \setminus {\rm cl}\,\mu^{-1}_{\cB}$ without getting
	a weaker upper bound.
	The reason is that the sharp upper bound requires that, for some $a \in \cA$ or $b \in \cB$,
	$\mu(s_0) \notin \big[ \hat\mu_N - q\tau_N\sigma(s_0), \infty \big)$
	for an $s_0\in \mathfrak{u}^{-1}_{\cA}$
	implies that
	$
	\hat{\cU}_{a + q\tau_N\sigma} \not\subseteq \cU_{a} 
	$,
	or $\mu(s_0) \notin \big( -\infty, \hat\mu_N - q\tau_N\sigma(s_0) \big]$
	for an $s_0\in \mathfrak{u}^{-1}_{\cB}$
	implies that
	$\hat{\cL}_{b - q\tau_N\sigma} \not\subseteq \cL_{b}
	$. This, however, cannot be guaranteed
	at points of discontinuity of $\mu$, $a$ or $b$, even if $G_N$ has continuous sample paths,
	compare Fig. \ref{fig:A4}.
\end{remark}
%-------------------------------------------------------------------------------------------------------
\begin{remark}\label{rmk:WeakerUpperBound}
	If \textbf{(A4)} is not satisfied, then there is always the following weaker upper bound
	in Theorem \ref{thm:MainSCoPES} on the limes superior:
	$
		\Prb\big[ T_{\cA,\cB}(G) \leq q \big]
	$.
	This upper bound can be easily derived from the SCoRE set Metatheorem, compare Remark
	\ref{eq:WeakendUpperBoundSCoREmeta} in Appendix \ref{scn:MetaTheorem}.
\end{remark}
%-------------------------------------------------------------------------------------------------------

%-------------------------------------------------------------------------------------------------------
\begin{figure}[h]
			\includegraphics[width=0.49\textwidth]{\figurepath 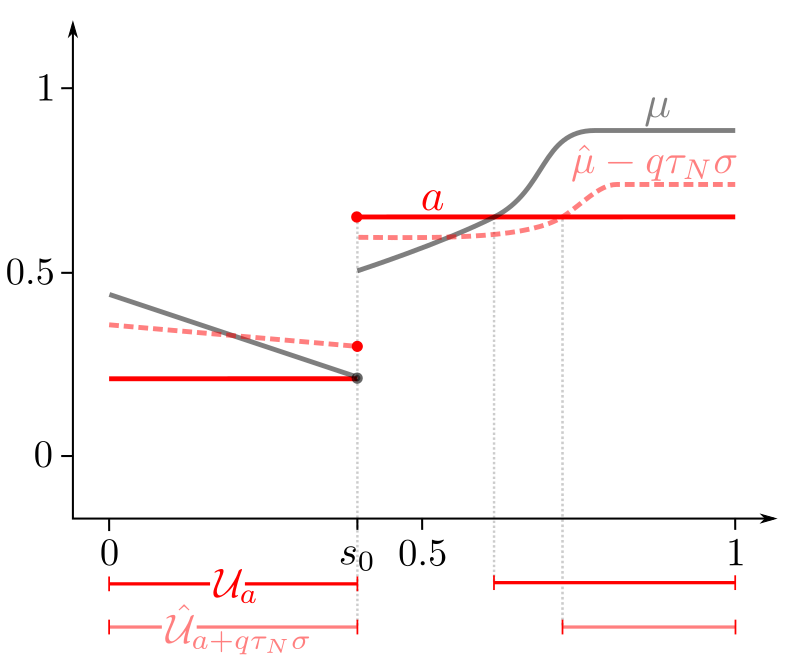}
			\includegraphics[width=0.49\textwidth]{\figurepath 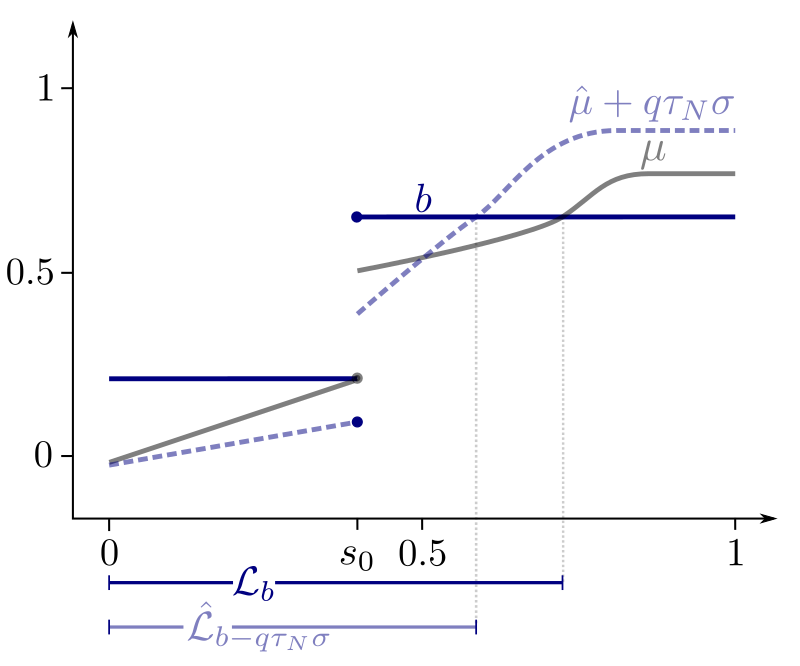}
		\caption{Illustration of why Assumption \textbf{(A4)} cannot be weakened to
				 include points
				 $ s_0 \in \mathfrak{u}^{-1}_{c} \setminus {\rm cl}\,\mu^{-1}_{c}$ for $c\in\{a,b\}$
				 even if $G_N$ is continuous in a neighbourhood of
				 $\mathfrak{u}^{-1}_{c} \setminus {\rm cl}\,\mu^{-1}_{c}$
				 in both panels.
				 The reason is that the SCoRE sets inclusions in both panels are satisfied although
				 $\hat\mu_N(s_0) - q\tau_N\sigma(s_0)>\mu(s_0)$ in the right panel and
				 $\hat\mu_N(s_0) + q\tau_N\sigma(s_0)<\mu(s_0)$ in the left panel.
		\label{fig:A4}		
		}
\end{figure}

%-------------------------------------------------------------------------------------------------------
Setting $\cA = \cB = \mathcal{F}(S)$ in Theorem \ref{thm:MainSCoPES} implies
$
	S = \mathfrak{u}^{\pm}_\cA
	  = \mu^{-1}_\cA = {\rm cl}\,\mu^{-1}_\cA
	  = {\rm int}\,\mu^{-1}_\cA
$
and therefore \textbf{(A2)}-\textbf{(A4)} are satisfied even if $S$ is non-compact by Remark \ref{rmk:A3}.
Therefore we obtain the following corollary.

%-------------------------------------------------------------------------------------------------------
\begin{corollary}\label{cor:SCB_CoPE}
	Let $q\in\mathbb{R}$ and assume {\rm \textbf{(A1)}} with
	$ \cA = \cB = \mathcal{F}(S) $. Then
	\begin{equation*}
		\lim_{N \rightarrow \infty}
			\mathbb{P}_* \left[\, \forall c \in \mathcal{F}(S):~ 
			\hat{\cL}_{c - \tau_N q\sigma} \subseteq \cL_c
			~\wedge~
			\hat{\mathcal{\cU}}_{c + \tau_N q\sigma} \subseteq \cU_c
			 \,\right]
		 	= \mathbb{P}\Big[\,  \max_{s \in S} \vert G(s) \vert \prec q \,\Big]\,.
	\end{equation*}
\end{corollary}
%-------------------------------------------------------------------------------------------------------
This result is connected to the standard construction of an $(1-\alpha)$-SCB for an estimator satisfying
a ULT. In this case the intervals $\big[ l_N(s), u_N(s) \big] = \big[\hat\mu_N(s) - q\tau_N\sigma(s), \hat\mu_N(s) + q\tau_N\sigma(s) \big]$,
$s\in S$, form an asymptotic $(1-\alpha)$-SCB for $\mu$, if $q_\alpha$ is the smallest $q>0$ such that
\begin{equation*}
	\mathbb{P}\Big[\,  \max_{s \in S} \vert G(s) \vert \leq q \,\Big] \geq 1-\alpha\,.
\end{equation*}
The estimated level sets in the above corollary are now exactly the upper excursion sets of $l_N$ and the
lower excursion sets of $u_N$ over $c$, i.e.,
\begin{equation*}
	\hat{\cU}_{c + \tau_N q\sigma} = \big\{ s\in S ~\vert~ l_N(s) > c(s) \big\}\,,~~~
		\hat{\cL}_{c - \tau_N q\sigma} = \big\{ s\in S ~\vert~ u_N(s) < c(s) \big\}\,.
\end{equation*}

%-------------------------------------------------------------------------------------------------------
Our next result will be used in Section \ref{scn:testing} to connect Theorem \ref{thm:MainSCoPES} to equivalence and relevance tests.
Readers familiar with the works \citep{Dette:2020functional, Dette:2021bio, Dette:2022cov} might spot
the structural similarities between the limiting distribution of Corollary \ref{cor:Extraction}
and their limiting distribution to compute the critical threshold of their tests. This is not a coincidence
as we will explain in Section \ref{scn:testing}.
%-------------------------------------------------------------------------------------------------------
\begin{corollary}\label{cor:Extraction}
	Assume $a, b\in \mathcal{F}(S)$, $\cA = \{a\}$, $\cB = \{b\}$, $q \in \mathbb{R}$ and
	 {\rm\textbf{(A1)}}. Then
	\begin{equation*}
	\begin{split}
		\lim_{N\rightarrow\infty} \mathbb{P}_*\Big[\, 
		\hat{\cU}_{a + q\tau_N\sigma} \subseteq \cU_{a}
		 ~ ~ \wedge ~ ~
		 \hat{\cL}_{b - q\tau_N\sigma} \subseteq \cL_{b}
		  \,\Big]
			 = \mathbb{P}\Big[\, \mathfrak{T}_{a, b}( G ) \prec q \,\Big]	\,.
	\end{split}
	\end{equation*}
	under {\rm\textbf{(A2)}},{\rm\textbf{(A3)}} / {\rm\textbf{(A4)}}.
\end{corollary}
%-------------------------------------------------------------------------------------------------------
For $a = b = c \in \cF(S)$, the above corollary generalizes Theorem $1$ from \citet{Sommerfeld:2018CoPE}
as it neither requires their Assumption 2.1(a) nor continuity assumptions on $\mu$ and $c$.
To showcase this and highlight important conceptual differences to their result, we include the next corollary which
deals with the continuous case.
%-------------------------------------------------------------------------------------------------------
\begin{corollary}\label{cor:MaxImproved}
	Let $S$ be compact, $\mu, c \in C(S)$, $\cA = \cB = \{c\}$ and $q\in \mathbb{R}$.
	Assume {\rm\textbf{(A1)}}-{\rm \textbf{(A2)}}. Then
	\begin{equation*}
		\lim_{N \rightarrow \infty}
			\mathbb{P}_* \left[\, 
			\hat{\mathcal{\cU}}_{c + q\tau_N\sigma} \subseteq \cU_c
			~\wedge~
			\hat{\cL}_{c - q\tau_N\sigma} \subseteq \cL_c
			 \,\right]
		 	= \mathbb{P}\big[\,  T_{c}( G ) \prec q \,\big]\,.
	\end{equation*}
\end{corollary}
%-------------------------------------------------------------------------------------------------------
The reason why the above result does not need the non-tangentiality assumption
(\citet[Assumption 2.1(a)]{Sommerfeld:2018CoPE}) is that the inclusion statement
on the r.h.s. slightly differs from the inclusion statement in \citet{Sommerfeld:2018CoPE}.
Expressing their inclusion statement $\hat{\cA}^+_c \subseteq \cA_c \subseteq \hat{\cA}^-_c$ in our notation yields
\begin{equation*}
\begin{split}
			S \setminus \hat{\cL}_{c + q\tau_N\sigma} \subseteq S \setminus \cL_{c}
			~ ~ ~&\wedge~ ~ ~
			\hat{\cL}_{c - q\tau_N\sigma} \subseteq \cL_c\,.
\end{split}
\end{equation*}
Thus, compared to \citet{Sommerfeld:2018CoPE} we replaced the statement
$S \setminus \hat{\cL}_{c + q\tau_N\sigma} \subseteq S \setminus \cL_{c}$
by $\hat{\mathcal{\cU}}_{c + q\tau_N\sigma} \subseteq \cU_c$.
An illustration why this is necessary to obtain the sharp upper bound without additional conditions
can be found in Figure \ref{fig:GlosedGammaC} from Appendix \ref{scn:ChoiceInclusion}.
%-------------------------------------------------------------------------------------------------------

Another strong argument for our change in the definition of the inclusion
statement compared to \citet{Sommerfeld:2018CoPE} is that the random sets
$$
	\hat{\cU}_{c + q\tau_N\sigma} \subseteq \cU_{c}\,,~~~~
	S \setminus \Big( \hat{\cU}_{c + q\tau_N\sigma} \cup \hat{\cL}_{c - q\tau_N\sigma} \Big)\,,~~~~
	\hat{\cL}_{c - q\tau_N\sigma}
$$
from Corollary \ref{cor:MaxImproved} are a partition of $S$, which satisfies with high probability
	$$
		\hat{\cU}_{c + q\tau_N\sigma} \subseteq \cU_{c}\,,\quad
			\mu^{-1}_c
			\subseteq S \setminus \big( \hat{\cU}_{c + q\tau_N\sigma}  \cup \hat{\cL}_{c - q\tau_N\sigma} \big)
			 \,,\quad
		 \hat{\cL}_{c - q\tau_N\sigma} \subseteq \cL_{c}.
	$$
This means that $S$ is divided into three regions. For the sets $ \hat{\cU}_{c + q\tau_N\sigma}$ and
$\hat{\cL}_{c - q\tau_N\sigma}$ we know that (asymptotically) with a probability given by $q$
the function $\mu$ on these regions is larger and smaller than $c$, respectively.
We can make no statement about any $s$ in the third region, yet we can guarantee with a probability depending on $q$
that this region is a superset of the level set $\mu^{-1}_c$.
Noteworthy, the partition $S \setminus \hat{\cL}_{c + q\tau_N\sigma}$,
$ \hat{\cL}_{c + q\tau_N\sigma} \cap S \setminus\hat{\cL}_{c - q\tau_N\sigma} $
and $\hat{\cL}_{c - q\tau_N\sigma}$ of $S$ resulting from \cite[Theorem 1]{Sommerfeld:2018CoPE}
does not have such nice probabilistic properties because
the preimage $\mu^{-1}_c$ can intersect 
	$S \setminus \hat{\cL}_{c + q\tau_N\sigma}$ and
	$ \hat{\cL}_{c + q\tau_N\sigma} \cap S \setminus\hat{\cL}_{c - q\tau_N\sigma}  $.

%-------------------------------------------------------------------------------------------------------

%-------------------------------------------------------------------------------------------------------
\subsection{Estimation of the Generalized Preimage and Bootstrapping the Quantile of SCoRE Sets}\label{scn:estSC}
%-------------------------------------------------------------------------------------------------------
Estimation of $q$ such that the families derived in Theorem \ref{thm:MainSCoPES} are
$(1-\alpha)$-SCoRE sets requires estimation of $\mathfrak{u}^{\pm}_\cC$ for
$\cC \subseteq\mathcal{F}(S)$.
In the case of $S$ being compact,
a Hausdorff-distance consistent estimator can be derived by replacing $\mu$ by
$\hat\mu_N$ in the definition of ${\rm cl}\,\mu^{-1}_{\cC_{\pm\eta}}$ and choosing $\eta$
depending on $N$ appropriately.
More precisely, if $(k_N)_{N\in\mathbb{N}}$ is a positive sequence converging to zero,
we define the following estimator of $\mathfrak{u}^{\pm}_\cC$ by
\begin{equation}\label{eq:estSC}
	\hat{\mathfrak{u}}_\cC^{\pm} = {\rm cl}\big\{\, s\in S~ \big\vert~ \exists c \in \cC:~
		0 \leq \mp \big( \hat\mu_N(s) - c(s) \big) \leq k_N\tau_N\sigma(s) \,\big\}\,.
\end{equation}
This idea is similar to the estimation of the extremal sets in \citet{Dette:2020functional} and
related work. The consistency in Hausdorff-distance of this estimator is stated in the next result.
%-------------------------------------------------------------------------------------------------------
\begin{theorem}\label{thm:SCest}
	Let $S$ be compact, $\cC \subseteq\mathcal{F}(S)$ and $(k_N)_{N\in\mathbb{N}}$ be a positive sequence such that $ \lim_{N\rightarrow \infty} k_N\tau_N = 0$
	and $\lim_{N\rightarrow \infty} k_N / \log\big( \tau_N^{-2} \big) = \kappa$ for some $\kappa>0$, then
	$
	d_H\big( \hat{\mathfrak{u}}_\cC^{\pm}, \mathfrak{u}^{\pm}_\cC \big)
			\rightarrow 0
	$
	 in outer probability as $N\rightarrow\infty$.
\end{theorem}
%-------------------------------------------------------------------------------------------------------
In principle, this consistency result allows us to estimate the quantile of SCoRE sets from Theorem
\ref{thm:MainSCoPES} along the lines described in \citet{Dette:2020functional}.
A general strategy\footnote{not necessarily the best in a given probabilistic model} to achieve this
is to show within the assumed probabilistic model that realizations of a bootstrap processes
$B_N^{(1)},\ldots, B^{(R)}_N$ can be obtained which satisfy
\begin{equation*}
\begin{split}
	\Big( G_N, B_N^{(1)},\ldots, B_N^{(R)} \Big) \rightsquigarrow \Big( G, G^{(1)},\ldots, G^{(R)} \Big)
\end{split}
\end{equation*}
weakly in $\big( \ell^\infty({\rm cl}\,\mu^{-1}_{\cA_{-\tilde{\eta}}} \cup {\rm cl}\,\mu^{-1}_{\cB_{+\tilde{\eta}}}) \big)^{R+1}$.
Here $G^{(1)},\ldots, G^{(R)}$ are i.i.d. copies of $G$. Combining this with Theorem
\ref{thm:SCest} and a simple generalization of
Lemma B.3 from \citet{Dette:2020functional}, which can be derived using our Appendix \ref{App:WeakMaxConvergence},
yields for any sequence $(\eta_N)_{N\in\mathbb{N}}$ converging to zero,
\begin{equation*}
\begin{split}
	&\Big( T_{{\rm cl}\cA_{-\eta_N}, {\rm cl}\cB_{-\eta_N}}(G_N),
		  T_{{\rm cl}\cA_{-\eta_N}, {\rm cl}\cB_{-\eta_N}}\big(B_N^{(1)}\big),\ldots,
		  T_{{\rm cl}\cA_{-\eta_N}, {\rm cl}\cB_{-\eta_N}}\big(B_N^{(R)}\big) \Big)\\
		   &\rightsquigarrow
		   \Big( \mathfrak{T}_{\cA, \cB}(G),
		   		 \mathfrak{T}_{\cA, \cB}\big(G^{(1)}\big),\ldots,
		         \mathfrak{T}_{\cA, \cB}\big(G^{(R)}\big) \Big)
\end{split}
\end{equation*}
weakly in $\mathbb{R}^{R+1}$ with
$\mathfrak{T}_{\cA, \cB}\big(G^{(1)}\big),\ldots, \mathfrak{T}_{\cA, \cB}\big(G^{(R)}\big) \sim \mathfrak{T}_{\cA, \cB}(G)$ i.i.d..

\section{Hypothesis Testing and SCoRE Sets}\label{scn:testing}
%-------------------------------------------------------------------------------------------------------
%-------------------------------------------------------------------------------------------------------
%-------------------------------------------------------------------------------------------------------
%-------------------------------------------------------------------------------------------------------

In this section we connect SCoRE sets to statistical hypothesis testing and use
Theorem \ref{thm:MainSCoPES} to develop
novel statistical hypothesis tests. Our developments can be seen as an extension of
the duality of confidence intervals and hypothesis tests in point hypothesis testing as given in
\cite[Thm 3.5.1]{Lehmann:2005testing} to multiple hypothesis testing.
The following definition will reduce and unify the notations in the upcoming sections.
\begin{definition}\label{def:ShorteningTest}
	Let $\Theta$ be a set, $\mathbf{H}_{0,\theta}$ be a null hypothesis and $\mathbf{H}_{1,\theta}$
	an corresponding alternative hypothesis for $\theta\in \Theta$. Assume there is a statistical test
	for $\mathbf{H}_{0,\theta}$ at each $\theta\in \Theta$ which decides whether $\mathbf{H}_{0,\theta}$ is rejected.
	The set of true null hypotheses and its estimate are
	\begin{equation*}
	\begin{split}
		\cH_0       &= \big\{ \theta\in \Theta ~\vert~ \mathbf{H}_{0,\theta} \text{ is true } \big\}\,,~ ~ ~
		\hat{\cH}_0 = \big\{ \theta\in \Theta ~\vert~\mathbf{H}_{0,\theta} \text{ is not rejected } \big\}\,.\\
	\end{split}
	\end{equation*}
	Similarly, the set were the alternative hypothesis is true and its estimate are
	\begin{equation*}
	\begin{split}
		\cH_1       &= \big\{ \theta\in \Theta ~\vert~ \mathbf{H}_{1,\theta} \text{ is true } \big\}\,,~ ~ ~
 		\hat{\cH}_1 = \big\{ \theta\in \Theta ~\vert~ \mathbf{H}_{0,\theta} \text{ is rejected } \big\}\,.
	\end{split}
	\end{equation*}
\end{definition}
\begin{definition}[Strong Family-wise Error Rate Control]\label{def:FWER}
	Let $\{ \cT_\theta \}_{\theta\in\Theta}$ be a family of tests such that each $\cT_\theta$ is
	a test for a null hypothesis $\mathbf{H}_{0,\theta}$,
	where $\theta\in\Theta$.
	We say that this family controls the \textit{family-wise error rate} (FWER)
	in the strong sense at level $\alpha$, if
	\begin{equation*}
		{\rm FWER}(\cH_0) = \Prb^*\big[\, \exists s\in \cH_0:~ \cT_\theta \text{ is rejected} \,\big]
		\leq \alpha
	\end{equation*}
	for all possible sets $\cH_0 \subseteq \Theta$.\footnote{We supress that $\Prb^*$ might depend on $\cH_0$.}
\end{definition}
\begin{remark}\label{rmk:FWER}
	It follows immediately from Definition \eqref{def:ShorteningTest} that
	$$ {\rm FWER}(\cH_0) = 1 - \Prb_*\big[\, \cH_0 \subseteq \hat{\cH}_0 \,\big]\,.$$
\end{remark}

%-------------------------------------------------------------------------------------------------------
\subsection{Multiple Hypothesis Tests Derived from SCoRE Sets}\label{sec:duality}
%-------------------------------------------------------------------------------------------------------
This section is complementary to the tests derived from confidence sets in
\cite{Aitchison:1964confidence}. While (simultaneous) confidence sets for a parameter $\mu$
allow to derive joint tests for essentially all possible hypotheses, SCoRE sets
restrict the set of jointly testable hypotheses, which can lead to an increase
in power compared to confidence set based tests.
In this section we assume a more general framework than in the rest of this article
to highlight the close connection between confidence regions for excursion sets and
tests that control the FWER in the strong sense.

Let $ \mu \in \ell^\infty(S)$ be a target function, $\cA,\cB\subseteq\cF(S)$
sets of known excursion functions and recall that $\cU_a$ and $\cL_b$ are the upper/lower
excursion sets of $\mu$ above/below $a$/$b$ respectively.
\begin{definition}(SCoRE sets)
	Let $\alpha\in(0,1)$, $\cP(S)$ denote the power set of $S$ and
	$\cU_a:\,\Omega \rightarrow \cP(S)$, $\omega \mapsto \hat \cU_a(\omega)$
	and $\cL_a:\,\Omega \rightarrow \cP(S)$, $\omega \mapsto \hat\cL_b (\omega)$ be set-valued function.
	If the families $\{\hat\cU_a \}_{a\in\cA}$ and $\{\hat\cL_b \}_{b\in\cB}$ satisfy
	\begin{equation*}
		\Prb_*\Big[ \forall a\in\cA\, \forall b\in \cB:~  \hat \cU_a \subseteq \cU_a
	\wedge  \hat \cL_b \subseteq \cL_b \Big] \geq 1-\alpha\,,
	\end{equation*}
	for all $\mu \in \ell^\infty(S)$ we say that  $\{\cU_a \}_{a\in\cA}$ and $\{\cL_b \}_{b\in\cB}$
	form $(1-\alpha)$-Simultanuous COnfidence	Region of Excursion (SCoRE) sets
	for $(\cA,\cB)$.
\end{definition}
The hypotheses, which we want to test jointly for $a\in \cA$ and $b\in \cB$, are
\begin{equation}\label{eq:generalHypotheses}
\begin{split}
	\mathbf{H}_{0,s}^{\leq a} :&~~ \mu(s) \leq a(s) \quad \text{ vs } \quad \mathbf{H}_{1,s}^{\leq a}: \mu(s) > a(s)\\
	\mathbf{H}_{0,s}^{\geq b}:&~~ \mu(s) \geq b(s) \quad \text{ vs } \quad \mathbf{H}_{1,s}^{\geq b}: \mu(s) < b(s)\\
\end{split}
\end{equation}
Using definition \ref{def:ShorteningTest} we define the joint FWER for a family of tests
on the hypotheses \eqref{eq:generalHypotheses} by
\begin{equation}\label{eq:FWERscore}
\begin{split}
	{\rm FWER}\Big( \cH_{0}^{\leq \cA},\cH_{0}^{\geq \cB} \Big)
	= 1 - 	\Prb_*\Big[ \forall a\in\cA\, \forall b\in \cB:~  \hat{\cU}_{a} \subseteq \cU_a
	\wedge  \hat{\cL}_{b} \subseteq \cL_b \Big]
\end{split}
\end{equation}
which by Remark \ref{rmk:FWER} is consistent with Definition \ref{def:FWER} and therefore
the following immediate consequence is not surprising.
\begin{proposition}(SCoRE sets and Multiple Hypothesis Test Duality)\label{prop:duality}
	Assume that  $\{\hat\cU_a \}_{a\in\cA}$ and $\{\hat\cL_b \}_{b\in\cB}$ form $(1-\alpha)$-SCoRE
	sets for $(\cA,\cB)$, then the family of tests given by
	\begin{equation*}
	\begin{split}
		\cT_{s}^{\leq a} = \begin{cases}
								\text{reject}\,, & \text{ if }
									s \in \hat{\cU}_{a}\\
								\text{accept}\,, & \text{ else }
							\end{cases}\,,	\quad					
		\cT_{s}^{\geq b} = \begin{cases}
								\text{reject}\,, & \text{ if }
									s \in \hat{\cL}_{b}\\
								\text{accept}\,, & \text{ else }
							\end{cases}\,,
	\end{split}
	\end{equation*}	
	for $a\in\cA$ and $b\in\cB$ control the joint FWER at level $\alpha$ for all $\mu\in\ell^\infty(S)$.
	
	Conversely, assume we have tests $\cT_{s}^{\leq a}$ and $\cT_{s}^{\geq b}$
	which  control the joint FWER for all $a\in\cA$ and $b\in\cB$ in the strong sense,
	then $\{\hat{\cH}_{1}^{\leq a} \}_{a\in\cA}$ and $\{ \hat{\cH}_{1}^{\geq a} \}_{b\in\cB}$
	form $(1-\alpha)$-SCoRE sets for $(\cA,\cB)$.
\end{proposition}
\begin{remark}
	The same duality holds between asymptotic tests and asymptotic SCoRE sets and for
	the inclusion statement from \cite{Sommerfeld:2018CoPE}. 
	Additionally, it establishes that, for example, stagewise multiple testing procedures such as the
	Bonferroni-Holm procedure \cite{Holm:1979}, can be used to
	construct $(1-\alpha)$-SCoRE sets for $(\cA,\cB)$.
\end{remark}

An immediate application of this duality are simultaneous confidence bands because a slight variation
on the main theorem of \cite{Ren:2022} shows that they define simultaneous SCoRE
sets for $\big(\cF(S),\cF(S) \big)$; here the change of the inclusion statement greatly simplifies the proof.
\begin{proposition}\label{prop:CoPEasSCB}
	Assume that $ \mu \in \ell^\infty(S)$ and that the family indexed by $s\in S$
	of random intervals $\big[ \hat{l}(s), \hat{u}(s) \big]$
	forms an $(1-\alpha)$-SCB for $\mu$. Then
\begin{equation}\label{eq:SCBCoPE}
	\Prb_*\left[\, \forall c  \in \mathcal{F}( S ):~
		\hat{\cU}_{c} \subseteq \cU_c
			 ~\wedge~ 
	 	\hat{\cL}_{c} \subseteq \cL_c
		\,\right]
 \geq 1-\alpha
\end{equation}
\end{proposition}
\begin{remark}
	As in \cite{Ren:2022} the proof shows the equivalence of
	\begin{equation*}
	\begin{aligned}
		&\textbf{(E1)}~\forall s \in  S :~  \hat{l}(s)  \leq \mu(s) \leq \hat{u}(s)\\
		&\textbf{(E2)}~\forall c \in \mathcal{F}( S ):~
			\hat{\cL}_{c} \subseteq \cL_c
			~\wedge~
			\hat{\cU}_{c} \subseteq \cU_c\,.
	\end{aligned}
	\end{equation*}	
	Therefore an upper bound of the form
	\begin{equation*}
		\Prb^*\left[\, \forall c  \in \mathcal{F}( S ):~
		\hat{\cU}_{c} \subseteq \cU_c
			 ~\wedge~ 
	 	\hat{\cL}_{c} \subseteq \cL_c
		\,\right]
 		\leq 1-\beta
	\end{equation*}
	for $\beta \leq \alpha$ yields an $(1-\beta)$ upper bound for \eqref{eq:SCBCoPE} as well.
\end{remark}
A SCB even yields joint tests for less standard hypotheses such
as the local relevance hypothesis or the equivalence hypothesis, i.e.,
\begin{equation*}
\begin{split}
	\mathbf{H}_{0,rel}^{s,\Delta}:~ \mu(s) \in [-\Delta, \Delta] ~~\,&vs.~~\mathbf{H}_{1,rel}^{s,\Delta}:~ \mu(s) \notin [-\Delta, \Delta]\\
	\mathbf{H}_{0,eq}^{\Delta}:~ \exists s \in S: \mu(s) \notin (-\Delta, \Delta) ~~\,&vs.~~\mathbf{H}_{1,eq}^{\Delta}:~\forall s \in S: \mu(s) \in (-\Delta, \Delta)\\
\end{split}
\end{equation*}
for some $\Delta>0$, by defining the decision rules as
\begin{equation*}
\begin{split}
	~\text{\textbf{reject}}~~ H_{0,rel}^{s,\Delta}
	 		~ &:\Longleftrightarrow ~ s \in \hat\cU_{\Delta} \cup \hat\cL_{-\Delta}
	 ~\Longleftrightarrow~ \hat{u}(s) < -\Delta~\vee~\hat{l}(s) > \Delta\\
	\text{\textbf{reject}}~~ H_{0.eq}^{\Delta}
	 		~ &:\Longleftrightarrow ~  \hat\cU_{-\Delta} \cap \hat\cL_{\Delta}  = S
	 ~\Longleftrightarrow~ \forall s\in S: \big[ \hat{l}(s), \hat{u}(s) \big] \in (-\Delta, \Delta) 
\end{split}
\end{equation*}
How these observations relate to recently established relevance and equivalence tests
 for $C(S)$-valued data based on the supremums norm and why SCoRE sets  despite the duality
 given in Proposition \ref{prop:duality} are a useful concept, are the topics of the next sections.

%-------------------------------------------------------------------------------------------------------
\subsection{Oracle Limit Distributions of Multiple Hypothesis Tests}\label{sec:oracle}
%-------------------------------------------------------------------------------------------------------
From the viewpoint of the previous section, we will now present a possible interpretation
of Theorem \ref{thm:MainSCoPES} as providing oracle limit distributions of standard multiple
hypothesis tests derived from a ULT and thereby answering the question why the concept of
SCoRE sets is useful.

Here we assume the setting and notations of Theorem \ref{thm:MainSCoPES}. In particular,
we assume that an estimator $\hat\mu_N$ of a target function $\mu$ satisfying \textbf{(A1)}-\textbf{(A4)}
is given, as well as, two sets $\cA,\cB\subseteq\cF(S)$. 
Thus, as $\hat\mu_N$ satisfies a ULT the canonical tests for the hypotheses \eqref{eq:generalHypotheses} are
\begin{equation}\label{eq:ScoRESTestallg}
\begin{split}
	&\cT^{s, a}_{\leq,q} = \begin{cases}
							\text{reject}\,, & \text{ if } \hat\mu_N(s) \geq a(s) + q\tau_N\sigma(s)\,,\text{ i.e., }
							s \in \hat{\cU}_{a+ q\tau_N\sigma}\\
							\text{accept}\,, & \text{ else }
						\end{cases}	\\					
	&\cT^{s, b}_{\geq,q} = \begin{cases}
							\text{reject}\,, & \text{ if } \hat\mu_N(s) \leq a(s) - q\tau_N\sigma(s)\,,\text{ i.e., }
							s \in \hat{\cL}_{b - q\tau_N\sigma}\\
							\text{accept}\,, & \text{ else }
						\end{cases}\,,
\end{split}
\end{equation}
where $q\geq 0$ is a parameter, which can be chosen to control the level of the test. By Definition \ref{def:ShorteningTest} it holds that
$ \cH_0^{\leq a} = S \setminus \cU_a$, 
$\hat{\cH}_0^{\leq a} = S \setminus \hat{\cU}_{a+ q\tau_N\sigma}$,
$\cH_0^{\geq b} = S \setminus \cL_b$ and
$ \hat{\cH}_0^{\geq b} = S \setminus \hat{\cL}_{b- q\tau_N\sigma}
$.
Thus, applying the limes superior and limes inferior to \eqref{eq:FWERscore} we obtain from
Theorem \ref{thm:MainSCoPES} the next corollary. We call this result the \textit{oracle limit distribution}
of the joint tests given in \eqref{eq:ScoRESTestallg} over all $a\in\cA$ and $b\in\cB$ because it provides
explicitly the sharp limit distribution of the joint FWER for each possible $\mu\in\ell^\infty(S)$.
\begin{corollary}[Oracle Limit Distribution of FWER]\label{cor:oracleLimit}
	Assume the notation and assumptions of Theorem \ref{thm:MainSCoPES}. Then the
	joint FWER of the tests given by \eqref{eq:ScoRESTestallg} on the hypotheses
	\eqref{eq:generalHypotheses} for $a\in\cA$ and $b\in\cB$ satisfies
	\begin{equation*}
	\begin{split}
		\lim_{N\rightarrow\infty}{\rm FWER}\Big(\cH_0^{\leq\cA},\cH_0^{\geq\cB}\Big)
		= 1 - \mathbb{P}\big[\, \mathfrak{T}_{\cA, \cB}( G ) \prec q \,\big]
	\end{split}
	\end{equation*}
	for all $\mu \in \ell^\infty(S)$. In particular, given $\alpha\in(0,1)$ and $q_\alpha$ such that 
	\begin{equation*}
	\begin{split}
		 \mathbb{P}_*\big[\, \mathfrak{T}_{\cA, \cB}( G ) < q_\alpha \,\big] \geq 1 - \alpha\,,
	\end{split}
	\end{equation*}
	it holds that the tests given by \eqref{eq:ScoRESTestallg} using $q=q_\alpha$ control the joint FWER
	over all hypotheses $\mathbf{H}_{0,s}^{\leq a}$ and $\mathbf{H}_{0,s}^{\geq b}$ for all $s\in S$, $a\in\cA$
	and $b\in\cB$ at level $\alpha$.	
\end{corollary}
\begin{remark}
	It is noteworthy that changing the  hypotheses in \eqref{eq:generalHypotheses} to
	\begin{equation*}
	\begin{split}
		H_{0,s}^{< a}:&~~ \mu(s) < a(s) \quad \text{ vs } \quad H_{1,s}^{< a}: \mu(s) \geq a(s)\\
		H_{0,s}^{> b}:&~~ \mu(s) > b(s) \quad \text{ vs } \quad H_{1,s}^{> b}: \mu(s) \leq b(s)\\
	\end{split}
	\end{equation*}
	 does not have the oracle limit distribution from Corollary \ref{cor:oracleLimit}.
	 The reason is that the natural test obtained from the ULT is
	\begin{equation*}
	\begin{split}
		&\cT_{s,q}^{\leq a} = \begin{cases}
								\text{reject}\,, & \text{ if }
									s \in S \setminus \hat{\cL}_{a+ q\tau_N\sigma}\\
								\text{accept}\,, & \text{ else }
							\end{cases}	\\					
		&\cT_{s,q}^{\geq b} = \begin{cases}
								\text{reject}\,, & \text{ if }
									s \in S \setminus \hat{\cU}_{b - q\tau_N\sigma}\\
								\text{accept}\,, & \text{ else }
							\end{cases}\,.
	\end{split}
	\end{equation*}	 
	 Thus, while the asymptotic upper bound of the FWER remains the same, the lower bound
	 is only identical under further assumptions, for example, on $\cA$ and $\cB$, or under
	 restrictions of the set of possible target functions $\mu$ as provided by
	 \cite[Assumption (A.ii)]{Mammen:2013} and \cite[Assumption 2.1(a)]{Sommerfeld:2018CoPE}, compare Section
	 \ref{scn:ChoiceInclusion} in particular Figure \ref{fig:Cope_Def_Problem}. 
	Thus, from our asymptotic oracle distribution view on multiple testing it is preferable that the null
	 hypothesis at each $s$ is a closed set because otherwise the FWER can only be controlled conservatively
	 for some $\mu$. 
\end{remark}
\begin{remark}
	Corollary \ref{cor:oracleLimit} means that Theorem \ref{thm:MainSCoPES} yields
	joint tests for the hypotheses \eqref{eq:generalHypotheses} for any two sets of
	functions $\cA$ and $\cB$.
	An inconvenience is that $\mathfrak{T}_{\cA, \cB}( G )=-\infty$ for all $\mu$ such that
	$\mathfrak{u}^-_\cA \cup \mathfrak{u}^+_\cB = \emptyset$.
	Thus, any $q\in\mathbb{R}$ yields in this case a test that asymptotically controls the FWER
	in the strong sense at any $\alpha>0$.
	In the next sections we propose solutions based on ideas from
	\cite{Dette:2020functional} that make the quantile $q$ unique.
\end{remark}

%-------------------------------------------------------------------------------------------------------
\subsection{Global Relevance Tests}\label{sec:globalRel}
%-------------------------------------------------------------------------------------------------------
In the literature on statistical inference in the Banach space $C(S)$ the global relevance hypothesis, i.e.,
\begin{equation*}
\begin{split}
	\mathbf{H}_{0}\mathbf{:}~ ~ \forall s\in S:~\mu(s) \in \big[-\Delta, \Delta \big]
	~ ~ ~\text{ vs.}~ ~ ~
	\mathbf{H}_{1}\mathbf{:}~ ~\exists s\in S:~ \mu(s) \notin \big[-\Delta, \Delta \big]
\end{split}
\end{equation*}
for $\Delta\geq 0$, has recently gained attention in \cite{Dette:2020functional,Dette:2022cov,Dette2021FunctionOnFunction}.
In all these articles the strategy to obtain
a test at level $\alpha$ is based on the supremums norm $\Vert f \Vert_\infty = \sup_{s\in S} \vert f(s)\vert$ for
$s\in S$. The next box summarizes this approach conceptually.
\begin{mdframed}
\begin{enumerate}
	\item[I.] Define an estimator $\hat\mu_N$ of $\mu$ and prove that it satisfies a ULT in $C(S)$, i.e.,
		  there exists a random process $G$ on $S$ such that
		  $\big( \hat\mu_N - \mu \big) / \tau_N \rightsquigarrow G$ weakly in $C(S)$.
	\item[II.] Use the CLT from I. to prove
			   $\big( \Vert \hat\mu_N \Vert_\infty - \Vert \mu \Vert_\infty \big) / \tau_N \rightsquigarrow T\big( \cE \big)$
		  for
			\begin{equation}\label{eq:DetteLimit}
				T\big( \cE \big) = \max\Big\{ \sup_{s\in\cE^+} G(s), \sup_{s\in\cE^-} -G(s) \Big\}
 			\end{equation}
 			depending on the extremal sets $\cE^\pm = \big\{ s\in S ~\vert~ \mu(s) = \pm\Vert \mu \Vert_\infty \big\}$.
 	\item[III.] Obtain the $(1-\alpha)$ quantile $q_{1-\alpha}$ of $T\big( \cE \big)$. 
 	\item[IV.] Reject $\mathbf{H}_{0}$ at significance level $\alpha$, if $\Vert \hat\mu_N \Vert_\infty > \Delta  + \tau_N q_{1-\alpha}$ 
\end{enumerate}
\end{mdframed}
The latter is justified by their observation that II. implies
\begin{equation}\label{eq:DetteRelTestProps}
	\lim_{N\rightarrow\infty}\Prb\Big[ \Vert \hat\mu_N \Vert_\infty > \Delta  + \tau_Nq_{1-\alpha} \Big]
	= \begin{cases}
			0\,,      & \Vert \mu \Vert_\infty < \Delta \\
			\alpha\,, & \Vert \mu \Vert_\infty = \Delta \\
			1\,, & \Vert \mu \Vert_\infty > \Delta
	   \end{cases}
\end{equation}

As the structure of the limit distribution \eqref{eq:DetteLimit} is suspiciously similar to
\begin{equation}\label{eq:Limit2functions}
	\mathfrak{T}_{a, b} = \max\Bigg\{ \sup_{s \in \mathfrak{u}^{-}_a} G(s),
				~ \sup_{s \in \mathfrak{u}^{+}_b} -G(s) \Bigg\}\,,
\end{equation}
the limit distribution of Corollary \ref{cor:Extraction}, it is natural to ask whether and how they are related.
To clarify this, we first note that all considered functions or sample paths of processes are in $C(S)$ and $S$ is compact.
Hence \textbf{(A3)} and  \textbf{(A4)} are satisfied. Moreover,
we have that Step I. implies \textbf{(A1)} and  \textbf{(A2)} for the choice $\sigma(s) = 1$
for all $s \in S$ by \cite[1.3.10 Theorem]{Vaart:1996weak}. Finally, the extremal sets $\mathcal{E}^\pm $ satisfy
$$
 \mathcal{E}^\pm = \big\{s \in S ~\vert~ \mu(s) = \pm \Vert \mu \Vert_\infty \big\}
 				 = \mu^{-1}_{\pm\Vert \mu \Vert_\infty} = \mathfrak{u}^{\pm}_{\pm\Vert \mu \Vert_\infty}.
$$
which shows using $a(s) = \Vert \mu \Vert_\infty$ and $b(s) =  -\Vert \mu \Vert_\infty$ for all $s\in S$ in \eqref{eq:Limit2functions} that 
\begin{equation}
	T\big( \cE \big) = \mathfrak{T}_{\Vert \mu \Vert_\infty, -\Vert \mu \Vert_\infty}\,.
\end{equation}
Assuming $q_{1-\alpha}$ being given by Step III,
we obtain from Corollary \ref{cor:Extraction} that
	\begin{equation*}
	\begin{split}
		 1-\alpha
		 	&= \lim_{N\rightarrow\infty} \mathbb{P}_*\Big[\, 
		\hat{\cU}_{\Vert \mu \Vert_\infty + \tau_Nq_{1-\alpha}} \subseteq \cU_{\Vert \mu \Vert_\infty}
		 ~ ~ \wedge ~ ~
		 \hat{\cL}_{-\Vert \mu \Vert_\infty - \tau_Nq_{1-\alpha}} \subseteq \cL_{-\Vert \mu \Vert_\infty}
		  \,\Big]\\
			&=\lim_{N\rightarrow\infty} \mathbb{P}_*\Big[\, 
		\hat{\cU}_{\Vert \mu \Vert_\infty + \tau_Nq_{1-\alpha}} = \emptyset
		 ~ ~ \wedge ~ ~
		 \hat{\cL}_{-\Vert \mu \Vert_\infty - \tau_Nq_{1-\alpha}} = \emptyset
		  \,\Big]\\
			&=\lim_{N\rightarrow\infty} \mathbb{P}_*\Big[\,
				\Vert \hat\mu_N \Vert_\infty \leq \Vert \mu \Vert_\infty + \tau_Nq_{1-\alpha}
		  \,\Big]
	\end{split}
	\end{equation*}
which is equivalent to the case $\Delta = \Vert \mu \Vert_\infty$ in \eqref{eq:DetteRelTestProps}.
Similarly, Corollary \ref{cor:Extraction} yields for $\Delta > \Vert \mu \Vert_\infty$ that
	\begin{equation*}
	\begin{split}
		1 &= \mathbb{P}\Big[\, \mathfrak{T}_{\Delta, -\Delta}( G ) \leq q_{1-\alpha} \,\Big]\\
		 	&= \lim_{N\rightarrow\infty} \mathbb{P}_*\Big[\, 
		\hat{\cU}_{\Delta + \tau_Nq_{1-\alpha}} \subseteq \cU_{\Delta}
		 ~ ~ \wedge ~ ~
		 \hat{\cL}_{-\Delta - \tau_Nq_{1-\alpha}} \subseteq \cL_{-\Delta}
		  \,\Big]\\
			&=\lim_{N\rightarrow\infty} \mathbb{P}_*\Big[\,
				\Vert \hat\mu_N \Vert_\infty \leq \Delta + \tau_Nq_{1-\alpha}
		  \,\Big]
	\end{split}
	\end{equation*}
as $\mu^{-1}_{\pm\Delta} = \emptyset$. Note that this is equivalent to the case
$\Delta > \Vert \mu \Vert_\infty$ in \eqref{eq:DetteRelTestProps}.
In order to obtain the case $\Delta < \Vert \mu \Vert_\infty$, we compute
	\begin{equation*}
	\begin{split}
		&\liminf_{N\rightarrow\infty} \mathbb{P}_*\Big[\,
				\Vert \hat\mu_N \Vert_\infty > \Delta + \tau_Nq_{1-\alpha}
		  \,\Big]\\
		&\quad\quad=  \liminf_{N\rightarrow\infty} \mathbb{P}_*\Big[\, 
		\hat{\cU}_{\Delta + \tau_Nq_{1-\alpha}} \neq \emptyset
		 ~ ~ \vee ~ ~
		 \hat{\cL}_{-\Delta - \tau_Nq_{1-\alpha}} \neq \emptyset
		  \,\Big]\\
		&\quad\quad=  1 - \limsup_{N\rightarrow\infty} \mathbb{P}^*\Big[\, 
		\hat{\cU}_{\Delta + \tau_Nq_{1-\alpha}} = \cU_{\Vert \mu \Vert_\infty}
		 ~ ~ \wedge ~ ~
		 \hat{\cL}_{-\Delta - \tau_Nq_{1-\alpha}} =  \cL_{-\Vert \mu \Vert_\infty}
		  \,\Big]\\
		&\quad\quad\geq  1 - \limsup_{N\rightarrow\infty} \mathbb{P}^*\Big[\, 
		\hat{\cU}_{\Delta + \tau_Nq_{1-\alpha}} \subseteq \cU_{\Vert \mu \Vert_\infty}
		 ~ ~ \wedge ~ ~
		 \hat{\cL}_{-\Delta - \tau_Nq_{1-\alpha}} \subseteq  \cL_{-\Vert \mu \Vert_\infty}
		  \,\Big] \\
		  &\quad\quad = 1
	\end{split}
	\end{equation*}
where the last equality follows from Lemma \ref{lem:CoPEContainLemma} from Appendix \ref{Appendix:Auxillary}. Thus, we recovered also the
case $\Delta < \Vert \mu \Vert_\infty$ in \eqref{eq:DetteRelTestProps}.
%-------------------------------------------------------------------------------------------------------

After the above analysis it is not surprising that Corollary \ref{cor:Extraction} enables us to test the hypothesis
\begin{equation*}
\begin{split}
	\mathbf{H}_{0}\mathbf{:}~ ~ \forall s\in S:~\mu(s) \in \big[ b(s), a(s) \big]
	~ \text{ vs.}~ 
	\mathbf{H}_{1}\mathbf{:}~ ~\exists s\in S:~ \mu(s) \notin \big[ b(s), a(s) \big]\,.
\end{split}
\end{equation*}
for $a,b\in\mathcal{F}(S)$ under weak condition. The global relevance test for the above hypothesis based on Corollary \ref{cor:Extraction} is
\begin{equation}\label{eq:GlobalRelevanceSCoRESI}
	\text{\textbf{Reject} } \mathbf{H}_{0} \quad\Longleftrightarrow\quad \hat{\cU}_{a + q\tau_N\sigma} \neq \emptyset
							~~\vee~~         \hat{\cL}_{-b - q\tau_N\sigma} \neq \emptyset\,.
\end{equation}
In contrast to the general strategy discussed above this test can take the asymptotic variance of the
estimator $\hat\mu_N$ into account by setting $\sigma$ identical to the square root of the asymptotic variance.
Similar arguments as above then prove the following theorem which generalizes \eqref{eq:DetteRelTestProps}.
\begin{theorem}\label{thm:Dette_Generalized}
	 Let $\alpha \in (0,1)$, $a,b\in\cF(S)$, $\delta = \max\{ \delta_a, \delta_b \}$ with
	 \begin{equation}
	 	\delta_a = \sup_{ s \in S } \big\{ \mu(s) - a(s) \big\}\,,\quad\quad\quad \delta_b = \sup_{ s \in S } \big\{ b(s) - \mu(s)\big\}\,,
	 \end{equation}
	 $\cA = \big\{ a + \delta \big\}$  and  $\cB =  \big\{ b - \delta \big\}$.
	 Assume that {\rm\textbf{(A1)}} is satisfied and $q_\alpha$ fulfills
	 $
		\Prb\big[\, \mathfrak{T}_{ a + \delta,  b - \delta }(G) \geq q_\alpha \,\big] \leq \alpha
	 $.
	Then the global relevance test given in \eqref{eq:GlobalRelevanceSCoRESI} with quantile $q = q_\alpha$ satisfies
	\begin{itemize}
			\item[(a)]  {\rm \textbf{Case} $\boldsymbol\delta \boldsymbol= \mathbf{0}$\textbf{:}}
				If
				{\rm\textbf{(A2)}} and {\rm\textbf{(A3)}} are satisfied, then
				\begin{equation*}
					\begin{split}
						\limsup_{N\rightarrow\infty}&\,\,
						\mathbb{P}^*\big[\,
							 	\mathbf{H}_0 \text{ is rejected }
							  \,\big]
							  \leq \alpha\,.
					\end{split}
				\end{equation*}
				If {\rm\textbf{(A4)}} is satisfied, then
				\begin{equation*}
					\begin{split}
						\liminf_{N\rightarrow\infty}&\,\,
						\mathbb{P}_*\big[\,
								\mathbf{H}_{0} \text{ is rejected }
							\,\big]
							\geq
							\Prb\Big[\, 
								\mathfrak{T}_{a+\delta, b-\delta}(G) > q_\alpha
							\,\Big]\,.
					\end{split}
				\end{equation*}
		\item[(b)] {\rm \textbf{Case} $\boldsymbol\delta \boldsymbol< \mathbf{0}$\textbf{:}}
%				\begin{equation*}
%					\begin{split}
$						\lim_{N\rightarrow\infty}\,\,
							\mathbb{P}^*\big[\,
								\mathbf{H}_{0} \text{ is rejected }
							\,\big]
							= 0$.
%					\end{split}
%				\end{equation*}
		\item[(c)]  {\rm \textbf{Case} $\boldsymbol\delta \boldsymbol> \mathbf{0}$\textbf{:}}
%			\begin{equation*}
%					\begin{split}
						$\lim_{N\rightarrow\infty}\,\,
						\mathbb{P}_*\big[\,
							\mathbf{H}_{0} \text{ is rejected }
						\,\big]
							= 1$\,.
%					\end{split}
%			\end{equation*}		 	
	\end{itemize}
\end{theorem}
\begin{remark}
	If $\mathfrak{u}^{+}_{a + \delta}\cup\mathfrak{u}^{-}_{b - \delta} = \emptyset$, then all $q\in\mathbb{R}$ satisfy
	$ \Prb\big[\, \mathfrak{T}_{ a + \delta,  b - \delta }(G) \geq q \,\big] = 0 \leq \alpha $. However, 
	our construction implies
	$ \mathfrak{u}^{+}_{a + \delta} \cup \mathfrak{u}^{-}_{b - \delta} \neq\emptyset$.
	This follows directly
	from the definition of $\delta$ and the definition of the generalized
	preimage because $\delta$ satisfies that
	$a(s)+ \delta \geq \mu(s)$ and $b(s) - \delta \leq \mu(s)$ for all $s\in S$ and by construction
	either $\mathfrak{u}^{+}_{a + \delta}\neq \emptyset$ or $\mathfrak{u}^{-}_{b - \delta} \neq \emptyset$.
\end{remark}

\subsection{A Local Relevance Test with Strong FWER Control}

In the last section we saw that the global relevance hypothesis can be rejected at level $\alpha$
if certain estimated excursion sets are empty. In applications one is usually not only interested in
finding evidence for $\mu$ not being somewhere within a given envelop defined by the intervals $[b(s), a(s)]$, $s\in S$,
but wants to localize where $\mu(s) \notin [b(s), a(s)]$. Translated into the language of hypothesis testing this is
a multiple hypothesis test for the alternatives
\begin{equation*}
\begin{split}
	\mathbf{H}_{0,s}\mathbf{:}~ ~ \mu(s) \in \big[ b(s), a(s) \big]
	~ ~ ~ \text{ vs. } ~ ~ ~
	\mathbf{H}_{1,s}\mathbf{:}~ ~ \mu(s) \notin \big[ b(s), a(s) \big] \,,
\end{split}
\end{equation*} 
which controls the FWER in the strong sense at level $\alpha$. We will call such a test a local relevance test.
The idea is to use the membership in the excursion set for a decision on $\mathbf{H}_{0,s}$, i.e., 
\begin{equation}\label{eq:LocalRelevanceSCoRES}
	\text{\textbf{Reject} } \mathbf{H}_{0,s} \quad\Longleftrightarrow\quad s \in \hat{\cU}_{a + q\tau_N\sigma}
							~~\vee~~         s \in \hat{\cL}_{b - q\tau_N\sigma}\,.
\end{equation}
As the set of true null hypotheses is
$
	\mathcal{H}_0 = S \setminus (\, \cU_{a} \cup \cL_{b} \,)
$,
we immediately can show that a test defined by \eqref{eq:LocalRelevanceSCoRES} satisfies
	\begin{equation}\label{eqFWERlocTest}
	\begin{split}
		\mathbb{P}^*\big[\, \exists s \in \mathcal{H}_0:~ s \in \hat{\mathcal{H}}_1 \,\big]
		&= 1 -  \mathbb{P}_*\Big[\, \mathcal{H}_0 \subseteq S \setminus \big(\, 
					\hat{\cU}_{a + q\tau_N\sigma}
					~\cup~
					\hat{\cL}_{b - q\tau_N\sigma} \,\big)
				\,\Big]\\
		&\leq 1 -  \mathbb{P}_*\Big[\,
						\hat{\cU}_{a + q\tau_N\sigma} \subseteq \cU_{a}
					~\wedge~
						\hat{\cL}_{b - q\tau_N\sigma} \subseteq \cL_{b}
					 \,\Big]\,.
	\end{split}
	\end{equation}
Thus, we can asymptotically control the FWER using Corollary \ref{cor:Extraction} by choosing $q$ appropriately.
The only problem is --as in the global test case -- that in general
$ \mathfrak{u}^{+}_{a } \cup \mathfrak{u}^{-}_{b}$ can be the empty set which implies that the r.h.s of \eqref{eqFWERlocTest}
is asymptotically equal to $0$ for all $q\in\mathbb{R}$.
Therefore, inspired by the global relevance case and \cite{Dette:2020functional}, we  find the smallest
$\delta > 0$ such that $ \mathfrak{u}^{+}_{a + \delta } \cup \mathfrak{u}^{-}_{b - \delta} \neq \emptyset $,
yet $\cU_{a} = \cU_{a+\delta}$ and $\cL_{b} = \cL_{b-\delta}$ to make $q$ unique.
%The difference between the defintion of the
%quantile in this approach and the quantile of the global relevance test is shown in Figure \ref{??}.
%-------------------------------------------------------------------------------------------------------
%It is similar to what is done in
%\cite{Dette:2020functional, Dette:2022cov, Bucher2021deviations}
%with the important difference that our construction controls the FWER in the strong sense as
%shown in Theorem \ref{thm:RelTubeTest} while the current literature only provides FWER control in the weak sense.
%-------------------------------------------------------------------------------------------------------
%-------------------------------------------------------------------------------------------------------
\begin{theorem}\label{thm:RelTubeTest}
	 Let $\alpha \in (0,1)$, $a,b\in\cF(S)$, $\delta = \min\{ \delta_a, \delta_b \}$ with
	 \begin{equation}
	 	\delta_a = \sup_{ s \in S } \big\vert \mu(s) - a(s) \big\vert\,,\quad\quad\quad
	 	\delta_b = \sup_{ s \in S } \big\vert b(s) - \mu(s)\big\vert\,,
	 \end{equation}
	 $\cA = \big\{ a + \delta \big\}$  and  $\cB =  \big\{ b - \delta \big\}$. Assume that {\rm\textbf{(A1)}} is satisfied 
	 and $q_\alpha$ fulfills
	 $
		\Prb\big[\, \mathfrak{T}_{ a + \delta,  b - \delta }(G) \geq q_\alpha \,\big] \leq \alpha
	 $, if $\mathfrak{u}^{+}_{a + \delta} \cup \mathfrak{u}^{-}_{b - \delta} \neq \emptyset$, 
	and $q_\alpha = 0$ else.
	Then the local relevance test given by \eqref{eq:LocalRelevanceSCoRES} with quantile $q = q_\alpha$ satisfies
	\begin{itemize}
			\item[(a)]
				Assume {\rm\textbf{(A2)}} and {\rm\textbf{(A3)}}, then
				\begin{equation*}
					\begin{split}
						\limsup_{N\rightarrow\infty}&\,\,
										\mathbb{P}^*\big[\, \exists s \in \cH_0:~
																s\in \hat{\cH}_1
													\,\big]
													\leq \alpha\,.
				\end{split}
				\end{equation*}
				Assume ${\rm\textbf{(A4)}}$, then
				\begin{equation*}
				\begin{split}
						\liminf_{N\rightarrow\infty}&\,\,
										\mathbb{P}_*\big[\, \forall s \in \hat{\cH}_1:~
															s\in \cH_1
													\,\big]
													\geq 1-\alpha\,.
					\end{split}
				\end{equation*}
		\item[(b)]
			{\rm \textbf{Case} $\boldsymbol\delta \boldsymbol= \mathbf{0}$\textbf{:}} If {\rm\textbf{(A4)}}
			and either $\cH_0 = S$ or
				$\inf_{s\in S} \big(\, a(s) - b(s) \,\big) \geq M > 0$ hold true. Then
						\begin{equation*}
						\begin{split}
							\liminf_{N\rightarrow\infty}
							\mathbb{P}_*\big[\, \exists s \in \cH_0:~ s\in \hat{\cH}_1 \,\big]
							 \geq \Prb\Big[\, \mathfrak{T}_{a + \delta, b - \delta^+}(G) > q_\alpha \,\Big]\,.
						\end{split}
						\end{equation*}
			Define 
			$\cH_1^{\eta_N} = \cL_{b^--\eta_N} \cup \cU_{b^++\eta_N} $. If $\eta_N\xrightarrow{N\rightarrow\infty} 0 $ and $\tau_N^{-1}\eta_N \xrightarrow{N\rightarrow\infty} \infty$, then
			\begin{equation*}
			\begin{split}
							&\liminf_{N\rightarrow\infty}
							\mathbb{P}_*\big[\, \forall s \in \cH_1^{\eta_N}:~
													s\in \hat{\cH}_1
													\,\big] = 1\,.
			\end{split}
			\end{equation*}
		\item[(d)] {\rm \textbf{Case} $\boldsymbol\delta \boldsymbol > \mathbf{0}$\textbf{:}}
			Then
			\begin{equation*}
			\begin{split}
				\lim_{N\rightarrow\infty}
				\mathbb{P}^*\big[\, \exists s \in \cH_0:~
										s \in \hat{\cH}_1 \,\big]
										= 0\,,
				~~~~
				\lim_{N\rightarrow\infty}\mathbb{P}_*\big[\, \forall s \in \cH_1:~
										s \in \hat{\cH}_1 \,\big]
										= 1\,.
			\end{split}
			\end{equation*}
	\end{itemize}
\end{theorem}
%-------------------------------------------------------------------------------------------------------
%-------------------------------------------------------------------------------------------------------
%\begin{remark}\label{remark:emptymu}
%	For the lrT it can happen that even if $\cH_0 \neq \emptyset$ that
%	$\mathfrak{u}^{+}_{a + \delta} \cup \mathfrak{u}^{-}_{a - \delta} \neq \emptyset$, see
%	Fig. \ref{fig:EmptySC} in the appendix.
%\end{remark}
%-------------------------------------------------------------------------------------------------------
\begin{remark}
	One sided hypothesis tests can be obtained by setting either $a(s) = \infty $ or
	$b(s) = -\infty$ for all $s\in S$ for all $s\in S$. 
	The two-sided point hypotheses for a single $b\in\mathcal{F}(S)$
	is the special case where $a = b$. Theorem \ref{thm:RelTubeTest} shows that our test is
	asymptotically a consistent	test that controls the FWER at level $\alpha$. This test
	is uniformly more powerful than the standard asymptotic single-step test
	which uses the quantiles of $\max_{s \in S} \vert G(s) \vert$.
\end{remark}
%-------------------------------------------------------------------------------------------------------
\begin{remark}
	The local relevance test also yields a test for the global relevance hypothesis.
	Interestingly, the global test is not uniformly more powerful than the local test in
	detecting departures from the global null.
	The reason is that the tests from Theorem \ref{thm:Dette_Generalized} and
	\ref{thm:RelTubeTest} only differ in the definition of the critical sets over which the maxima of
	$\pm G$ are taken. In Figure \ref{fig:Comparison_DetteVsCoPE} we illustrate this
	under the assumption that $G$ is stationary.
\end{remark}
%-------------------------------------------------------------------------------------------------------

%-------------------------------------------------------------------------------------------------------
\begin{figure}[t!]
			\includegraphics[width=0.49\textwidth]{\figurepath 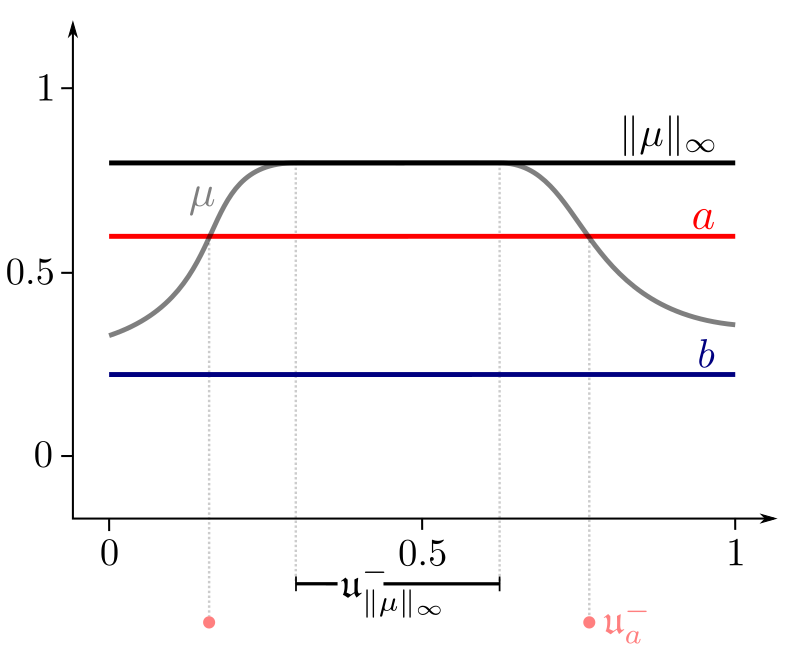}
			\includegraphics[width=0.49\textwidth]{\figurepath 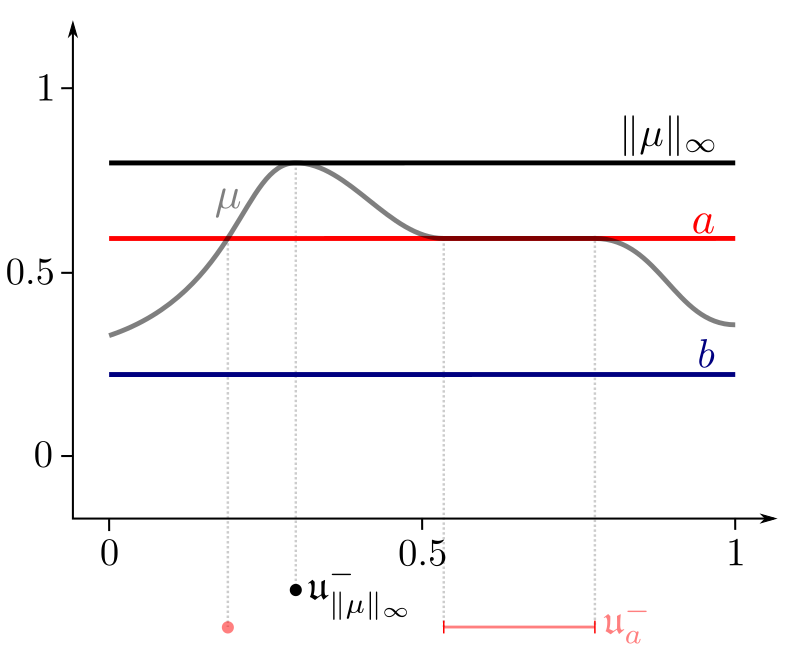}
		\caption{ Illustration of the critical sets $\mathfrak{u}^{-}_{a}$ and
				  $\mathfrak{u}^{-}_{\Vert \mu \Vert_\infty}$ of Theorem \ref{thm:Dette_Generalized} and
				  \ref{thm:RelTubeTest} respectively.
				  Under stationarity of $G$ it holds that the smaller the Lebesgue
				  volume of the critical set the smaller is the threshold $q$, i.e., the
				  power to reject the null hypothesis from Theorem \ref{thm:Dette_Generalized}
				  is higher.
				  \textit{Left:} the local relevance test has a higher power.
				  \text{Right:} the global relevance test has a higher power.
				   }\label{fig:Comparison_DetteVsCoPE}
\end{figure}
%-------------------------------------------------------------------------------------------------------

%-------------------------------------------------------------------------------------------------------
\subsection{Equivalence Tests}\label{scn:ETest}
%-------------------------------------------------------------------------------------------------------
In \cite{Dette:2021bio} a global equivalence test for data in $C(S)$, $S=[0,1]$, is proposed.
In this section we show that again Corollary \ref{cor:Extraction} can be used to weaken the assumptions
of their testing strategy.

Assume $a,b\in\cF(S)$ such that $b(s) < a(s)$ for all $s\in S$ are given. The statistical hypotheses of an \textit{equivalence test} are the hypotheses
\begin{equation*}
\begin{split}
	\mathbf{H}_{0}\mathbf{:}~ ~ \exists s \in S:~ \mu(s) \notin \big( b(s), a(s) \big)
	 ~ \text{ vs. } ~ 
\mathbf{H}_{1}\mathbf{:}~ ~\forall s \in S:~ \mu(s) \in \big( b(s), a(s) \big)\,.
\end{split}
\end{equation*}
Thus, the natural test statistic based on excursion sets is
\begin{equation}\label{eq:GlobalEquivalenceSCoRES}
	\text{\textbf{Reject} } \mathbf{H}_{0} \quad\Longleftrightarrow\quad \hat{\cU}_{a + q\tau_N\sigma} = \emptyset
							~~\wedge~~         \hat{\cL}_{-b - q\tau_N\sigma} = \emptyset\,.
\end{equation}
Using this we can generalize the test procedure proposed in \cite{Dette:2021bio} similar as in Theorem \ref{thm:Dette_Generalized}.
%-------------------------------------------------------------------------------------------------------
\begin{theorem}\label{thm:Equiv_2sides}
	 Let $\alpha \in (0,1)$, $a,b\in\cF(S)$ such that $\inf_{s\in S} \vert\, a(s) - b(s) \,\vert > 0$, $\delta$
	 be defined as in Theorem \ref{thm:Dette_Generalized},
	 $\cA = \big\{ a + \delta \big\}$  and  $\cB =  \big\{ b - \delta \big\}$.
	 Assume that {\rm\textbf{(A1)}} is satisfied  and $q_\alpha$ fulfills
	 $
		\Prb\big[\, \mathfrak{T}_{ a + \delta,  b - \delta }(G) \leq q_\alpha \,\big] \leq \alpha
	 $.
	Then the equivalence test given in \eqref{eq:GlobalEquivalenceSCoRES} with quantile $q = q_\alpha$ satisfies
	\begin{itemize}
			\item[(a)]
				{\rm \textbf{Case} $\boldsymbol\delta \boldsymbol = \mathbf{0}$\textbf{:}} If {\rm\textbf{(A4)}} is satisfied, then
				\begin{equation*}
					\begin{split}
						\limsup_{N\rightarrow\infty}&\,\,
							\mathbb{P}^*\big[\,
								\mathbf{H}_{0} \text{ is rejected }
							\,\big]
							\leq \alpha\,.
					\end{split}
				\end{equation*}
				If {\rm\textbf{(A2)}} and {\rm\textbf{(A3)}} are satisfied,	then
				\begin{equation*}
					\begin{split}
						\liminf_{N\rightarrow\infty}&\,\,
							\mathbb{P}_*\big[\,
									\mathbf{H}_{0} \text{ is rejected }
								\,\big]
							\geq
							\Prb\big[\,
									\mathfrak{T}_{a+\delta, b-\delta}( G ) < q_\alpha
								\,\big]\,.
					\end{split}
				\end{equation*}
		\item[(b)] {\rm \textbf{Case} $\boldsymbol\delta \boldsymbol> \mathbf{0}$\textbf{:}}
%				\begin{equation*}
%					\begin{split}
$						\lim_{N\rightarrow\infty}\,\,
							\mathbb{P}^*\big[\,
								\mathbf{H}_{0} \text{ is rejected }
							\,\big]
							= 0$.
%					\end{split}
%				\end{equation*}
		\item[(c)]  {\rm \textbf{Case} $\boldsymbol\delta \boldsymbol< \mathbf{0}$\textbf{:}}
%			\begin{equation*}
%					\begin{split}
						$\lim_{N\rightarrow\infty}\,\,
						\mathbb{P}_*\big[\,
							\mathbf{H}_{0} \text{ is rejected }
						\,\big]
							= 1$\,.
%					\end{split}
	\end{itemize}
\end{theorem}
%-------------------------------------------------------------------------------------------------------
\begin{remark}
	If we assume that $\mu,\hat\mu_N, a,b \in C([0,1])$ and $\sigma(s) = 1$, $s\in S$,
	then simple algebra shows that the above testing
	strategy is identical to the
	testing strategy from \cite{Dette:2021bio},
	because
	$
		 \hat{\cU}_{a + q\tau_N\sigma} = \emptyset
	$
	and
	$
	\hat{\cL}_{-b - q\tau_N\sigma} = \emptyset\
	$
	is equivalent to
	\begin{equation*}
		\max\left(\, \sup_{s\in S} \hat\mu_N(s) - a(s),~ \sup_{s\in S} b(s) - \hat\mu_N(s) \,\right)
		  \leq  \tau_N q_\alpha
	\end{equation*}
	and their theoretical quantile is identical to $q_\alpha$.
	The only difference is that in the definition of the quantile they use "$<$",
	which is irrelevant as in their model the cumulative distribution function of
	$\mathfrak{T}_{a+\delta, b-\delta}( G ) $ is continuous. 
\end{remark}
%-------------------------------------------------------------------------------------------------------
\begin{remark}
 Using Remark \ref{rmk:WeakerUpperBound} and comparing the assumptions
 of Theorem \ref{thm:Dette_Generalized} and \ref{thm:Equiv_2sides}
 we can conclude that the assumptions to get an equivalence test
 that asymptotically controls the FWER requires much weaker
 assumptions than in a relevance test, i.e., $\hat\mu_N$ needs to satisfy a ULT only on $\mu^{-1}_{\{a+\delta\}} \cup \mu^{-1}_{\{b-\delta\}}$, while Theorem \ref{thm:Dette_Generalized} requires a ULT on an
 $\tilde\eta$-thickening of $\mu^{-1}_{\{a-\delta\}} \cup \mu^{-1}_{\{b+\delta\}}$. 
\end{remark}
%-------------------------------------------------------------------------------------------------------

%-------------------------------------------------------------------------------------------------------
\subsection{Local Equivalence Tests}\label{scn:LocET}
%-------------------------------------------------------------------------------------------------------
We finally discuss a local version of the equivalence hypothesis, i.e.,
\begin{equation*}
\begin{split}
	\mathbf{H}_{0,s}\mathbf{:}~ ~ \mu(s) \notin \big( b(s), a(s) \big)
		\quad\text{vs.}\quad
	\mathbf{H}_{1,s}\mathbf{:}~ ~ \mu(s) \in \big( b(s), a(s) \big)\,,
\end{split}
\end{equation*}
for $a,b\in \cF(S)$ such that $\inf_{s\in S} a(s) - b(s) > 0$. A test based on excursion sets is the following:
\begin{equation}\label{eq:LocalEquivalenceSCoRES}
	\text{\textbf{Reject} } \mathbf{H}_{0,s} \quad\Longleftrightarrow\quad s \in S\setminus\hat{\cU}_{a - q\tau_N\sigma}
							\cap S\setminus\hat{\cL}_{b + q\tau_N\sigma}\,.
\end{equation}
We can proof a similar Theorem as Theorem \ref{thm:Dette_Generalized} for the local equivalence test.
%-------------------------------------------------------------------------------------------------------
\begin{theorem}\label{thm:lequivTest}
	 Let $\alpha \in (0,1)$, $a,b\in\cF(S)$, $\delta$ be defined as in Theorem \ref{thm:RelTubeTest},
	 $\cA = \big\{ a - \delta \big\}$  and  $\cB =  \big\{ b + \delta \big\}$.
	 Assume that {\rm\textbf{(A1)}} is satisfied 
	 and $q_\alpha$ fulfills
	 $
		\Prb\big[\, \mathfrak{T}_{ b - \delta,  a + \delta }(G) \geq q_\alpha \,\big] \leq \alpha
	 $, if $\mathfrak{u}^{+}_{b - \delta} \cup \mathfrak{u}^{-}_{a + \delta} \neq \emptyset$, 
	and $q_\alpha = 0$ else.
	Then the local equivalence test given by \eqref{eq:LocalEquivalenceSCoRES} with quantile $q = q_\alpha$ satisfies
	\begin{itemize}
			\item[(a)]
				Assume {\rm\textbf{(A2)}} and {\rm\textbf{(A3)}}, then
				\begin{equation*}
					\begin{split}
						\limsup_{N\rightarrow\infty}&\,\,
										\mathbb{P}^*\big[\, \exists s \in \cH_0:~
																s\in \hat{\cH}_1
													\,\big]
													\leq \alpha\,.
				\end{split}
				\end{equation*}
				Assume ${\rm\textbf{(A4)}}$, then
				\begin{equation*}
				\begin{split}
						\liminf_{N\rightarrow\infty}&\,\,
										\mathbb{P}_*\big[\, \forall s \in \hat{\cH}_1:~
															s\in \cH_1
													\,\big]
													\geq 1-\alpha\,.
					\end{split}
				\end{equation*}
		\item[(b)]
			{\rm \textbf{Case} $\boldsymbol\delta \boldsymbol= \mathbf{0}$\textbf{:}} If {\rm\textbf{(A4)}}
			and either $\cH_0 = S$ or
				$\inf_{s\in S} \big(\, a(s) - b(s) \,\big) \geq M > 0$ hold true. Then
						\begin{equation*}
						\begin{split}
							\liminf_{N\rightarrow\infty}
							\mathbb{P}_*\big[\, \exists s \in \cH_0:~ s\in \hat{\cH}_1 \,\big]
							 \geq \Prb\Big[\, \mathfrak{T}_{a + \delta, b - \delta}(G) > q_\alpha \,\Big]\,.
						\end{split}
						\end{equation*}
			Define 
			$\cH_1^{\eta_N} = \cL_{b^--\eta_N} \cup \cU_{b^++\eta_N} $. If $\eta_N\xrightarrow{N\rightarrow\infty} 0 $ and $\tau_N^{-1}\eta_N \xrightarrow{N\rightarrow\infty} \infty$, then
			\begin{equation*}
			\begin{split}
							&\liminf_{N\rightarrow\infty}
							\mathbb{P}_*\big[\, \forall s \in \cH_1^{\eta_N}:~
													s\in \hat{\cH}_1
													\,\big] = 1\,.
			\end{split}
			\end{equation*}
		\item[(d)] {\rm \textbf{Case} $\boldsymbol\delta \boldsymbol > \mathbf{0}$\textbf{:}}
			Then
			\begin{equation*}
			\begin{split}
				\lim_{N\rightarrow\infty}
				\mathbb{P}^*\big[\, \exists s \in \cH_0:~
										s \in \hat{\cH}_1 \,\big]
										= 0\,,
				~~~~
				\lim_{N\rightarrow\infty}\mathbb{P}_*\big[\, \forall s \in \cH_1:~
										s \in \hat{\cH}_1 \,\big]
										= 1\,.
			\end{split}
			\end{equation*}
	\end{itemize}
\end{theorem}
%-------------------------------------------------------------------------------------------------------
\begin{remark}\label{rmk:SCoPESPCII}
	A standard approach for an equivalence test of a parameter $\mu\in \mathbb{R}$ is
	the \textit{Principle of Confidence Interval Inclusion} (PCII) \cite[Chapter 3.1]{Wellek:2010EqvTesting}.
	Let $X$ denote the observations and $C_\pm(X; \alpha)$ the one-sided $(1-\alpha)$-confidence
	bounds for $\mu$, i.e.,
	$$
		\mathbb{P}\big[\, \mu \in \big(\! -\infty,\, C_-(X; \alpha) \big) \,\big]
		 = \mathbb{P}\big[\, \mu \in \big( C_+(X; \alpha ), \infty \big) \,\big] 
		 = \alpha\,.
	$$	
	Let $b < a$.  The alternatives
	$
	\mathbf{H}_0: \mu \notin (b, a)
	$
	vs.
	$
	 \mathbf{H}_1: \mu \in (b, a)
	$
	 can be tested at significance level $\alpha$ by rejecting $\mathbf{H}_0$ if
	$
		\big( C_-(X; \alpha),\, C_+(X; \alpha) \big) \subseteq ( b^-, b^+ )\,.
	$
	By construction the interval $\big( C_-(X; \alpha), C_+(X; \alpha) \big)$ constitutes a
	$(1-2\alpha)$-CI for $\mu$. This explains the name of the principle.
	The first conservative equivalence test was derived in the above fashion
	in  \cite{Westlake:1972} using a $(1-\alpha)$-CI.
	We argue now that the conservativeness has been a relic of using a
	CI instead of SCoRE sets. To see this, assume $\hat\mu_N$ satisfies
	$\tau_N^{-1}( \hat\mu_N - \mu ) / \sigma \rightsquigarrow G$
	and define
	$
		C_\pm(X; \alpha) = \hat{\mu}_N \pm q_\alpha^\pm\tau_N\sigma
	$
	with $q^\pm_\alpha$ satisfying
	$
		\mathbb{P}\big[\, \mp G \leq  q^\pm_\alpha \,\big] = 1-\alpha
	$.
	The PCII rejects $\mathbf{H}_0$ if
	$$
		\hat{\mu}_N \in \big(\, b  + \tau_N q^-_\alpha\sigma,\,  a  - \tau_N q^+_\alpha\sigma \,\big)\,.
	$$	
	The connection to the local equivalence test based on SCoRE sets is as follows.
	It can be easily verified that $q_\alpha$ from Theorem \ref{thm:lequivTest} satisfies
	\begin{equation*}
		q_\alpha = \begin{cases}
								q_\alpha^+\,, &\text{ if }~ ~ ~
											\vert\, \mu - a \,\vert \leq \vert\, \mu - b\,\vert\\
								q_\alpha^-\,, &\text{ else } 
					 \end{cases}
	\end{equation*}
	and that the local equivalence test based on SCoRE sets rejects $\mathbf{H}_0$ if
	$$
		\hat{\mu}_N \in \big(\, b  + \tau_N q_\alpha\sigma,\,  a  - \tau_N q_\alpha\sigma \,\big)\,,
	$$
	which is exactly the test derived from the PCII if $G$ has a symmetric distribution.
%	In fact, equivalence tests of this form are under weak conditions asymptotically optimal \cite{Romano:2005optimal}.
\end{remark}
%-------------------------------------------------------------------------------------------------------

\section{An Application to Scheff\'e Type Inference For Multiple Linear Regression}\label{scn:Scheffe}
In order to highlight potential benefits from statistical inference using SCoRE sets, we compare them
in this section with Scheff\'e's simultaneous testing for all contrasts in multiple linear regression models.
In order to explain the difference more concretely, we keep the
considered probabilistic model simple by restricting ourselves to homoscedastic Gaussian errors. 

Let $ (X_N)_{N \in \mathbb{N}} \in \mathbb{R}^{N \times K} $ be a sequence
of design matrices of rank $ K < N $,  $(\tau_N)_{N\in \mathbb{N}}$ be a
sequence satisfying
$ \lim_{N\rightarrow\infty} \tau_{N}^{-2} (X_N^TX_N)^{-1} = \mathfrak{X} \in \mathbb{R}^{K \times K}$
with $\mathfrak{X}$ invertible and
the observations $\boldsymbol{y}_N$ are generated from a linear model, i.e.,
$\boldsymbol{y}_N \sim \mathcal{N}\left( X_N\boldsymbol{\beta}, \xi^2 I_{N \times N }\right)$
with $\xi \in \mathbb{R}_{>0}$.
Recall that $\hat{\boldsymbol{\beta}}_N = (X_N^TX_N)^{-1}X^T_N\boldsymbol{y}_N$,
$N \in \mathbb{N}$, is the \textit{uniformly minimum-variance unbiased estimator} (UMVU)
of $\boldsymbol{\beta}$.
Usually not $\boldsymbol{\beta}$, but linear contrasts
$\mu(\boldsymbol{a}) = \boldsymbol{a}^T\boldsymbol{\beta}$, $\boldsymbol{a} \in \mathbb{R}^K$,
are of interest.
Interpreting $ \boldsymbol{a} \in \mathbb{R}^K $ as the parameter set,
we define a stochastic process indexed in $\mathbb{R}^K$ and its asymptotic variance by
$
	\hat\mu_N(\boldsymbol{a}) = \boldsymbol{a}^T \hat{\boldsymbol{\beta}}_N
$, 
$
	\sigma^2(\boldsymbol{a}) = \xi^2\boldsymbol{a}^T\mathfrak{X}\boldsymbol{a}
$.
Since all involved quantities are Gaussian and continuous in $\boldsymbol{a}$, we obtain
\begin{equation}\label{eq:ContrastULT}
	 G_N( \boldsymbol{a} ) = \frac{ \hat\mu_N(\boldsymbol{a}) - \mu(\boldsymbol{a}) }{ \tau_N \xi\sqrt{\boldsymbol{a}^T\mathfrak{X}\boldsymbol{a}} } \rightsquigarrow G(\boldsymbol{a})
\end{equation}
weakly in $C\big( \mathbb{R}^K \big)$, where $G$ is the zero-mean Gaussian process with covariance function
\begin{equation*}
	\mathfrak{c}(\boldsymbol{a}, \boldsymbol{b}) = \frac{\boldsymbol{a}^T\mathfrak{X}\boldsymbol{b}}{ \sqrt{\boldsymbol{a}^T\mathfrak{X}\boldsymbol{a}}\sqrt{\boldsymbol{b}^T\mathfrak{X}\boldsymbol{b}} }\,,~~~ \text{ for }\boldsymbol{a},\boldsymbol{b}\in \mathbb{R}^K\,.
\end{equation*}
Since $G_N(\boldsymbol{a})= G_N(\tilde{\boldsymbol{a}})$ and
$G(\boldsymbol{a})= G(\tilde{\boldsymbol{a}})$ with
$\tilde{\boldsymbol{a}} = \boldsymbol{a} / \Vert \boldsymbol{a} \Vert$,
the domain of the processes can be assumed to be the compact space
$\mathbb{S}^{K-1} = \{ \boldsymbol{a} \in \mathbb{R}^K ~\vert~ \Vert \boldsymbol{a} \Vert = 1 \}$
instead of $\mathbb{R}^K$.

A standard task in multiple linear regression is finding
contrasts $\boldsymbol{b} \in \mathbb{S}^{K-1}$ using the observations $\boldsymbol{y}_N$
such that with high probability $\boldsymbol{b} \in \{ \boldsymbol{a} \in \mathbb{S}^{K-1}~\vert~\boldsymbol{a}^T\boldsymbol{\beta} \neq 0 \} $.
This can be achieved for example using SCBs \cite{Scheffe:1953}.
The asymptotic analogue of Scheff\'e's ${(1-\alpha)}$-SCBs for contrasts
are given by the intervals with endpoints
\begin{equation}\label{eq:ScheffeBand}
	\hat{\mu}(\boldsymbol{a}) \pm q_{1-\alpha, K} \sigma(\boldsymbol{a})
		= \boldsymbol{a}^T \hat{\boldsymbol{\beta}}_N
				\pm \tau_Nq_{1-\alpha, K}\xi \sqrt{\boldsymbol{a}^T\mathfrak{X}\boldsymbol{a}}\,,
\end{equation}
where $q_{1-\alpha, K}^2$ is the $(1-\alpha)$-quantile of a $\chi^2_K$-distributed random variable
\cite[eq. (8.71)]{Rencher:2008linear}. Based on this and using Theorem 8.5 from \cite{Rencher:2008linear}
the asymptotic version of Scheff\'e's test rejects the null hypothesis of
$\boldsymbol{a}^T\boldsymbol{\beta} = 0$ for all $\boldsymbol{a}\in\mathbb{S}^{K-1}$
at significance level $\alpha$, if
\begin{equation*}
	\xi^{-2}\big( \hat{\boldsymbol{\beta}}_N^T \mathfrak{X}^T\mathfrak{X}\hat{\boldsymbol{\beta}}_N \big)^2 > \tau_N^2q_{1-\alpha, K}^2\,,
\end{equation*}
which is equivalent to the existence of $\boldsymbol{a}\in\mathbb{S}^{K-1}$
such that zero is not contained in the interval given by \eqref{eq:ScheffeBand}.
Our Corollary \ref{cor:SCB_CoPE} shows that this SCB contains
more information than allowing us to perform a valid hypothesis test for
$\boldsymbol{a}^T\boldsymbol{\beta} = 0$ for all $\boldsymbol{a}\in\mathbb{S}^{K-1}$.
We actually know that
\begin{equation*}
		\Prb\left[ \forall c  \in \mathcal{F}\big( \mathbb{R}^K \big):~ \hat{\cL}_{c - q_{1-\alpha, K}K\sigma} \subseteq \cL_c ~\wedge~
		 \hat{\cU}_{c + q_{1-\alpha, K}K\sigma} \subseteq \cU_c  \,\right]
 = 1-\alpha\,.
\end{equation*}
This implies for the function $c$ satisfying $c(\boldsymbol{a}) = 0$ for all $\boldsymbol{a}\in\mathbb{S}^{K-1}$
that all contrasts $\boldsymbol{b} \in \mathbb{S}^{K-1}$ contained in either
of the two sets
\begin{equation*}
\begin{split}
  	 \hat{\cL}_{-\tau_Nq_{1-\alpha, K}\sigma}
  	 		&= \Big\{\, \boldsymbol{a} \in \mathbb{S}^{K-1}:~ \boldsymbol{a}^T\hat{\boldsymbol{\beta}}
	 				< -\tau_Nq_{1-\alpha,K}\xi\sqrt{\boldsymbol{a}^T\mathfrak{X}\boldsymbol{a}} \,\Big\}\,,\\
 \hat{\cU}_{\tau_Nq_{1-\alpha, K}\sigma}
 			&= \Big\{\, \boldsymbol{a} \in \mathbb{S}^{K-1}:~ \boldsymbol{a}^T\hat{\boldsymbol{\beta}}
	 				> \tau_Nq_{1-\alpha,K}\xi\sqrt{\boldsymbol{a}^T\mathfrak{X}\boldsymbol{a}} \,\Big\}
\end{split}
\end{equation*}
are asymptotically with probability at least $1-\alpha$ correctly discovered to be non-zero contrasts, i.e.,
$\boldsymbol{b} \in \{ \boldsymbol{a}\in \mathbb{S}^{K-1}~\vert~ \boldsymbol{a}^T\boldsymbol{\beta} \neq 0 \} = \cL_c \cup \cU_c$. This result holds independently of the actual value of $\boldsymbol{\beta}$.
The drawback is low detection power since it allows to simultaneously construct
confidence subsets for any function $c\in\mathcal{F}(\mathbb{S}^{k-1})$.
The power can be improved by applying our Corollary
\ref{cor:MaxImproved} which gives upper and lower confidence
subsets solely for the function $c(\boldsymbol{a}) = 0$ for all $\boldsymbol{a}\in\mathbb{S}^{K-1}$.
\begin{corollary}\label{thm:ScheffeSSSCoPE_zero}
	Assume the multiple linear regression model depending on $N$ as defined above.
	Let $c(\boldsymbol{a}) = 0$ for all $\boldsymbol{a}\in \mathbb{S}^{K-1}$ and $q\in \mathbb{R}_{ \geq 0}$.
	Then
	\begin{equation*}
		\lim_{N \rightarrow \infty}
			\Prb \left[\, \hat{\cL}_{-\tau_N q \sigma } \subseteq  \cL_0 ~\wedge~ \hat{\cU}_{\tau_N q \sigma } \subseteq  \cU_0 \,\right]
		 	= \chi^2\big( q^2, K - 1 + \mathds{1}_{\boldsymbol{\beta}=0} \big)\,.
	\end{equation*}
	Here $\mathds{1}_{\boldsymbol{\beta}=0}$ is the indicator function,	i.e., one, if $\boldsymbol{\beta}=0$, and zero else, and
	$u \mapsto \chi^2(u, k)$ is the cumulative distribution function of
	a $ \chi^2_k $-distributed random variable.
\end{corollary}
If $\boldsymbol{\beta}=0$ is true or we want to secure ourselves against this case,
the above corollary does not and should not allow to improve the detection power
compared to the asymptotic SCB approach based on Scheff\'e as
in this case the limit distribution agrees with the limit distribution used to construct the SCB.
However, by taking into account that the limit distribution depends on the true $\boldsymbol{\beta}$ as
shown in the above result, we can devise a more powerful strategy to detect non-zero contrasts
$\boldsymbol{a}\in\mathbb{R}^K$ for $\boldsymbol{\beta}$, while controlling the FWER at level $\alpha$.

To do so we use contrary to the SCBs the $\boldsymbol{\beta}\neq 0$ case to tune the quantile for detection.
If $q_{1-\alpha, K-1}^2$ is the $(1-\alpha)$-quantile of a $\chi^2_{K-1}$-distributed
random variable, then Corollary \ref{thm:ScheffeSSSCoPE_zero} yields
	\begin{equation*}
	\begin{split}
		&\lim_{N \rightarrow \infty}
			\Prb \left[\, 
			\hat{ \cL}_{-\tau_N q_{1-\alpha,K-1} \sigma } \subseteq  \cL_0
			 ~\wedge~
			 \hat{ \cU}_{\tau_N q_{1-\alpha,K-1} \sigma } \subseteq  \cU_0 \,\right]\\
		 	&=\begin{cases}
		 		1-\alpha\,, &\text{ if } \boldsymbol{\beta} \neq 0 \\
		 		 \chi^2\big( q_{1-\alpha, K-1}^2, K \big)\,, &\text{ if } \boldsymbol{\beta} = 0
		 	  \end{cases}\,.
	\end{split}
	\end{equation*}
This means that, if $\boldsymbol{\beta} \neq 0$, all as non-zero discovered contrasts $\boldsymbol{a}$, i.e.,
$\boldsymbol{a} \in \hat{\cL}_{-\tau_N q_{1-\alpha,K-1} \sigma } \cup \hat{\cU}_{\tau_N q_{1-\alpha,K-1} \sigma }$,
are asymptotically with probability ${1-\alpha}$ correctly identified to be non-zero.
On the other hand, if $\boldsymbol{\beta} = 0$, it holds that $\cL_0 = \cU_0 = \emptyset$.
Thus, asymptotically the probability to find (wrongly) a non-zero contrast
is $1-\chi^2\big( q_{1-\alpha,K-1}^2, K \big)>\alpha$. This is the quantifiable price to pay
for a higher power.
To illustrate it, assume $ {\alpha} = 0.05 $ and $k=4$. Then all discovered contrasts are with
probability $0.95$ correctly identified to be non-zero contrasts, if
the true $\boldsymbol{\beta} \neq 0$. In the worst case that $\boldsymbol{\beta} = 0$,
then the probability of wrongly discovering a non-zero contrast is
$1-\chi^2\big( q_{0.95, 3}^2, 4 \big)\approx 0.1$. As it is impossible to judge with certainty whether
$\boldsymbol{\beta} = 0$, the researcher can use the information from $\hat{\boldsymbol{\beta}}$ to judge how
risky it is to ignore the worst case scenario and communicate this in forms of confidence bands for $\boldsymbol{\beta}$ or alike.

%-------------------------------------------------------------------------------------------------------
%-------------------------------------------------------------------------------------------------------
%-------------------------------------------------------------------------------------------------------

\section{Discussion}\label{scn:discussion}
In this article we refined, extended and unified different statistical inference tools
for a target function $\mu$ from an estimator $\hat\mu_N$ which
control FWER-like criteria over a metric space $S$.
In particular, we demonstrated that CoPE sets \cite{Sommerfeld:2018CoPE},
SCBs and recently proposed relevance and equivalence tests for $C(S)$-valued data based on the supremum
norm, among others \cite{Dette:2020functional},
can all be derived from Theorem \ref{thm:MainSCoPES}.
Furthermore, our abstract viewpoint allowed us to weaken the assumptions of
the aforementioned methods and clarify some of their conceptual shortcomings,
for example, by changing the definition of the inclusion statement
from \cite{Sommerfeld:2018CoPE}.
We finish our current endeavor with a few remarks
and directions for future research.

	As, for example, in nonparametric statistics $\hat\mu_N$ often does not satisfy a ULT,
	it is worthwhile to explore assumptions different to \textbf{(A1)}-\textbf{(A3)}  which enable us
	to compute the limit distributions given in the SCoRE set Metatheorem from Appendix
	\ref{scn:MetaTheorem}. Especially, it might be possible to derive
	the testing strategy from \cite{Bucher2021deviations} MARKER (maybe add more) from our SCoRE sets framework as
	it is very similar to the testing problems discussed in this work,
	yet their estimator does not satisfy a ULT. Additionally, we believe that the SCoRE set metatheorem and
	the change of the inclusion statement might allow us to remove Assumption 2.2.3
	from \cite{Maullin:2022} and extend their construction of confidence regions for
	intersections and unions of excursion sets of several estimators to more complicated combinations
	of excursion sets such as the symmetric difference.
	Furthermore, it would be of interest to extend our current theory to develop honest and adaptive SCoRE sets
	which would generalize the concept of honest and adaptive simultaneous confidence bands,
	among others, \cite{li:1989honest,Gine:2010confidence,Hoffmann:2011},
	or to relax the relatively strict FWER like requirement for  the inclusion statements to hold
	with a certain probability in our SCORE sets to weaker assertions resembling FDR
	(e.g.,\cite{Benjamini:1995controlling}) or $k$-FWER (e.g., \cite{Lehmann:2005})
	multiple comparison	corrections.

	Finally, as one of our main contributions was to shed light on the connections between confidence regions
	for excursion sets and hypothesis testing, we want to remedy a misconception from \cite{Bowring:2019} which
	motivates CoPE sets as a solution to the paradox caused by the fallacy of the null hypothesis
	\cite{Rozeboom:1960}.
	They write \textit{"[...] the paradox is that
	while statistical models conventionally assume mean-zero
	noise, in reality all sources of noise will never cancel, and therefore improvements
	in experimental design will eventually lead to statistically significant
	results. Thus, the null hypothesis will, eventually, always be rejected
	\cite{Meehl:1967}. [...]"} and later they write \textit{"[...] Unlike hypothesis testing, our spatial
	Confidence Sets (CSs) allow for inference on non-zero raw effect sizes. [...]"}.
	The fallacy of the null hypothesis can be an important practical problem.
	Yet CoPE sets fall short being a conceptual solution because they still assume a zero-mean noise model.
	Therefore their inference on level sets $c$ of the true signal suffers from the same problem
	of finding signals above $c$ if the noise is not zero-mean.
	To make this point more clear by our established duality Proposition \ref{prop:duality} any FWER
	test with strong control at level $\alpha$ which tests the alternatives
	$\mathbf{H}_{0,s}:~\mu(s) = c$ vs. $\mathbf{H}_{1,s}:~\mu(s) \neq c$
	would be a solution to the fallacy of the null hypothesis, if CoPE sets are.
	An actual possibility to dissipate Meehl's concern
	are for example $(1-\alpha)$-SCoRE sets over $\big( \{ c + \Delta \}, \{ c - \Delta  \} \big)$,
	$\Delta >0$, if we can assume that
	the mean of the error process is bounded within
	$[-\Delta, \Delta]$. Yet again this approach is
	dual to a local relevance test by Proposition \ref{prop:duality}.
	
	So what are the advantages of SCoRE sets over hypothesis testing?
	First and foremost SCoRE sets break with the dogma of phrasing
	research questions in terms of statistical hypotheses.
	They emphasize what really matters: a quantifiable observable $\mu$
	and what can be concluded from an experiment about the uncertainty of preimages
	of $\mu$ which are relevant
	for the researcher. Secondly, they allowed us to derive the oracle limiting
	distributions for the canonical hypotheses tests derived from a ULT that control the FWER in the
	strong sense at level $\alpha$ and thereby disclose
	the actual target of such multiple tests.
	Thirdly, the natural interpretation of SCoRE sets, nicely presented\footnote{The description actually
	corresponds to the inclusion statement from this article and not to the inclusion statement
	from \cite{Sommerfeld:2018CoPE} that was used to generate the figure.} for $\cA=\cB = \{c\}$
	in Fig.1 of \cite{Bowring:2019}, is closer to the interpretation of confidence intervals
	and therefore hopefully does cause less confusion among students,
	practitioners and experts than the interpretation of tests and $p$-values.

%========================================================================
\section*{Acknowledgments}
F.T. is funded by the Deutsche Forschungsgemeinschaft (DFG) under Excellence Strategy The
Berlin Mathematics Research Center MATH+ (EXC-2046/1, project ID:390685689).
F.T. and A.S. were partially supported by NIH grant R01EB026859.
The first ideas of this paper emerged during a revision of a project with A.
Bowring and T. Nichols and we are thankful to both for helpful discussions
in early stages of the manuscript.
F.T. also wants to thank the WIAS Berlin, where parts of this research was performed,
for offering a guest researcher status and especially K. Tabelow and J. Pohlzehl
for their general hospitality.
F.T. thanks D. Liebl for reading through the introduction and the example on multiple linear regression and
giving precious ideas how to streamline the presentation and
providing the opportunity to present parts of this work on
a conference and in a seminar in Bonn.
F.T. owes special gratitude to B. Stankewitz for tremendous psychological support
during Covid times, helpful discussions on uniform convergence
(Lemma \ref{lem:ConvergenceMaxima}), struggling through the introduction
and giving valuable feedback when the notation was still a complete mess.
%========================================================================

%========================================================================
% Bibliography
%========================================================================
\bibliographystyle{unsrtnat}
\bibliography{paper-ref}

\begin{thebibliography}{46}
\providecommand{\natexlab}[1]{#1}
\providecommand{\url}[1]{\texttt{#1}}
\expandafter\ifx\csname urlstyle\endcsname\relax
  \providecommand{\doi}[1]{doi: #1}\else
  \providecommand{\doi}{doi: \begingroup \urlstyle{rm}\Url}\fi

\bibitem[Neyman(1937)]{Neyman:1937confidence}
Jerzy Neyman.
\newblock Outline of a theory of statistical estimation based on the classical
  theory of probability.
\newblock \emph{Philosophical Transactions of the Royal Society of London.
  Series A, Mathematical and Physical Sciences}, 236\penalty0 (767):\penalty0
  333--380, 1937.

\bibitem[Lehmann et~al.(2005)Lehmann, Romano, and Casella]{Lehmann:2005testing}
Erich~Leo Lehmann, Joseph~P Romano, and George Casella.
\newblock \emph{Testing statistical hypotheses}, volume~3.
\newblock Springer, 2005.

\bibitem[Aitchison(1964)]{Aitchison:1964confidence}
John Aitchison.
\newblock Confidence-region tests.
\newblock \emph{Journal of the Royal Statistical Society: Series B
  (Methodological)}, 26\penalty0 (3):\penalty0 462--476, 1964.

\bibitem[Aitchison(1965)]{Aitchison:1965likelihood}
J~Aitchison.
\newblock Likelihood-ratio and confidence-region tests.
\newblock \emph{Journal of the Royal Statistical Society: Series B
  (Methodological)}, 27\penalty0 (2):\penalty0 245--250, 1965.

\bibitem[Gabriel(1969)]{Gabriel:1969simultaneous}
K~Ruben Gabriel.
\newblock Simultaneous test procedures--some theory of multiple comparisons.
\newblock \emph{The Annals of Mathematical Statistics}, 40\penalty0
  (1):\penalty0 224--250, 1969.

\bibitem[Stefansson(1988)]{Stefansson:1988}
Gunnar Stefansson.
\newblock On confidence sets in multiple comparisons.
\newblock \emph{Statistical Decision Theory and Related Topics IV}, 2:\penalty0
  89--104, 1988.

\bibitem[Hayter and Hsu(1994)]{Hayter:1994relationship}
Anthony~J Hayter and Jason~C Hsu.
\newblock On the relationship between stepwise decision procedures and
  confidence sets.
\newblock \emph{Journal of the American Statistical Association}, 89\penalty0
  (425):\penalty0 128--136, 1994.

\bibitem[Holm(1999)]{Holm:1999}
Sture Holm.
\newblock Multiple confidence sets based on stagewise tests.
\newblock \emph{Journal of the American Statistical Association}, 94\penalty0
  (446):\penalty0 489--495, 1999.

\bibitem[Guilbaud(2008)]{Guilbaud:2008}
Olivier Guilbaud.
\newblock Simultaneous confidence regions corresponding to holm's step-down
  procedure and other closed-testing procedures.
\newblock \emph{Biometrical Journal: Journal of Mathematical Methods in
  Biosciences}, 50\penalty0 (5):\penalty0 678--692, 2008.

\bibitem[Magirr et~al.(2013)Magirr, Jaki, Posch, and
  Klinglmueller]{Magirr:2013}
Dominic Magirr, Thomas Jaki, Martin Posch, and F~Klinglmueller.
\newblock Simultaneous confidence intervals that are compatible with closed
  testing in adaptive designs.
\newblock \emph{Biometrika}, 100\penalty0 (4):\penalty0 985--996, 2013.

\bibitem[Wasserstein and Lazar(2016)]{Wasserstein:2016_ASA}
Ronald~L Wasserstein and Nicole~A Lazar.
\newblock The asa statement on p-values: context, process, and purpose.
\newblock \emph{The American Statistician}, 70\penalty0 (2):\penalty0 129--133,
  2016.

\bibitem[Westlake(1972)]{Westlake:1972}
Wilfred~J Westlake.
\newblock Use of confidence intervals in analysis of comparative
  bioavailability trials.
\newblock \emph{Journal of Pharmaceutical Sciences}, 61\penalty0 (8):\penalty0
  1340--1341, 1972.

\bibitem[Schuirmann(1981)]{Schuirmann:1981}
DL~Schuirmann.
\newblock On hypothesis-testing to determine if the mean of a
  normal-distribution is contained in a known interval.
\newblock In \emph{Biometrics}, volume~37, pages 617--617, 1981.

\bibitem[Schuirmann(1987)]{Schuirmann:1987}
Donald~J Schuirmann.
\newblock A comparison of the two one-sided tests procedure and the power
  approach for assessing the equivalence of average bioavailability.
\newblock \emph{Journal of Pharmacokinetics and Biopharmaceutics}, 15\penalty0
  (6):\penalty0 657--680, 1987.

\bibitem[Hauck and Anderson(1992)]{Hauck:1992types}
WW~Hauck and S~Anderson.
\newblock Types of bioequivalence and related statistical considerations.
\newblock \emph{International Journal of Clinical Pharmacology, Therapy, and
  Toxicology}, 30\penalty0 (5):\penalty0 181--187, 1992.

\bibitem[Berger and Hsu(1996)]{Berger:1996}
Roger~L Berger and Jason~C Hsu.
\newblock Bioequivalence trials, intersection-union tests and equivalence
  confidence sets.
\newblock \emph{Statistical Science}, 11\penalty0 (4):\penalty0 283--319, 1996.

\bibitem[Sommerfeld et~al.(2018)Sommerfeld, Sain, and
  Schwartzman]{Sommerfeld:2018CoPE}
Max Sommerfeld, Stephan Sain, and Armin Schwartzman.
\newblock Confidence regions for spatial excursion sets from repeated random
  field observations, with an application to climate.
\newblock \emph{Journal of the American Statistical Association}, 113\penalty0
  (523):\penalty0 1327--1340, 2018.

\bibitem[Bowring et~al.(2019)Bowring, Telschow, Schwartzman, and
  Nichols]{Bowring:2019}
Alexander Bowring, Fabian Telschow, Armin Schwartzman, and Thomas~E Nichols.
\newblock Spatial confidence sets for raw effect size images.
\newblock \emph{NeuroImage}, 203:\penalty0 116187, 2019.

\bibitem[Bowring et~al.(2021)Bowring, Telschow, Schwartzman, and
  Nichols]{Bowring:2021}
Alexander Bowring, Fabian~JE Telschow, Armin Schwartzman, and Thomas~E Nichols.
\newblock Confidence {S}ets for {C}ohen’s $d$ effect size images.
\newblock \emph{NeuroImage}, 226:\penalty0 117477, 2021.

\bibitem[Ren et~al.(2022)Ren, Telschow, and Schwartzman]{Ren:2022}
Junting Ren, Fabian~JE Telschow, and Armin Schwartzman.
\newblock Inverse set estimation and inversion of simultaneous confidence
  intervals.
\newblock \emph{arXiv preprint arXiv:2210.03933}, 2022.

\bibitem[Dette and Tang(2021)]{Dette2021FunctionOnFunction}
Holger Dette and Jiajun Tang.
\newblock Statistical inference for function-on-function linear regression.
\newblock \emph{arXiv preprint arXiv:2109.13603}, 2021.

\bibitem[Dette and Kokot(2022)]{Dette:2022cov}
Holger Dette and Kevin Kokot.
\newblock Detecting relevant differences in the covariance operators of
  functional time series: a sup-norm approach.
\newblock \emph{Annals of the Institute of Statistical Mathematics},
  74\penalty0 (2):\penalty0 195--231, 2022.

\bibitem[Dette and Kokot(2021)]{Dette:2021bio}
Holger Dette and Kevin Kokot.
\newblock Bio-equivalence tests in functional data by maximum deviation.
\newblock \emph{Biometrika}, 108\penalty0 (4):\penalty0 895--913, 2021.

\bibitem[Mammen and Polonik(2013)]{Mammen:2013}
Enno Mammen and Wolfgang Polonik.
\newblock Confidence regions for level sets.
\newblock \emph{Journal of Multivariate Analysis}, 122:\penalty0 202--214,
  2013.

\bibitem[Qiao and Polonik(2019)]{Qiao:2019}
Wanli Qiao and Wolfgang Polonik.
\newblock Nonparametric confidence regions for level sets: Statistical
  properties and geometry.
\newblock \emph{Electronic Journal of Statistics}, 13\penalty0 (1):\penalty0
  985--1030, 2019.

\bibitem[French et~al.(2017)French, McGinnis, and
  Schwartzman]{French:2017assessing}
Joshua~P French, Seth McGinnis, and Armin Schwartzman.
\newblock Assessing narccap climate model effects using spatial confidence
  regions.
\newblock \emph{Advances in Statistical Climatology, Meteorology and
  Oceanography}, 3\penalty0 (2):\penalty0 67--92, 2017.

\bibitem[Maullin-Sapey et~al.(2023)Maullin-Sapey, Schwartzman, and
  Nichols]{Maullin:2022}
T~Maullin-Sapey, A~Schwartzman, and TE~Nichols.
\newblock Spatial confidence regions for combinations of excursion sets in
  image analysis.
\newblock \emph{Journal of the Royal Statistical Society: Statistical
  Methodology Series B}, 2023.

\bibitem[Degras(2011)]{Degras:2011SCB}
David~A Degras.
\newblock Simultaneous confidence bands for nonparametric regression with
  functional data.
\newblock \emph{Statistica Sinica}, pages 1735--1765, 2011.

\bibitem[Telschow and Schwartzman(2022)]{Telschow:2022SCB}
Fabian~JE Telschow and Armin Schwartzman.
\newblock Simultaneous confidence bands for functional data using the
  {G}aussian kinematic formula.
\newblock \emph{Journal of Statistical Planning and Inference}, 216:\penalty0
  70--94, 2022.

\bibitem[Liebl and Reimherr(2019)]{Liebl:2019fast}
Dominik Liebl and Matthew Reimherr.
\newblock Fast and fair simultaneous confidence bands for functional
  parameters.
\newblock \emph{arXiv preprint arXiv:1910.00131}, 2019.

\bibitem[Dette et~al.(2020)Dette, Kokot, and Aue]{Dette:2020functional}
Holger Dette, Kevin Kokot, and Alexander Aue.
\newblock Functional data analysis in the {B}anach space of continuous
  functions.
\newblock \emph{The Annals of Statistics}, 48\penalty0 (2):\penalty0
  1168--1192, 2020.

\bibitem[Dette et~al.(2018)Dette, M{\"o}llenhoff, Volgushev, and
  Bretz]{Dette2018:equivalenceReg}
Holger Dette, Kathrin M{\"o}llenhoff, Stanislav Volgushev, and Frank Bretz.
\newblock Equivalence of regression curves.
\newblock \emph{Journal of the American Statistical Association}, 113\penalty0
  (522):\penalty0 711--729, 2018.

\bibitem[Van Der~Vaart et~al.(1996)Van Der~Vaart, van~der Vaart, van~der Vaart,
  and Wellner]{Vaart:1996weak}
Aad~W Van Der~Vaart, Adrianus~Willem van~der Vaart, Aad van~der Vaart, and Jon
  Wellner.
\newblock \emph{Weak convergence and empirical processes: with applications to
  statistics}.
\newblock Springer-Verlag, New York, 1996.

\bibitem[Tuzhilin(2020)]{Tuzhilin:2020}
Alexey~A Tuzhilin.
\newblock Lectures on {H}ausdorff and {G}romov-{H}ausdorff distance geometry.
\newblock \emph{arXiv preprint arXiv:2012.00756}, 2020.

\bibitem[Holm(1979)]{Holm:1979}
Sture Holm.
\newblock A simple sequentially rejective multiple test procedure.
\newblock \emph{Scandinavian journal of statistics}, pages 65--70, 1979.

\bibitem[Wellek(2010)]{Wellek:2010EqvTesting}
Stefan Wellek.
\newblock \emph{Testing statistical hypotheses of equivalence}.
\newblock Chapman and Hall/CRC, Boca Raton, 2010.

\bibitem[Scheff{\'e}(1953)]{Scheffe:1953}
Henry Scheff{\'e}.
\newblock A method for judging all contrasts in the analysis of variance.
\newblock \emph{Biometrika}, 40\penalty0 (1-2):\penalty0 87--110, 1953.

\bibitem[Rencher and Schaalje(2008)]{Rencher:2008linear}
Alvin~C Rencher and G~Bruce Schaalje.
\newblock \emph{Linear models in statistics}.
\newblock John Wiley \& Sons, Inc., Hoboken, New Jerseys, 2008.

\bibitem[B{\"u}cher et~al.(2021)B{\"u}cher, Dette, and
  Heinrichs]{Bucher2021deviations}
Axel B{\"u}cher, Holger Dette, and Florian Heinrichs.
\newblock Are deviations in a gradually varying mean relevant? {A} testing
  approach based on sup-norm estimators.
\newblock \emph{The Annals of Statistics}, 49\penalty0 (6):\penalty0
  3583--3617, 2021.

\bibitem[Li(1989)]{li:1989honest}
Ker-Chau Li.
\newblock Honest confidence regions for nonparametric regression.
\newblock \emph{The Annals of Statistics}, 17\penalty0 (3):\penalty0
  1001--1008, 1989.

\bibitem[Gin{\'e} and Nickl(2010)]{Gine:2010confidence}
Evarist Gin{\'e} and Richard Nickl.
\newblock Confidence bands in density estimation.
\newblock \emph{The Annals of Statistics}, 38\penalty0 (2):\penalty0
  1122--1170, 2010.

\bibitem[Hoffmann and Nickl(2011)]{Hoffmann:2011}
Marc Hoffmann and Richard Nickl.
\newblock On adaptive inference and confidence bands.
\newblock \emph{The Annals of Statistics}, 39\penalty0 (5):\penalty0
  2383--2409, 2011.

\bibitem[Benjamini and Hochberg(1995)]{Benjamini:1995controlling}
Yoav Benjamini and Yosef Hochberg.
\newblock Controlling the false discovery rate: a practical and powerful
  approach to multiple testing.
\newblock \emph{Journal of the Royal Statistical Society: series B
  (Methodological)}, 57\penalty0 (1):\penalty0 289--300, 1995.

\bibitem[Lehmann and Romano(2005)]{Lehmann:2005}
Erich~Leo Lehmann and Joseph~P Romano.
\newblock Generalizations of the familywise error rate.
\newblock \emph{The Annals of Statistics}, 33\penalty0 (3):\penalty0
  1138--1154, 2005.

\bibitem[Rozeboom(1960)]{Rozeboom:1960}
William~W Rozeboom.
\newblock The fallacy of the null-hypothesis significance test.
\newblock \emph{Psychological Bulletin}, 57\penalty0 (5):\penalty0 416, 1960.

\bibitem[Meehl(1967)]{Meehl:1967}
Paul~E Meehl.
\newblock Theory-testing in psychology and physics: A methodological paradox.
\newblock \emph{Philosophy of Science}, 34\penalty0 (2):\penalty0 103--115,
  1967.

\end{thebibliography}

%=============================================== =========================
% Appendix
%========================================================================
\appendix
Add somewhere:
\begin{remark}
Figure \ref{fig:ContMuProblem} shows that, if $\mu$ is not continuous or there is no set $\cA'\subseteq C(S)$
such that $\Gamma\big( \cA \big) = \Gamma\big( \cA' \big)$, then it is necessary in general to consider
the generalized preimage $\mathfrak{u}^{-1}_\cA$ instead of $\mu^{-1}_\cA$ in our main theorem.
However, in Appendix \ref{App:frakSeqS} we give a general condition for which $\mathfrak{u}^{-1}_\cA = \mu^{-1}_{\cA}$.
\end{remark}

\FloatBarrier
\newpage
%-------------------------------------------------------------------------------------------------------
%-------------------------------------------------------------------------------------------------------
\section{Additional Figures}
%-------------------------------------------------------------------------------------------------------
%-------------------------------------------------------------------------------------------------------
\FloatBarrier

\begin{figure}[b!]
			\includegraphics[width=0.49\textwidth]{\figurepath 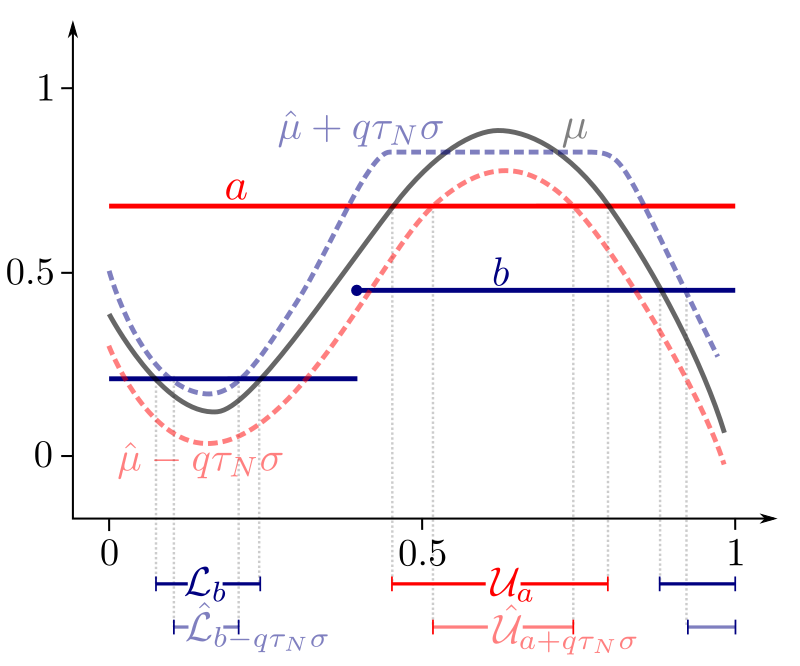}
			\includegraphics[width=0.49\textwidth]{\figurepath 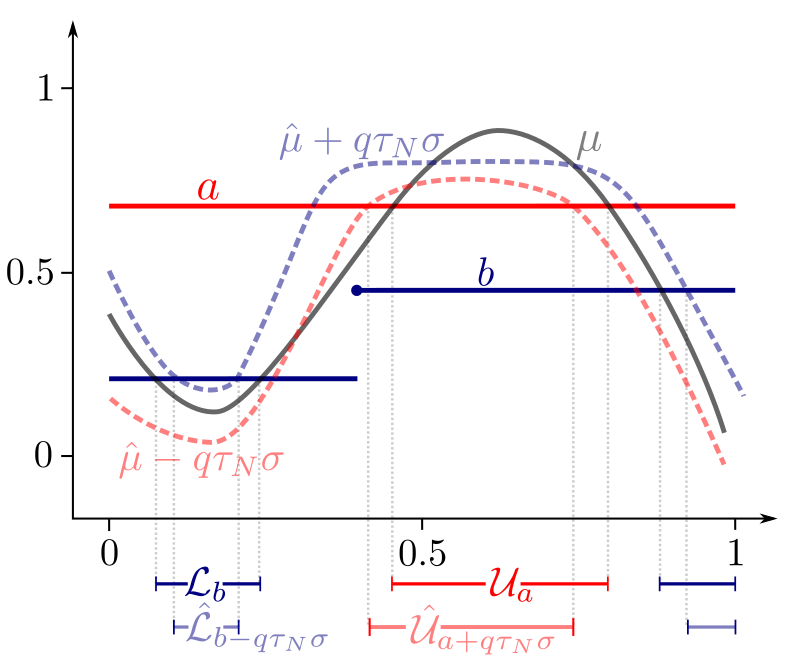}
		\caption{Illustration of SCoPE sets for $\cA = \{a\}$ and $\cB = \{b\}$ using the band given by
				 $\hat\mu_N \pm q\tau_N\sigma$. \textit{Left:} In contrast to a simultaneous confidence
				 band the true $\mu$ does not need to be inside the band $\hat\mu_N \pm q\tau_N\sigma$
				 everywhere. Only close to $\mu^{-1}_a$ and $\mu^{-1}_b$ it is necessary that
				 $\hat\mu_N - q\tau_N\sigma < \mu$ and $\hat\mu_N + q\tau_N\sigma > \mu$ respectively.
				 \textit{Right:} An example that the SCoRE set inclusion for $a$ is not satisfied because
				 $\hat\mu_N - q\tau_N\sigma > \mu$ in a neighbourhood of $\mu^{-1}_a$.	
		}\label{Fig:ScoPESconcept}
\end{figure}

%-------------------------------------------------------------------------------------------------------
\begin{figure}[b!]
			\includegraphics[width=0.49\textwidth]{\figurepath 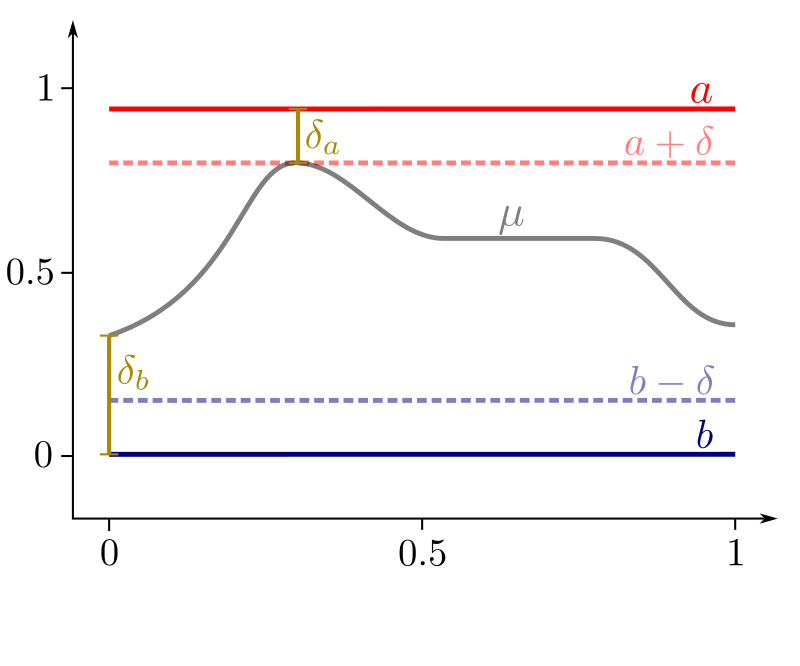}
			\includegraphics[width=0.49\textwidth]{\figurepath 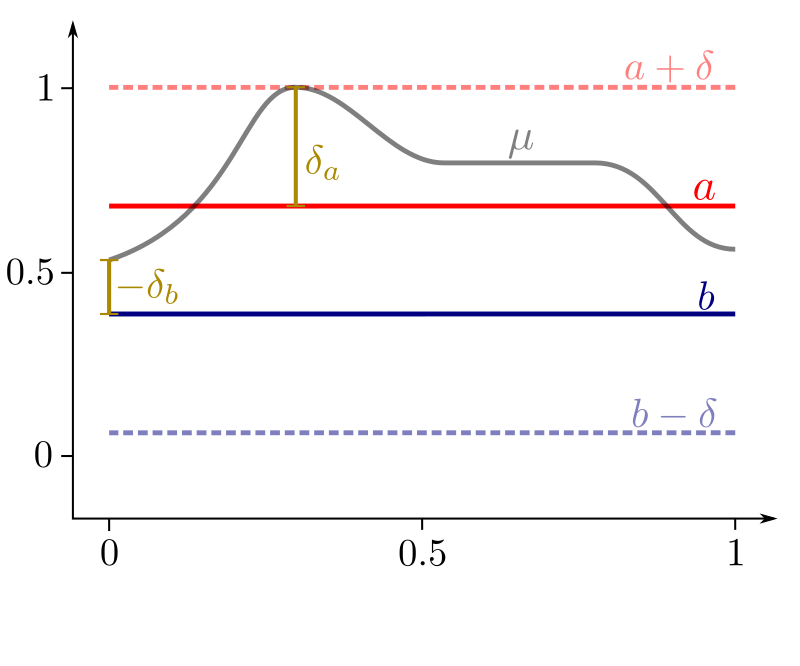}
		\caption{\textit{Left:} Illustration of obtaining the critical sets for the global relevance
		and the global relevance test. Right: global relevance hypothesis is true. Left:
		global equivalence hypothesis  is true.} \label{fig:RelEqivSimilarity}
\end{figure}

%\begin{figure}[b!]
%		\begin{center}
%			\includegraphics[width=5.5cm]{\figurepath RelevanceTest_Disc1.png}\quad\quad\quad\quad
%			\includegraphics[width=5.5cm]{\figurepath RelevanceTest_Disc2.png}
%		\end{center}
%		\caption{Illustration that the critical set
%		$\mathfrak{u}^{-1}_{ b_\Delta^- } \cup \mathfrak{u}^{-1}_{ b_\Delta^+ }$
%		can be either empty or non-empty.
%		\textit{Left:} The critical set
%		is empty because $\mathbf{H}_{1,s}^{rel} $ is true at the point $s$ with minimal
%		distance between $\mu$ and $b^-$ and $b^+$.
%		\textit{Right:} The critical set is non-empty because $\mathbf{H}_{0,s}^{rel} $
%		is true at the point $s$ with minimal distance between $\mu$ and $b^-$ and $b^+$.
%		\label{fig:EmptySC}}
%\end{figure}

\FloatBarrier
\newpage

%-------------------------------------------------------------------------------------------------------
\section{Auxiliary Lemmata}\label{Appendix:Auxillary}
%-------------------------------------------------------------------------------------------------------
\subsection{A Lemma on Inner Probability}
%-------------------------------------------------------------------------------------------------------
The following result should be well-known. We include it for completeness since we will use it often
in our proofs.
\begin{lemma}\label{lem:LowerBoundInnerProb}
	For all $A,B\subseteq \Omega$ it holds that
	\begin{equation}
	\begin{split}
		&\Prb_*(\, A \cap B\,) \geq \Prb_*(\, A \,) - \Prb^*(\, \Omega\setminus B \,)
				= \Prb_*(\, A \,) + \Prb_*(\, B \,) - 1
	\end{split}
	\end{equation}
\end{lemma}
%-------------------------------------------------------------------------------------------------------
\begin{proof}
	This follows from $ (A \cap B)_* = A_* \cap B_* $ (e.g., VW 1.2 Exc.15), since
	\begin{equation*}
		\Prb_*(\, A \cap B\,) = \Prb(\, A_* \cap B_*\,) \geq \Prb(\, A_* \,) + \Prb(\, B_* \,) - 1 = \Prb_*(\, A \,) + \Prb_*(\, B \,) - 1\,.
	\end{equation*}
	Note that for $A\subset\Omega$ the set $A_*\subseteq A$ is the (always existing) measurable
	set such that $\Prb_*(A) = \Prb(A_*)$. The claim follows from
	$ \Prb_*(\, B \,) = 1 - \Prb^*(\, \Omega \setminus B \,) $.
\end{proof}

%%-------------------------------------------------------------------------------------------------------
%\begin{lemma}\label{lem:Measurable}
%	Assume $c\in \mathcal{F}(S)$ be measurable, $\tau_N, q\in \mathbb{R}$ and the random field
%	\begin{equation}\label{eq:GNapp}
%		G_N(s) = \frac{\hat\mu_N(s) - c(s)}{\tau_N \sigma(s)}
%	\end{equation}
%	to be $\big( \mathcal{B}(S)\otimes \mathfrak{P}\big)-\mathcal{B}(\mathbb{R})$
%	measurable.
%	Then the events $ \hat{\cU}_{c^{+q}} \subseteq \cU_c $ and $ \mathcal{A}_c \subseteq \hat{\mathcal{A}}_{c_q^-} $
%	are measurable.
%\end{lemma}
%\begin{proof}
%	We only demonstrate the result for the event $ \hat{\cU}_{c^{+q}} \subseteq \cU_c $, which is equivalent to the
%	event $ \hat{\cU}_{c^{+q}} \cap \cU_c^\complement = \emptyset$.
%	The set
%	$ \big\{\, (s, \omega) \in S \times \Omega:~ G_N(s) \geq q,~ s\in \cU_c^\complement  \,\big\} $
%	is measurable, because $c$ being measurable implies that $ \cU_c^\complement $ is measurable and $G_N(s)$ is $\big( \mathcal{B}(S) \otimes \mathfrak{P}\big)-\mathcal{B}(\mathbb{R})$
%	measurable. The projection of this set onto $\Omega$ is equal to the set
%	$ \hat{\cU}_{c^{+q}} \cap \cU_c^\complement \neq \emptyset $ and therefore measurable by Theorem E.5,
%	c.f., \cite{Molchanov:2005}. The latter implies that its complement $ \hat{\cU}_{c^{+q}} \cap \cU_c^\complement = \emptyset $
%	is measurable.
%\end{proof}

%-------------------------------------------------------------------------------------------------------
\subsection{Inclusion Lemmas}
%-------------------------------------------------------------------------------------------------------
In this section we assume that $N\in\mathbb{N}$, $\hat\mu_N$ is an estimator of $\mu\in\ell^\infty(S)$,
$\tau_N>0$, $c\in\mathcal{F}(S)$, $ q, \sigma \in \ell^\infty(S) $ such that
$
0 < \mathfrak{o} \leq \sup_{s\in S} \sigma(s) \leq \mathfrak{O} < \infty
$
and we use the notation
$
	 G_N =  \frac{ \hat\mu_N - \mu }{ \tau_{N}\sigma}
$
from \eqref{eq:FiniteG}.
%-------------------------------------------------------------------------------------------------------
\begin{lemma}\label{lem:SCoPESlemma}
	Let $c\in\cF(S)$, $ {\tilde{\eta}} > \eta > 0 $ and
	assume there exists $ K > 0 $ and $ Z_N:~S\rightarrow \mathbb{R} $ such that for all
	$ s \in S \setminus \mu^{-1}_{\{c\}_{\tilde{\eta}} } $ it holds that
    \begin{equation}\label{eq:assumption0}
		\big(\, \hat\mu_N(s) - c(s) \,\big)
    	\cdot{\rm sgn}\big(\, \mu(s) - c(s) \,\big)
		\geq \sigma(s) \big(\, K  + Z_N(s) \,\big)\,.
	\end{equation}
	\begin{itemize}
		\item[(i)] Let $ c_q = c + q \sigma \tau_N $. Assume
			\begin{equation*}
			\begin{split}
  				\sup_{s\in \mu^{-1}_{\{c\}_{-\eta}}}  G_N(s) - &q(s) < 0\,,
  				\quad\quad\quad\quad
  				\sup_{ s \in \mu^{-1}_{\{c\}_{-\tilde\eta}} } G_N(s) - q(s) < \tfrac{\eta}{\tau_N\mathfrak{O}}\\
     			&\inf_{ s \in S } \tau_N^{-1}Z_N(s) + q(s) > - K\tau_N^{-1}\,,
   			\end{split}
			\end{equation*}
			then $ \hat{\cU}_{c_q} \subseteq \cU_c $
			and $ S \setminus \hat{\cL}_{c_q} \subseteq S \setminus \mathcal{\cL}_c $.
		\item[(ii)]  Let $ c_q = c - q \sigma \tau_N $. Assume
			\begin{equation*}
			\begin{split}
  				\inf_{s \in \mu^{-1}_{\{c\}_{+\eta} }}\!\!\!G_N(s) + &q(s) > 0\,,\quad\quad\quad\quad
						  \inf_{s \in \mu^{-1}_{\{c\}_{+\tilde\eta}}}\!\!\! G_N(s) + q(s)
						   > - \tfrac{\eta}{\tau_N\mathfrak{O}}\\
						  &\inf_{ s \in S }
						    \tau_N^{-1}Z_N(s) + q(s)
						   > - K\tau_N^{-1}\,,
						   \end{split}
			\end{equation*}
			then $ \hat{\cL}_{c_q} \subseteq \cL_c $
			and $ S \setminus \hat{\cU}_{c_q} \subseteq S \setminus \cU_c $.
	\end{itemize}
\end{lemma}
%-------------------------------------------------------------------------------------------------------
\begin{proof}
	Only the statements for open excursion sets are proven.
	The proofs for the closed excursion sets are almost identical.
	
	We begin with the proof of $(i)$. Note that
	$ \hat{\cU}_{c_q} \subseteq \cU_c $ is equivalent to
	$ S \setminus \cU_c \subseteq S \setminus \hat{\cU}_{c_q} $.
	 Let $ s_0 \in \big(S \setminus \cU_c\big) \cap \mu^{-1}_{\{c\}_{-\eta}}  $ which implies
	 $\mu(s_0) \leq c(s_0) $. This together with the first inequality of the assumptions yields
	\begin{equation*}
		\frac{\hat\mu_N(s_0) - c(s_0) }{\tau_N \sigma(s_0)}
						\leq G_N(s_0) < q(s_0)
	\end{equation*}
	showing
	$
	\big(S \setminus \cU_c\big) \cap \mu^{-1}_{\{c\}_{-\eta}}
			\subseteq S \setminus \hat{\cU}_{c_q}
	$.
	Similarly,
	$
	s_0 \in \big(S \setminus \cU_c\big)  \cap
			\mu^{-1}_{\{c\}_{-\tilde\eta}} \cap
			S \setminus \mu^{-1}_{\{c\}_{-\eta}}
	$
	satisfies $ c(s_0) - \mu(s_0) > \eta $ which is equivalent to $\mu(s_0) + \eta < c(s_0) $.
	Together with the second inequality of the assumptions this yields
	\begin{equation*}
		\frac{\hat\mu_N(s_0) - c(s_0)}{\tau_N \sigma(s_0)}
			\leq G_N(s_0) - \frac{\eta }{\sigma(s_0)\tau_N}
			\leq G_N(s_0) - \frac{\eta }{\mathfrak{O}\tau_N}
			 < q(s_0)
	\end{equation*}
	showing $ \big( S \setminus \cU_c \big)  \cap
			\mu^{-1}_{\{c\}_{-\tilde\eta}} \cap
			S \setminus \mu^{-1}_{\{c\}_{-\eta}} \subseteq S \setminus \hat{\cU}_{c_q} $.
	Finally,
	$ s_0 \in \big(S \setminus \cU_c\big) \cap S \setminus \mu^{-1}_{\{c\}_{-\tilde\eta}} $
	and
	\begin{equation*}
	  \sup_{ s \in S }
	   -\tau_N^{-1}Z_N(s) - q(s)
	   < K \tau_N^{-1} ~ ~ ~\Leftrightarrow  ~ ~ ~ \inf_{ s \in S }
	   \tau_N^{-1}Z_N(s) + q(s)
	   > - K \tau_N^{-1}
	\end{equation*}
	combined with \eqref{eq:assumption0} implies that
	\begin{equation*}
	\frac{ \hat\mu_N(s_0) - c(s_0) }{ \tau_N \sigma(s_0) }
	\leq
	- K \tau_N^{-1} - \tau_N^{-1} Z_N(s_0) < - K \tau_N^{-1}  + q(s_0) + K \tau_N^{-1}  = q(s_0)\,.
	\end{equation*}
	Thus,
	$ \big( S \setminus \cU_c \big) \cap S \setminus \mu^{-1}_{\{c\}_{-\tilde\eta}} \subseteq S \setminus \hat{\cU}_{c_q} $.
	Hence we proved
	$ S \setminus \cU_c \subseteq  S \setminus \hat{\cU}_{c_q} $.
	
	We prove $(ii)$ similarly. Again note that
	$ \hat{\cL}_{c_q} \subseteq \cL_c $ is equivalent to
	$ S \setminus \cL_c \subseteq S \setminus \hat{\cL}_{c_q} $.
	Assume $s_0 \in \big( S \setminus \cL_c \big) \cap \mu^{-1}_{\{c\}_{+\eta}}  $.
	Thus, $\mu(s_0) \geq c(s_0) $ and the first inequality of the assumptions yields
	\begin{equation*}
		\frac{\hat\mu_N(s_0) - c(s_0) }{\tau_N \sigma(s_0)}
		\geq G_N(s_0)
		 > -q(s_0)
	\end{equation*}
	which shows
	$ \big( S \setminus \cL_c \big) \cap \mu^{-1}_{\{c\}_{+\eta}} \subseteq S \setminus \hat{\cL}_{c_q} $.
	Similarly, $ s_0 \in  \big( S \setminus \cL_c \big) \cap \big( \mu^{-1}_{\{c\}_{+\tilde\eta}} \setminus  \mu^{-1}_{\{c\}_{+\eta}}  \big) $ implies $\mu(s_0) -\eta  > c(s_0) $. Hence, the second inequality of the assumptions yields
	\begin{equation*}
		\frac{\hat\mu_N(s_0) - c(s_0)}{\tau_N \sigma(s_0)}
			> G_N(s_0) + \frac{\eta}{\tau_N\mathfrak{O}}
			> -q(s_0)
	\end{equation*}
	showing
	$\big( S \setminus \cL_c \big) \cap \big( \mu^{-1}_{\{c\}_{+\tilde\eta}} \setminus  \mu^{-1}_{\{c\}_{+\eta}}  \big) \subseteq S \setminus \hat{\cL}_{c_q}$.
	Finally, for $ s_0 \in \big( S \setminus \cL_c \big) \cap  S \setminus \mu^{-1}_{\{c\}_{+\tilde\eta}} $
	it holds that
	\begin{equation*}
		\frac{ \mu_N(s_0) - c(s_0) }{ \tau_N \sigma(s_0) }
		\geq
		K \tau_N^{-1} + \tau_N^{-1} Z_N(s_0) > K \tau_N^{-1} -q(s_0) - K \tau_N^{-1}  = -q(s_0)\,.
	\end{equation*}
	Thus,
	$ \big( S \setminus \cL_c \big) \cap  S \setminus \mu^{-1}_{\{c\}_{+\tilde\eta}} \subseteq S \setminus \hat{\cL}_{c_q} $.
	Collecting the results yields
	$ S \setminus \cL_c \subseteq S \setminus \hat{\cL}_{c_q} $.
\end{proof}
%-------------------------------------------------------------------------------------------------------
\begin{lemma}\label{lem:CoPEContainLemma}
	Let $b,c\in\mathcal{F}(S)$, $\tilde\eta > 0$, $Z_N$ be defined as in
	Lemma \ref{lem:SCoPESlemma} and $c_q^\pm = c \pm q\tau_N\sigma$.
	Assume
	$
		\sup\big\{ G_N(s)~\vert~s \in \mu^{-1}_{\{c\}_{\tilde{\eta}}} \big\}$, 
		$\inf_{s \in S} \tau_N^{-1}Z_N(s)
	$	
	and $G_N(s)$, $s \in \mu^{-1}_{\{c\}_{\tilde{\eta}}}$, are asymptotically tight.
	
	If $\inf_{s\in S} \{ c(s) - b(s) \} > 0$, then
	\begin{equation*}
		\lim_{N\rightarrow\infty}\Prb_*\big[\,  \hat{\cU}_{c_q^+} \subseteq \cU_{b} \,\big]
								  		 = 1\,,~ ~ ~ ~
		\lim_{N\rightarrow\infty}\Prb^*\big[\, 
								   \hat{\cL}_{ c_q^-} \subseteq \cL_{b}  \,\big]
								  		 = 0\,,
	\end{equation*}
	where we additionally assume $ \cL_{c} \setminus \cL_{b} \neq \emptyset $ for the second equality.
	
	If $\sup_{s\in S} \{ c(s) - b(s) \} < 0$, it holds
	\begin{equation*}
		\lim_{N\rightarrow\infty}\Prb_*\big[\, 
								  \hat{\cL}_{c_q^-} \subseteq \cL_{b}  \,\big]
								  		 = 1\,,~ ~ ~ ~
		\lim_{N\rightarrow\infty}\Prb^*\big[\,  \hat{\cU}_{c_q^+} \subseteq \cU_{b} \,\big]
								  		 = 0\,,
	\end{equation*}
	where we additionally assume 
	$
	\cU_{c} \setminus \cU_{b} \neq \emptyset
	$ for the second equality.
\end{lemma}	
\begin{proof}
		We only proof the first claim as the proof of the second is similarly. The assumption $\delta > 0$ yields
		\begin{equation*}
		\begin{split}
			\Delta^N_s
			=     \frac{\mu(s) - c (s)}{\tau_N\sigma(s)}
			\leq -\frac{ \delta }{ \tau_N\mathfrak{O} }
		\end{split}
		\end{equation*}
		for all $s \in S \setminus \cU_b$.
		Using this and defining $\cU_b^c = S \setminus \cU_b$ we obtain
		\begin{equation*}
		\begin{split}
			&\Prb_*\Big[  \hat{\cU}_{c_q^+} \subseteq \cU_{b} \Big]\\
				&= \Prb_*\Bigg[
								\sup_{s \in \cU_b^c} G_N(s) + \Delta^N_s - q(s) \leq 0
							\Bigg] \\
				&\geq \Prb_*\Bigg[
						\sup_{s \in \cU_b^c \cap \mu^{-1}_{\{c\}_{\tilde{\eta}}}} \hspace{-0.5cm}G_N(s) - q(s)
							\leq \tfrac{\delta}{\tau_N\mathfrak{O}}~ \wedge
						\hspace{-0.2cm}\sup_{s \in \cU_b^c \cap S \setminus\mu^{-1}_{\{c\}_{\tilde{\eta}}}}
									 \hspace{-0.5cm}G_N(s) + \Delta^N_s - q(s) \leq 0
						\Bigg]\\
				&\geq \Prb_*\Bigg[
							\sup_{s \in \cU_b^c \cap \mu^{-1}_{\{c\}_{\tilde\eta}}} \hspace{-0.5cm}G_N(s) - q(s)
												\leq \tfrac{\delta}{\tau_N\mathfrak{O}} ~ \wedge \hspace{-0.1cm}
							\inf_{s \in \cU_b^c \cap S \setminus \mu^{-1}_{\{c\}_{\tilde{\eta}}}} \hspace{-0.5cm}\tfrac{K + Z_N(s)}{\tau_N} + q(s) \geq 0
							\Bigg]\\
				&\geq \Prb_*\Bigg[\,
						\sup_{s \in \mu^{-1}_{\{c\}_{\tilde{\eta}}}} G_N(s)
						\leq \tfrac{\delta}{\tau_N\mathfrak{O}} - \Vert q \Vert_\infty 
						\,\Bigg] + \Prb_*\Bigg[\,
							\inf_{s \in S} \tfrac{K + Z_N(s)}{\tau_N} \geq \Vert q \Vert_\infty
							\,\Bigg] - 1
		\end{split}
		\end{equation*}
		Applying the limes inferior to both sides implies
		$
		\liminf_{N\rightarrow\infty}\Prb_*\Big[\,
						\hat{\cU}_{c_q^+} \subseteq \cU_{b}
						\,\Big] = 1
		$
		by the asymptotic tightness assumption.
		
	Similarly, let $s^* \in \cL_{c} \setminus \cL_{b} \neq \emptyset$, then
	$
			\Delta^N(s^*)
			\leq - \frac{\delta}{\tau_N\mathfrak{O}}
	$.
	Therefore, if $s^*\in \mu^{-1}_{\{c\}_{\tilde{\eta}}}$,
	\begin{equation*}
	\begin{split}
			\Prb^*\Big[\, \hat{\cL}_{c_q^-} \subseteq \cL_{b}  \,\Big]
				&= \Prb^*\Big[  G_N(s^*) + \Delta^N_s \geq -q(s^*) \Big]\\
				&\leq 1 - \Prb_*\Big[  G_N(s^*) > -q(s^*) + \tfrac{\delta}{\tau_N\mathfrak{O}} \Big]
	\end{split}
	\end{equation*}
	and, if $s^*\in S \setminus \mu^{-1}_{\{c\}_{\tilde{\eta}}} $,
	\begin{equation*}
	\begin{split}
			\Prb^*\Big[ \hat{\cL}_{c_q^-} \subseteq \cL_{b}  \big]
				= \Prb^*\Big[  \tfrac{ \hat\mu_N(s^*) - c(s^*) }{\tau_N\sigma(s^*)} \geq -q(s^*) \Big]
				\leq 1 - \Prb_*\Big[  \tfrac{K + Z_N(s^*)}{\tau_N} > q(s^*) \Big]\,.
	\end{split}
	\end{equation*}
	In both cases applying the limes superior the r.h.s. converges to zero by the asymptotic
	tightness assumption.
\end{proof}

\begin{lemma}\label{lem:IntersectionCope}
	Let $a, b\in\mathcal{F}(S)$, $a_q = a - q\tau_N\sigma$,
	$b_q = b + q\tau_N\sigma$ and
	$\inf_{s \in S} \big(\, a(s) - b(s) \,\big) > \delta > 0 $.
	Assume that
	$\inf\big\{ G_N(s) ~\vert~ s \in \mu^{-1}_{\{b\}_{\tilde{\eta}}} \}$ and
	$\inf\big\{ \tau_N^{-1}Z_N(s) ~\vert~ s \in S \big\}$
	are asymptotically tight. Then 
	\begin{equation*}
	\begin{split}
		\limsup_{N\rightarrow\infty }
				\Prb^*\Big[\,
					S \setminus \big( \cU_{b} \cup \hat{\cL}_{a_q} \big) \neq \emptyset
				\,\Big]
				= \limsup_{N\rightarrow\infty }\Prb^*\Big[\,
								S \setminus \big( \cL_{a}
									\cup \hat{\cU}_{b_q}  \big) \neq \emptyset
								\,\Big] = 0\,.
	\end{split}
	\end{equation*}
\end{lemma}
\begin{proof}
	We only show one of the two claims as the proofs are similar. Define $\cL_{a}^c = S \setminus \cL_{a}$, then
	\begin{equation*}
	\begin{split}
		&\Prb^*\Big[ S \setminus \big( \cL_{a}
									\cup \hat{\cU}_{b_q}  \big) \neq \emptyset
				\Big]\\
		&\leq
		\Prb^*\Bigg[
			\inf_{s \in \cL_{a}^c}  \tfrac{\hat\mu_N(s) - b(s)}{\tau_N\sigma(s)} - q(s) \leq 0
			 \Bigg]\\
		&\leq
		\Prb^*\Bigg[
				\inf_{s \in \mu^{-1}_{\{b\}_{\tilde{\eta}}} \cap \cL_{a}^c } 
														G_N(s) - q(s) \leq - \tfrac{\delta}{\tau_N\mathfrak{o}}
				\Bigg]+
		\Prb^*\Bigg[
			\inf_{s \in S} \tfrac{K + Z_N(s)}{\tau_N} - q(s) \leq 0
		\Bigg]\\
		&\leq
		\Prb^*\Bigg[
				\inf_{s \in \mu^{-1}_{\{b\}_{\tilde{\eta}}} }
							G_N(s) - q(s) \leq- \tfrac{\delta}{\tau_N\mathfrak{o}}
					\Bigg]
		+\Prb^*\Bigg[\,
				\inf_{s \in S} \tfrac{K + Z_N(s)}{\tau_N} - q(s) \leq 0
				\,\Bigg]\\
				&\leq 2 -
		\Prb_*\Bigg[
				\inf_{s \in \mu^{-1}_{\{b\}_{\tilde{\eta}}} }
							G_N(s) > \Vert q \Vert_\infty - \tfrac{\delta}{\tau_N\mathfrak{o}}
					\Bigg] -
		\Prb_*\Bigg[
				\inf_{s \in S} \tfrac{Z_N(s)}{\tau_N} > \Vert q \Vert_\infty - \tau_N K
				\Bigg]
	\end{split}
	\end{equation*}	
	Applying the limes superior to both sides implies that the r.h.s. converges
	to zero as $N$ tends to infinity by the asymptotical tightness assumption.
\end{proof}
%-------------------------------------------------------------------------------------------------------

%-------------------------------------------------------------------------------------------------------
\subsection{Lemmas on Uniform Convergence}\label{App:WeakMaxConvergence}
%-------------------------------------------------------------------------------------------------------
This appendix collects some facts about uniform convergence and supremum statistics. 
%-------------------------------------------------------------------------------------------------------
\begin{lemma}
\label{lem:ConvergenceMaxima}
    Let $ \big(\mathcal{A}_N\big)_{ N \in \mathbb{N} } \subseteq S$ and
	$\mathcal{A}, \mathcal{B}\subset S $, $ (\varepsilon_N)_{ N \in \mathbb{N} } $ converging to zero such
	that $ d_H(\mathcal{A}_N, \mathcal{A})\leq \varepsilon_N $.
    Let $f,g\in \mathcal{F}(S)$. 
    \begin{equation*}
    \begin{split}
        &(i)~~\Big\vert \sup_{s\in \mathcal{A}_N}f(s) - \sup_{t \in \mathcal{A}}f(t)\Big\vert\\
        		&\hspace{2cm}\leq \max\Big( 0,\, \sup \Big\{ \vert f(s)-f(t) \vert ~\Big\vert~ (s,t) \in \big( \mathcal{A}_N\setminus\mathcal{A} \big) \times \big(  \mathcal{A} \setminus {\rm int}(\mathcal{A}) \big):~d(s,t)\leq \varepsilon_N \Big\}\Big)\\
    	&(ii)~~\Big\vert  \sup_{s\in \mathcal{B}}f(s)
          			-\sup_{t\in \mathcal{B}}g(t) \Big\vert
          			\leq   \sup_{s\in \mathcal{B}} \Big\vert f(s)
          			-g(t) \Big\vert
    \end{split}
    \end{equation*}
    Assume additionally $\mathcal{A}_N \supseteq \mathcal{A}_{N+1} \supseteq \mathcal{A}$
    for all $N\in \mathbb{N}$ and
    $ ( f_N )_{ N \in \mathbb{N} } \subset \ell^\infty(\mathcal{A}_1) $ such that
	\begin{equation*}
		\sup_{s\in \mathcal{A}_1} \vert f_N(s) - f(s) \vert \xrightarrow{N \rightarrow \infty} 0
	\end{equation*}	  
	and $f$ uniform continuous on
    $\mathcal{A}_1 \setminus {\rm int}\big(\mathcal{A}\big)$. Then
	    \begin{equation*}
    \begin{split}
    	&\hspace{-7.5cm}(iii)~~\Big\vert \sup_{s\in \mathcal{A}_N}f(s) - \sup_{t\in \mathcal{A}}f(t)\Big\vert
    							  \xrightarrow{N \rightarrow \infty} 0\\
        &\hspace{-7.5cm}(iv)~~\Big\vert  \sup_{s\in \mathcal{A}_N}f_N(s)
          			-\sup_{t\in \mathcal{A}}f(t) \Big\vert  \xrightarrow{N \rightarrow \infty} 0
    \end{split}
    \end{equation*}
\end{lemma}
%-------------------------------------------------------------------------------------------------------
\begin{proof}
For statement (i), fix $ \epsilon > 0 $ and $ N > 0 $, without loss of
generality, assume that $ \sup_{s\in \mathcal{A}_N}f(s)\geq\sup_{t\in \mathcal{A}}f(t) $.
Let $ s^* \in \mathcal{A}_N $ be such that $ \sup_{s \in \mathcal{A}_N}f(s) = f(s^*) + \epsilon $
and $ t^* \in \mathcal{A} $ such that $ d(t^*, s^*) < \varepsilon_N $
(exists since $d_H(\mathcal{A}_N, \mathcal{A})< \varepsilon_N$). Then
\begin{equation*}
\begin{split}
	\sup_{s\in \mathcal{A}_N}f(s)-\sup_{t\in \mathcal{A}}f(t)
	&\leq f(s^*)-f(t^*)+\epsilon\\
	&\leq \sup\Big\{ \vert\, f(s)-f(t) \,\vert ~\big\vert~
	(s,t) \in \mathcal{A}_N \times \mathcal{A}:~d(s,t)\leq \varepsilon_N  \Big\} + \epsilon
\end{split}
\end{equation*}
Since $ \epsilon $ can be arbitrarily small, 
\begin{equation*}
	\Big\vert \sup_{s\in \mathcal{A}_N}f(s) - \sup_{t\in \mathcal{A}}f(t) \Big\vert
	\leq \sup\Big\{ \vert\, f(s)-f(t) \,\vert ~\big\vert~
	(s,t) \in \mathcal{A}_N \times \mathcal{A}:~d(s,t)\leq \varepsilon_N  \Big\}\,.
\end{equation*} 

%-------------------------------------------------------------------------------------------------------
For statement (ii), 
\begin{align*}
   & \sup_{s\in \mathcal{B}}f_N(s) = \sup_{s\in \mathcal{B}}\big( f_N(s) - f(s) + f(s) \big)
   				\leq \sup_{s\in \mathcal{B}}\big(f_N(s)-f(s)\big)+\sup_{ s \in \mathcal{B}} f(s) \\
   \Longleftrightarrow ~~~& \sup_{s\in \mathcal{B}}f_N(s) -\sup_{s\in \mathcal{B}} f(s)
   \leq \sup_{s\in \mathcal{B}}\big( f_N(s) - f(s) \big)\leq \sup_{s\in \mathcal{B}} \vert f_N(s) - f(s) \vert
\end{align*}
Using  $\sup_{s\in \mathcal{B}}f(s)$ a similar calculation yields
\begin{equation*}
	\sup_{ s \in \mathcal{B} } f_N(s) - \sup_{ s \in \mathcal{B}} f(s) \geq -\sup_{s\in \mathcal{B}} \vert f_N(s) - f(s) \vert\,,
\end{equation*}
which proves the claim.

%-------------------------------------------------------------------------------------------------------
For statement (iii), note that since $\mathcal{A}_N\supset \mathcal{A}$ it is possible
to replace $s \in \mathcal{A}_N$ by $s\in \mathcal{A}_N \setminus \mathcal{A}$ and $t\in\mathcal{A}$ by $ t\in \mathcal{A} \setminus {\rm int}(\mathcal{A})$ in the supremum on the
r.h.s. of $(i)$, i.e.,
\begin{equation*}
\begin{split}
	\Big\vert \sup_{s\in \mathcal{A}_N}&f(s) - \sup_{t\in \mathcal{A}}f(t)\Big\vert\\
        		&\leq \max\Big( \sup\Big\{ \vert\, f(s)-f(t) \,\vert ~\big\vert~
	(s,t) \in \big( \mathcal{A}_N\setminus\mathcal{A} \big) \times \big(  \mathcal{A} \setminus {\rm int}(\mathcal{A}) \big):~d(s,t)\leq \varepsilon_N  \Big\}, 0 \Big)
\end{split}
\end{equation*}
Since $f$ is uniformly continuous on
$
\mathcal{A}_N \setminus {\rm int}\big( \mathcal{A} \,\big) \subseteq \mathcal{A}_1 \setminus {\rm int}\big( \mathcal{A} \,\big)
$,
r.h.s converges to zero as $\varepsilon_N\rightarrow 0$.
%-------------------------------------------------------------------------------------------------------
Statement (iv) follows directly from the triangle inequality and $(ii)$ and $(iii)$.
\end{proof}

%-------------------------------------------------------------------------------------------------------
\begin{lemma} \label{lem:Tweak}
	Let $ \mathcal{A}, \mathcal{B} \subseteq S $ and
	$(\mathcal{A}_N)_{N\in\mathbb{N}}, (\mathcal{B}_N)_{N\in\mathbb{N}}$
	be sequences of sets such that
	$\mathcal{A}_N \supseteq \mathcal{A}_{N+1} \supseteq \mathcal{A}$ and
	$\mathcal{B}_N \supseteq \mathcal{B}_{N+1} \supseteq \mathcal{B}$ for all $N \in \mathbb{N}$
	as well as 
    $ d_H(\mathcal{A}_N,  \mathcal{A}) \rightarrow 0 $ and $ d_H(\mathcal{B}_N,  \mathcal{B}) \rightarrow 0 $
	as $N \rightarrow \infty$.
    Assume $ G_{N} \rightsquigarrow G $ in $ \ell^\infty(\mathcal{A}_1 \cup \mathcal{B}_1) $
    with $G$ Borel measurable and the restriction of $G$ and $G_N$ to
    $
    \mathcal{A}_1 \setminus {\rm int}\big( \mathcal{A} \,\big)
    \cup
    \mathcal{B}_1 \setminus {\rm int}\big( \mathcal{B} \,\big)
    $
    have uniformly continuous sample paths. Then
    $$
    \max\left(\,\sup _{s \in \mathcal{A}_N} -G_{N}(s),~ \sup _{s \in \mathcal{B}_N} G_{N}(s)\,\right)
    		\rightsquigarrow
    \max\left(\,\sup _{s \in \mathcal{A}} -G(s),~ \sup _{s \in \mathcal{B}} G(s)\,\right)
    $$
\end{lemma}
%-------------------------------------------------------------------------------------------------------
\begin{proof}
    This is a consequence of the extended continuous mapping theorem \cite[Theorem 1.11.1]{Vaart:1996weak}
    and Lemma \ref{lem:ConvergenceMaxima}(iv). More precisely, define the maps
    $$
    H: \ell^\infty(\mathcal{A}_1 \cup \mathcal{B}_1) \rightarrow \mathbb{R}, \quad
    	f \mapsto \max\Bigg( \sup _{s \in \mathcal{A}} -f(s), \sup _{s \in \mathcal{B}} f(s) \Bigg)
    $$
    $$
    H_N: \ell^\infty(\mathcal{A}_1 \cup \mathcal{B}_1) \rightarrow \mathbb{R}, \quad
     f \mapsto  \max\Bigg( \sup _{s \in \mathcal{A}_N} -f(s), \sup _{s \in \mathcal{B}_N} f(s) \Bigg)
    $$

   The claim follows from the extended continuous mapping theorem, if for any sequence
   $(f_N)_{N\in\mathbb{N}} \subset \ell^\infty(\mathcal{A}_1 \cup \mathcal{B}_1)$ converging to
   $f$ in $\ell^\infty(\mathcal{A}_1 \cup \mathcal{B}_1)$ such that the restrictions of
   $f_N$ and $f$ to
   $
   \mathcal{A}_1 \setminus {\rm int}\big( \mathcal{A} \,\big)
   \cup
   \mathcal{B}_1 \setminus {\rm int}\big( \mathcal{B} \,\big)
   $ are uniformly continuous, it holds that $H_{N}(f_N) \rightarrow H(f)$.
   Using
   $
   		\max(a,b) = 2^{-1}\big(\, a + b + \vert a - b \vert \,\big)
   $ for $a,b\in\mathbb{R}$ and triangle inequalities yields
	\begin{equation*}
	   \big\vert H_{N}(f_N) - H(f) \big\vert
	   \leq \Big\vert \sup _{s \in \mathcal{A}_N} -f_N(s) - \sup _{s \in \mathcal{A}} -f(s) \Big\vert
	   			+ \Big\vert \sup _{s \in \mathcal{B}_N} f_N(s) - \sup _{s \in \mathcal{B}} f(s) \Big\vert\,,
	\end{equation*}      
   which converges to zero by Lemma \ref{lem:ConvergenceMaxima}(iv). Note that we do not need to assume
   $G$ to be separable since our $\mathbb{D}_N$ in \cite[Theorem 1.11.1]{Vaart:1996weak} is always
   the subspace of $\ell^\infty(\mathcal{A}_1 \cup \mathcal{B}_1)$ where the restriction to
   $
   \mathcal{A}_1 \setminus {\rm int}\big( \mathcal{A} \,\big)
   \cup
   \mathcal{B}_1 \setminus {\rm int}\big( \mathcal{B} \,\big)
   $ are uniformly continuous \cite[Problems and Complements 1., p.70]{Vaart:1996weak}.
\end{proof}
%-------------------------------------------------------------------------------------------------------

\FloatBarrier
%-------------------------------------------------------------------------------------------------------
%\section{Additional Figures Explaining Introduced Concepts}\label{App:figures}
%-------------------------------------------------------------------------------------------------------

%\begin{figure}[b!]
%		\begin{center}
%			\includegraphics[width=0.4\textwidth]{\figurepath CoPE_DefinitionSC3.png}
%		\end{center}
%		\caption{Illustration of the definition of the sets $\Gamma(\mathcal{B})$
%		and $f^{-1}_\mathcal{B}$ for $S$ being one dimensional.
%		Moreover, we show the graph of an $\eta$-thickening of
%		$\mathcal{B}$ and the resulting set$\mu^{-1}_{\mathcal{B}_\eta}$. }\label{fig:Definitions}
%\end{figure}

 %-------------------------------------------------------------------------------------------------------

\FloatBarrier
\newpage

%-------------------------------------------------------------------------------------------------------
%-------------------------------------------------------------------------------------------------------
\section{The Asymptotic SCoRE Set Metatheorem}\label{scn:MetaTheorem}
%-------------------------------------------------------------------------------------------------------
%-------------------------------------------------------------------------------------------------------
In this section we prove a general SCoRE Set Metatheorem. All theorems of the main manuscript are
corollaries of this result.
Its main benefit is that it has weaker assumptions on $\hat{\mu}_N$ than assuming a ULT.

As in the main manuscript, let $(\Omega,\mathfrak{P}, \Prb)$ be a
probability space, $(S, d)$ be a metric space and $\cA, \cB \subseteq \mathcal{F}(S)$.
Let $\mu \in \ell^\infty(S)$ and $\sigma \in\ell^\infty(S)$ such that
$0 < \mathfrak{o} < \sigma(s) < \mathfrak{O}$ for all $s \in S$.
Let $ \hat\mu_N:~\Omega \rightarrow \ell^\infty(S)$, $ N \in \mathbb{N}$,
and $(\tau_N)_{N\in\mathbb{N}}$ a positive sequence converging to zero.
Define a map $G_N:~ \Omega \rightarrow \ell^\infty(S) $ by
\begin{equation*}
	G_N = \frac{\hat\mu_N - \mu}{\tau_N\sigma}\,.
\end{equation*}
Since we do not assume that the map $G_N$ is measurable, weak convergence involving $G_N$
is understood in the sense of \cite[Definition 1.3.3]{Vaart:1996weak}.
For $\cC \subseteq \mathcal{F}(S)$ we define the sets $	U^{\pm\eta}_\cC = {\rm cl}\mu^{-1}_{\cC_{\pm\eta}}$.
Moreover, we define the set where $G_N$ has continuous sample paths from the inside of $\mu^{-1}_\cC$ in a neighbourhood around $s\in S$ by
\begin{equation*}
\Omega_{ct}^s = \left\{\omega\in\Omega ~\vert~
	\exists \text{ open } V \ni s :~G_N\vert_{V\cap {\rm cl}\mu^{-1}_\cC}(\omega) \in C\big( V\cap {\rm cl}\mu^{-1}_\cC \big)  \right\}
\end{equation*}
and for any $W \subseteq \partial \mu^{-1}_\cC \setminus \mu^{-1}_\cC$ we define
$\Omega_{ct}^W =	\bigcap_{s\in W} \Omega_{ct}^s$ the subset of $\Omega$ where $G_N$ has continuous sample paths on $W$.

For $A, B \subseteq S$, $f, q^\pm\in\ell^\infty(S)$ and $q=\{q^+,q^-\}$, we define
\begin{equation*}
\begin{split}
	T^{q}_{A, B}( f )
	&= \max\Bigg( \sup_{s \in A } f(s) - q^+(s)
				,~ \sup_{s \in B} -f(s)- q^-(s) \Bigg)\,.
\end{split}
\end{equation*}

Let $\mathfrak{G}^{q}$ and $\mathfrak{H}^{q}$ depending on $q$ be random variables with
values in $\mathbb{R}$, $\tilde\eta > 0$ and $(\eta_N)_{N\in\mathbb{N}}$ a positive sequence
such that $\lim_{N\rightarrow\infty}\eta_N\tau_N^{-1} = \infty$.
We require the following assumptions:
\begin{itemize}[leftmargin=1.4cm]
	\item[\textbf{(M1)}] 
			Assume $T^{q}_{U^{-\eta_N}_\cA, U^{+\eta_N}_\cB}( G_N )
		 							\rightsquigarrow \mathfrak{G}^{q}$
		weakly in $\mathbb{R}$.
	\item[\textbf{(M2)}] The sequence
						$
						T^{q}_{U^{-\tilde\eta}_{\cA}, U^{+\tilde\eta}_{\cB}}( G_N )
						$
						is asymptotically tight in the sense of \cite[p.21]{Vaart:1996weak}.
	\item[\textbf{(M3)}] Assume there is a constant $ K > 0 $ such that
			    \begin{equation*}
					\big( \hat\mu_N(s) - c(s) \big)
			    	\cdot{\rm sgn}\big( \mu(s) - c(s)\big)
					\geq \sigma(s) \big( K  + Z_N(s) \big)
				\end{equation*}
    	for all $ c \in \cA \cup  \cB$
    	and all
    	$ s \in S \setminus \big( \mu^{-1}_{\cA_{-\tilde{\eta}}} \cup \mu^{-1}_{\cB_{+\tilde{\eta}}} \big) $
    	and an sequence $ Z_N:~\Omega \rightarrow \ell^\infty(S) $
    	such that $ \inf_{s \in S } \tau_N^{-1}Z_N(s) $ is asymptotically tight.
    	\item[\textbf{(M4)}] Let $W_\cA \subseteq \partial \mu^{-1}_\cA \setminus \mu^{-1}_\cA$
    						 and $W_\cB \subseteq \partial \mu^{-1}_\cB \setminus \mu^{-1}_\cB$
    						 with $\Prb^*\big( \Omega_{ct}^{W_\cA} \big) = \Prb^*\big( \Omega_{ct}^{W_\cB} \big) = 1$
    						 and assume that
    			$T^{q}_{\mu^{-1}_{\cA} \cup W_{\cA},\mu^{-1}_{\cB} \cup W_{\cB}}( G_N )
		 							\rightsquigarrow \mathfrak{H}^{q}$
				weakly in $\mathbb{R}$.
\end{itemize}
With this at hand, we can prove our SCoRE Set Metatheorem.
\begin{theorem}\label{thm:SCoPESMetatheorem}
	Let $ \cA, \cB \subseteq \mathcal{F}(S) $ and $ q^\pm\in\ell^\infty(S) $
	be arbitrary.
	Define
	$a_q = a + \tau_Nq^+\sigma$ for all $a \in \Gamma\big( \cA \big)$
	and $b_q = b - \tau_Nq^-\sigma$ for all $b \in \Gamma\big( \cB \big)$.
	\begin{enumerate}
		\item 	Assume
				$\mathfrak{u}^{-}_{\cA} \cup \mathfrak{u}^{+}_{\cB} \neq \emptyset$
				and {\rm\textbf{(M1)}}, {\rm\textbf{(M2)}} and {\rm\textbf{(M3)}}. Then
				\begin{equation*}
					\begin{split}
						\liminf_{N\rightarrow\infty} \Prb_*\Bigg[\,
							\forall a \in \Gamma\big( \cA \big)\, \forall b \in \Gamma\big( \cB \big):~&
						 \hat{\cU}_{a_q} \subseteq \cU_{a}
						 ~\wedge ~
						 \hat{\cL}_{b_q} \subseteq \cL_{b}
						  \,\Bigg]
						  \geq \Prb\Big[\, \mathfrak{G}^{q} < 0 \,\Big]\,.
					\end{split}
				\end{equation*}
		\item 	Assume
				$\mathfrak{u}^{-}_{\cA} \cup \mathfrak{u}^{+}_{\cB} = \emptyset$, {\rm\textbf{(M2)}} and {\rm\textbf{(M3)}}. Then
				\begin{equation*}
					\begin{split}
						\lim_{N\rightarrow\infty} \Prb_*\Bigg[\,
							\forall a \in \Gamma\big( \cA \big)\, \forall b \in \Gamma\big( \cB \big):~
								\hat{\cU}_{a_q} \subseteq \cU_{a}
									~\wedge ~
						 		\hat{\cL}_{b_q} \subseteq \cL_{b}
						  \,\Bigg]
						 = 1\,.
					\end{split}
				\end{equation*}
		\item 	Assume
				$\mathfrak{u}^{-}_{\cA}
					\cup \mathfrak{u}^{+}_{\cB} \neq \emptyset$. Then under {\rm\textbf{(M4)}} it holds that 
				\begin{equation*}
				\begin{split}
					\limsup_{N\rightarrow\infty} \Prb^*\Bigg[\,
						\forall a \in \Gamma\big( \cA \big)\, \forall b \in \Gamma\big( \cB \big):~&
								\hat{\cU}_{a_q} \subseteq \cU_{a}
									~\wedge ~
						 		\hat{\cL}_{b_q} \subseteq \cL_{b}
						  \,\Bigg] \leq \Prb\Big[\, \mathfrak{H}^{q} \leq 0 \,\Big]\,.
				\end{split}
				\end{equation*}
				\end{enumerate}
\end{theorem}
%-------------------------------------------------------------------------------------------------------
\begin{proof}
	
	%---------------------------------------------------------------------------------------------------
	% Proof of 1.
	We begin with proving the first and the second claim.
	For any $ \eta > 0 $ such that $\tilde{\eta} >\eta$, Lemma \ref{lem:SCoPESlemma}(i) yields that
	\begin{equation}\label{eq:Lemma1iiCons}
	\begin{split}
		\sup_{s \in {\rm cl}\,\mu^{-1}_{\cA_{-\eta}}} \hspace{-0.2cm} G_N(s) &-q^+(s) < 0\,,
  		\quad
  		\sup_{ s \in {\rm cl}\,\mu^{-1}_{\cA_{-\tilde{\eta}}}} \hspace{-0.2cm} G_N(s)
   				-q^+(s) < \mathfrak{O}\eta \tau_N^{-1}\\
   		&\inf_{ s \in S } \tau_N^{-1}Z_N(s) + q^+(s) > - K \tau_N^{-1}\,,
    \end{split}
	\end{equation}
	implies $ \hat{\cU}_{a_q } \subseteq \cU_{a} $
	for all $ a \in \Gamma( \cA) $ and Lemma \ref{lem:SCoPESlemma}(ii) shows that
	\begin{equation}\label{eq:Lemma1iCons}
	\begin{split}
		\sup_{s \in {\rm cl}\,\mu^{-1}_{\cB_{+\eta}}} \hspace{-0.2cm} -G_N(s) &- q^-(s) < 0\,,\quad
  		\sup_{ s \in {\rm cl}\,\mu^{-1}_{\cB_{+\tilde{\eta}}}} \hspace{-0.2cm} -G_N(s) - q^-(s)
   				< \eta \tau_N^{-1}\,,\\
     		&\inf_{ s \in S } \tau_N^{-1}Z_N(s) + q^-(s) > - K \tau_N^{-1}\,,
    \end{split}
	\end{equation}
	implies
	$\hat{\cL}_{b_q } \subseteq {\cL}_{b}$  for all $ b \in \Gamma( \cB) $.
	
	Let $\left( \eta_{N} \right)_{ N \in \mathbb{N} } $
	be a sequence of positive numbers such that $ \eta_{N} \rightarrow 0 $ and
	$\eta_{N} \tau_{N}^{-1} \rightarrow \infty$. Combining Lemma \ref{lem:LowerBoundInnerProb},
	\eqref{eq:Lemma1iCons} and \eqref{eq:Lemma1iiCons} yields
	\begin{equation*}
	\begin{split}
		&\Prb_* \Big[\,
					\forall a \in \Gamma\big( \cA \big)\, \forall b \in \Gamma\big( \cB \big):~
					 	\hat{\cU}_{a_q} \subseteq \cU_{a}
					 ~ ~ \wedge ~ ~
						 \hat{\cL}_{b_q} \subseteq \cL_{b}  \,\Big]\\
	 		&\geq
	 		\Prb_*\Big[\, T^{q}_{U^{-\eta_N}_{\cA}, U^{+\eta_N}_{\cB}}( G_N ) < 0 \,\Big]  -
	 		  \Prb^*\left[\,
	 		  	T^{q}_{U^{-\tilde\eta}_{\cA}, U^{+\tilde\eta}_{\cB}}( G_N )
	 \geq \mathfrak{O}\eta_N \tau_N^{-1}\,\right]\\
	  		&\hphantom{\geq} -
	  		\Prb^*\left[\,
		  		\inf_{ s \in S } \tau_N^{-1}Z_N(s)
	   			< -\inf_{ s \in S } \max\big( q^-(s), q^+(s) \big) - K \tau_N^{-1}
	   		\,\right]\\
	 		&\geq
	 	\Prb_*\Big[\,  T^{q}_{U^{-\eta_N}_{\cA}, U^{+\eta_N}_{\cB}}( G_N ) < 0 \,\Big]  -
	 		  \Prb^*\left[\,
	 		  	T^{q}_{U^{-\tilde\eta}_{\cA}, U^{+\tilde\eta}_{\cB}}( G_N )
	 \geq  \mathfrak{O}\eta_N \tau_N^{-1}\,\right]\\
	  		&\hphantom{\geq} -
	  		\Prb^*\left[\,
		  		\inf_{ s \in S } \tau_N^{-1}Z_N(s)
	   			< -\inf_{ s \in S } \max\big( q^-(s), q^+(s) \big) - K \tau_N^{-1}
	   		\,\right].
	 	\end{split}
		\end{equation*}
	The last two summands can be made arbitrarily small because the asymptotically tightness
	condition implies that for any $\epsilon>0$ we find $\tilde K>0$ such that for all $\delta>0$
	\begin{equation*}
	\begin{split}
				\liminf_{N\rightarrow\infty}-\Prb^*&\left[\,
	 		  		\inf_{ s \in S } \tau_N^{-1}Z_N(s)
	   				< -\tilde K - \delta
	 		  	\,\right]\\
	 		  	&= 
				-1 + \liminf_{N\rightarrow\infty}\Prb_*\left[\,
\inf_{ s \in S } \tau_N^{-1}Z_N(s)
	   				> -\tilde K - \delta
	 		  	\,\right]\\
	 		  	&\geq 
				-1 + \liminf_{N\rightarrow\infty}\Prb_*\left[\,
\inf_{ s \in S } \tau_N^{-1}Z_N(s) \in 
	   				[-\tilde K - \delta, \tilde K + \delta]
	 		  	\,\right]\\
	 		  	&\geq \epsilon\,.
	 \end{split}
	\end{equation*}
	Thus, setting  $\tilde K = \inf_{ s \in S } \max\big( q^-(s), q^+(s) \big)$
	we see that we can bound the last term by $\epsilon$.
	The same argument applies to the second summand.	
	
	If $\mathfrak{u}^{-}_{\cA} \cup \mathfrak{u}^{+}_{\cB} \neq \emptyset$, then
	the Portmanteau Theorem \cite[1.3.4]{Vaart:1996weak} and {\rm \textbf{(M1)}} yields		
	\begin{equation*}
		\begin{split}
			\liminf_{ N \rightarrow \infty }\, &\Prb_*\left[\,
			T^{q}_{U^{-\eta_N}_{\cA}, U^{+\eta_N}_{\cB}}( G_N ) < 0 \,\right]
			\geq \Prb\big[\, \mathfrak{G}_{q} < 0 \,\big]\,,
	 \end{split}
	\end{equation*}
	which proves the first claim.
	
	If $\mathfrak{u}^{-}_{\cA} \cup \mathfrak{u}^{+}_{\cB} = \emptyset$, then there is an $\eta>0$ such that
	${\rm cl}\,\mu^{-1}_{\cA_{-\eta}} \cup {\rm cl}\,\mu^{-1}_{\cB_{+\eta}} = \emptyset$
	for some $\eta >0$, then
	$
		T^{q}_{U^{-\eta_N}_{\cA}, U^{+\eta_N}_{\cB}}( G_N ) = -\infty
	$
	for all $N$ large enough. Thus, 
	\begin{equation*}
		\begin{split}
			\liminf_{ N \rightarrow \infty }\, &\Prb_*\left[\,
			T^{q}_{U^{-\eta_N}_{\cA}, U^{+\eta_N}_{\cB}}( G_N ) < 0 \,\right]
			\geq 1\,,
	 \end{split}
	\end{equation*}
	which finishes the proof of the second claim.
	%---------------------------------------------------------------------------------------------------
	% Proof of 3.
	
	In order to prove the third claim, we first prove that on $\Omega_{ct}^\cA \cap \Omega_{ct}^\cB$ we have that
	\begin{equation*}
		\begin{split}
			&\textbf{(E1)}~ \forall a \in \Gamma(\cA)\,
						   \forall b \in \Gamma(\cB)\!:~
			 	\hat{\cU}_{a_q } \subseteq \cU_{a} 
			 		~ \wedge ~
			 	\hat{\cL}_{b_q} \subseteq \cL_{b}\\
			\Longrightarrow ~ ~ ~
			&\textbf{(E2)}~T^{q}_{\mu^{-1}_\cA \cup W_\cA,\mu^{-1}_\cB \cup W_\cB}(G_N) \leq 0\,,
		\end{split}
	\end{equation*}
	which is equivalent to
	\begin{equation*}
	\begin{aligned}
	&\mathbf{\neg}\textbf{(E2)}\quad\quad
	\exists s^* \in \mu^{-1}_\cA \cup W_\cA:~~\hphantom{ -}G_N(s^*) > q^+(s^*)\\
	&\hphantom{\mathbf{\neg}\textbf{(E2)}} ~\,\vee\,~
	\exists s^* \in \mu^{-1}_\cB \cup W_\cB:~ -G_N(s^*) > q^-(s^*)\\
	\Longrightarrow ~ ~
	&\mathbf{\neg}\textbf{(E1)}~
	\exists a\in \Gamma(\cA):\quad
	  	\hat{\cU}_{a_q} = \Big\{ G_N(s) > \frac{a(s) - \mu(s)}{\tau_N\sigma(s)} + q^+(s) \Big\} \not\subseteq \cU_{a}\\
		  &\quad\quad \vee\,~ \exists b\in \Gamma(\cB):\quad
		\hat{\cL}_{b_q} = \Big\{ G_N(s) < \frac{b(s) - \mu(s)}{\tau_N\sigma(s)} - q^-(s) \Big\} \not\subseteq {\cL}_{b}\,.
	\end{aligned}
	\end{equation*}	
	
	To prove this, we first assume $s^* \in \mu^{-1}_{\cA}$ such that $G_N(s^*) > q^+(s^*)$.
	Hence there is $ a \in \cA $ such that $\mu(s^*) = a(s^*)$. Thus,
	$s^* \notin \cU_{a}$, yet $G_N(s^*) > q^+(s^*)$ implies $s^* \in \hat{\cU}_{a_q }$.
	the case $s^* \in \overline{U}_{\cA} \setminus \mu^{-1}_{\cA}$.
	Thus, assume $s^* \in W_{\cA}$
	such that $G_N(s^*) > q^+(s^*)$. By the continuity assumption in the definition of $W_{\cA}$ we find an
	$s'\in \mu^{-1}_{\cA}$ such that $G_N(s') > q^+(s')$. This, again implies
	the existence of an $a \in\cA$ such that $s' \notin \cU_{a}$,
	but $s' \in \hat{\cU}_{a_q }$.	
	
	The case that $s^* \in \mu^{-1}_{\cB} \cup W_\cB$ such that $-G_N(s^*) > q^-(s^*)$
	implies $\mathbf{\neg}\textbf{(E1)}$ is almost identical to the previous argument
	and therefore omitted. Together we have proven $\textbf{(E1)} \Rightarrow \textbf{(E2)}$.
	
	Consequentially, as $\Prb^*\big( \Omega_{cont}^{W_\cA} \cap \Omega_{cont}^{W_\cB} \big) = 1$ holds, we have that
	\begin{equation*}
	\begin{split}
			&\Prb^*\Big[\,
					\forall a \in \Gamma\big( \cA \big)\, \forall b \in \Gamma\big( \cB \big):~
					 	\hat{\cU}_{a_q} \subseteq \cU_{a}
					 ~ ~ \wedge ~ ~
						 \hat{\cL}_{b_q} \subseteq \cL_{b} 
					  \,\Big]\\
			 &= \Prb^*\Big[\,
					\left\{\forall a \in \Gamma\big( \cA \big)\, \forall b \in \Gamma\big( \cB \big):~
					 	\hat{\cU}_{a_q} \subseteq \cU_{a}
					 ~ ~ \wedge ~ ~
						 \hat{\cL}_{b_q} \subseteq \cL_{b}\right\} \cap \Omega_{cont}
					  \,\Big]\\
					 &\leq \Prb^*\Big[\, \left\{ T^{q}_{\mu^{-1}_\cA \cup W_\cA,\mu^{-1}_\cB \cup W_\cB}(G_N) \leq 0 \right\} \cap  \Omega_{cont} \,\Big]\\
					 &=  \Prb^*\Big[\, T^{q}_{\mu^{-1}_\cA \cup W_\cA,\mu^{-1}_\cB \cup W_\cB}(G_N) \leq 0 \,\Big]\,.
	\end{split}
	\end{equation*}	
	Applying the limit superior to both sides and using the Portmanteau Theorem yields the claim.
\end{proof}
\begin{remark}\label{eq:WeakendUpperBoundSCoREmeta}
	As we tried to make the upper bound as tight as possible we needed to introduce the sets $W_\cA$ and $W_\cB$ where
	$G_N$ is continuous from the inside of $\mu^{-1}_\cA$ and $\mu^{-1}_\cB$ respectively.
	However, setting $W_\cA = W_\cB = \emptyset$ we always get an upper bound that at most misses points from
	$\partial \mu^{-1}_\cA \setminus \mu^{-1}_\cA$ and $\partial \mu^{-1}_\cB \setminus \mu^{-1}_\cB$.
	In particular, if $ \mu^{-1}_\cA$ and $ \mu^{-1}_\cB$ are closed it holds that $W_\cA = W_\cB = \emptyset$
	and the upper bound is tight.
\end{remark}

The proof of the SCoRE Set Metatheorem can be thought of as taking the limit of the following
non-asymptotic bounds.

\begin{proposition}\label{prop:NonAsymMetaTheorem}
	Let $ \cA, \cB \subseteq \mathcal{F}(S) $, $ q^\pm\in\mathcal{F}(S) $
	and $\eta_N < \tilde \eta$.
	Define
	$c^\pm_q = c^\pm \pm q\tau_N\sigma$ for all $c^\pm \in \Gamma\big( \cC^\pm \big)$.
	Then
		\begin{equation*}
		\begin{split}
		\Prb_*\Bigg[\,
				&\forall c^- \in \Gamma\big( \cA \big)\,
				\forall c^+ \in \Gamma\big( \cB \big):~
			 \hat{\cL}_{c_q^-} \subseteq \cL_{c^-}
			 ~ \wedge ~ ~
			 \hat{\cU}_{c_q^+} \subseteq \cU_{c^+}
			  \,\Bigg]\\
			 &\begin{cases}
			 \geq
  			 \Prb_*\Big[\,
			 		T^{q}_{U^{-\eta_N}_{\cA}, U^{+\eta_N}_{\cB}}( G_N ) < 0
 			~\wedge~
 		  				T^{q^-, q^+}_{S, S}( G_N )
 		  	 			<  \mathfrak{O}\eta_N \tau_N^{-1}
 		  	 \,\Big]\\
 		  	 \leq
  			\Prb^*\Big[\,
			 		T^{q}_{\mu^{-1}_{\cA} \cup W_{\cA},\mu^{-1}_{\cB} \cup W_{\cB}}( G_N ) \leq 0
 			\,\Big]
			 \end{cases}\,.
 	\end{split}
	\end{equation*}
\end{proposition}
\begin{remark}
	 Using Lemma \ref{lem:LowerBoundInnerProb} it can be easily verified that
	 Lemma 2.1 from \cite{Mammen:2013} is the special case of the above proposition
	 with $\cC^\pm = \{0\}$.
	 The benefit of our version is that it more clearly shows that the probability of the lower
	 bound needs to be tuned by $q^\pm$ to obtain valid non-asymptotic control of the
	 inclusions and the gap in exact control is given by the difference
	 between the given upper and lower bound.
\end{remark}

\FloatBarrier
\newpage

%-------------------------------------------------------------------------------------------------------
\section{A condition for $\mathfrak{u}^{-1}_\cC = \mu^{-1}_\cC$}\label{App:frakSeqS}
%-------------------------------------------------------------------------------------------------------

In the main article we introduced the generalized preimage $\mathfrak{u}^{-1}_\cC$
which collects all points in $S$ such that either
$\big( s,\mu(s) \big)\in \Gamma(\cC)$ or $\big( s,\mu(s) \big)$ is a touching point
of $\Gamma(\cC)$ in the sense that the graph $\Gamma(\mu)$ gets arbitrary close to $\Gamma(\cC)$. 
The sharp upper bound \textbf{(A4)} in Theorem \ref{thm:MainSCoPES} holds true
if $\mathfrak{u}^{-1}_\cC = {\rm cl}\mu^{-1}_\cC$. Therefore we now discuss fairly
general conditions under which $\mathfrak{u}^{-1}_\cC = \mu^{-1}_\cC$. A key concept we
will need is the boundary of a set $\cC \subseteq \mathcal{F}(S)$. 
\begin{definition}[\textbf{Boundary of a Set of Functions}]
	A boundary $\partial\cC $ of $ \cC \subseteq \mathcal{F}(S) $ is a set
	$\mathcal{D} \subseteq  \mathcal{F}(S)$ such that
	\begin{equation*}
		\Gamma(\mathcal{D}) = \Bigg\{ (s,r) \in S \times \mathbb{R} ~\big\vert~
			r \in \partial\Bigg( \bigcup_{c\in \cC} c(s) \Bigg) \Bigg\}\,.
	\end{equation*}
	Here $\partial I$ is the topological boundary of $I\subseteq\mathbb{R}$ under
	the standard topology.
	While the boundary of $\cC$ is not a unique set, any two boundaries $\mathcal{D}$
	and $\mathcal{D}'$ satisfy that
	$\Gamma(\mathcal{D}) = \Gamma(\mathcal{D}')$.
	Therefore $\Gamma\big( \partial\cC \big)$ is a unique set.
\end{definition}
The next lemma generalizes Lemma 1 from \cite{Sommerfeld:2018CoPE} and implies
$\mathfrak{u}^{-1}_\cC = \mu^{-1}_\cC$.
%-------------------------------------------------------------------------------------------------------
\begin{lemma} \label{lem:HausdorffCompactConvergence}
    Let $\cC \subset \mathcal{F}(S)$, $\mu^{-1}_{\cC_{\tilde{\eta}}}$
    be compact for some $\tilde\eta>0$ and $\big( \eta_N \big)_{n\in \mathbb{N}}\subset\mathbb{R}$ a positive
    zero-sequence. Assume that $ \Gamma\big( \partial\cC \big) $
    is closed and the restriction of $\mu$ to
    $ {\rm cl}\,\mu^{-1}_{\cC_{\tilde{\eta}}}\setminus {\rm int}\,\mu^{-1}_\cC$
	is continuous. Then
    \begin{equation*}
    	\lim_{ N \rightarrow \infty } d_H\Big( \mu^{-1}_{\cC_{\eta_N} }, \mu^{-1}_{ \cC } \Big) = 0\,. 
    \end{equation*}
\end{lemma}
%-------------------------------------------------------------------------------------------------------
\begin{proof}
        Define the set $M^\varepsilon:=\{ s\in
        S:~\inf_{s' \in \mu^{-1}_{ \cC }} d( s, s' ) <\varepsilon\}$.
        By definition $ \mu^{-1}_{ \cC } \subset \mu^{-1}_{\cC_{\eta_N} } $
        for all $ N $. Therefore convergence in Hausdorff distance of
        $ \mu^{-1}_{\cC_{\eta_N} } $ to $ \mu^{-1}_{ \cC } $
        follows, if for any $ \varepsilon > 0 $
        there exists an $ N_0 \in \mathbb{N} $ such that for all
        $ N > N_0$ it holds that $\mu^{-1}_{\cC_{\eta_N} } \subset
        M^\varepsilon$. To this end, assume the contrary. Then, there
        exists  $ \varepsilon > 0 $ such that for some subsequence
        $ (N_k)_{k\in\mathbb{N}} $ there are $ s_{ N_k } \in
        \mu^{-1}_{\cC_{\eta_{N_k}} } \setminus~ \mu^{-1}_{ \cC } $
         with $ \inf_{s\in \mu^{-1}_{ \cC }} d(s_{ N_k }, s) \geq \varepsilon $.
         Since for large enough $k$ the sequence $ (s_{ N_k } )_{k\in\mathbb{N}} $ is contained
         in the compact set $ \mu^{-1}_{\cC_{\tilde{\eta}}} $ it can w.l.o.g.
         be assumed that it converges to a limit
         $ s^* \in S $, say.
		 Assume that $ k $ is large enough such that
         $ \mu^{-1}_{\cC_{\eta_{N_k}} } \setminus~ \mu^{-1}_{ \cC } \subseteq
		 \mu^{-1}_{\cC_{\eta_{0}} }  \setminus~ \mu^{-1}_{ \cC } $.
		 Since $ \mu $ is continuous on the closed set
		 ${\rm cl}\,\mu^{-1}_{\cC_{\tilde{\eta}}}\setminus {\rm int}\,\mu^{-1}_\cC$
		 it follows that
         $ \mu(s_{N_k}) \rightarrow \mu(s^*)$ for some
		 $ s^* \in {\rm cl}\,\mu^{-1}_{\cC_{\tilde{\eta}}}\setminus {\rm int}\,\mu^{-1}_\cC$.
		 By construction there is a sequence
         $ \big( s_{N_k}, c_{N_k}(s_{N_k}) \big) \in \Gamma(\partial\cC) $ such that
         \begin{equation}\label{eq:lemhausdorfconv1}
         	\vert \mu(s_{N_k}) - c_{N_k}(s_{N_k}) \vert
         			\leq \eta_{N_k} \xrightarrow{k\rightarrow\infty}0.
         \end{equation}
         Thus, $ \lim_{ k \rightarrow \infty } c_{N_k}(s_{N_k}) = \mu(s^*)$. Since
         $\Gamma\big( \partial\cC \big) $ is closed and
         $ \big(s_{N_k}, c_{N_k}(s_{N_k}) \big) \in \Gamma\big( \partial\cC \big) $
         for all $ k \in \mathbb{N} $, its limit
         $ \big( s^*, \mu(s^*) \big) $ is contained in $\Gamma\big( \partial\cC \big) $.
         By \eqref{eq:lemhausdorfconv1} it follows that
         $ \big( s^*, \mu(s^*) \big) \in \Gamma( \partial\cC) \cap \Gamma(\mu) $.
         This implies $ s^* \in \mu^{-1}_\cC $,
         since $\Gamma( \partial\cC)$ is closed. A contradiction.
\end{proof}
%-------------------------------------------------------------------------------------------------------
\begin{remark}
	The importance of $\Gamma( \partial\cC)$ being closed is visualized in Fig.
	\ref{fig:GlosedGammaC} which gives an
	example for $\cC = \{c\}$ where $c$ is discontinuous at $s_0\in S$.
	The problem here is that $\Gamma(\mu)$ intersects
	$\Gamma\big(\partial\{c\} \big) = \Gamma(c)$ at a point which does
	not belong to $\Gamma(c)$. This situation can be circumvented by finding a
	set $\tilde{\cC}$ with
	$\Gamma\big( \tilde{\cC} \big) = {\rm cl}\big( \Gamma(\cC) \big)$.
	This means adding functions such that their graph may pass through the boundary points
	of $\Gamma(\cC)$ while otherwise being contained in
	$\Gamma\big( \cC \big)$. This is illustrated in the right panel of Fig. \ref{fig:GlosedGammaC}.
	%-------------------------------------------------------------------------------------------------------
	\begin{figure}[b!]
		\begin{center}
			\includegraphics[width=0.49\textwidth]{\figurepath 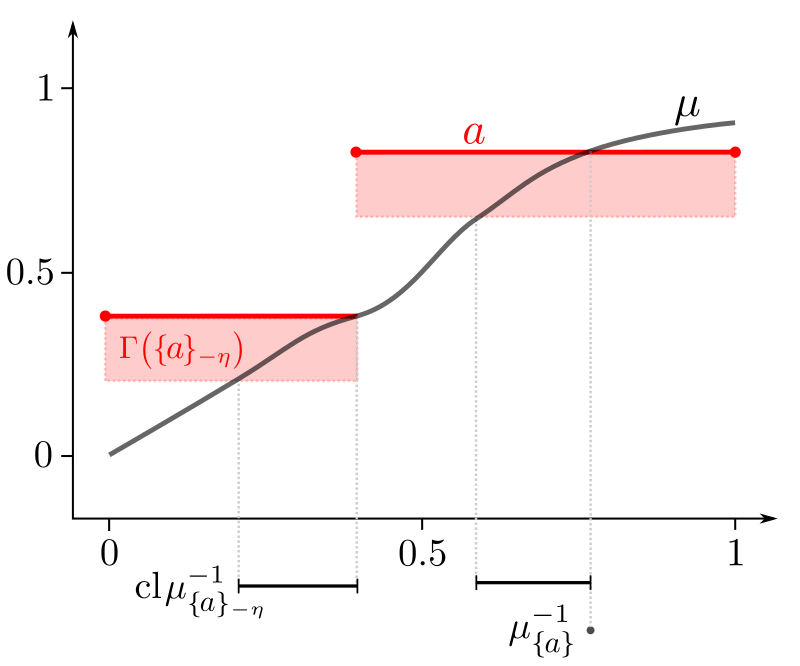}
			\includegraphics[width=0.49\textwidth]{\figurepath 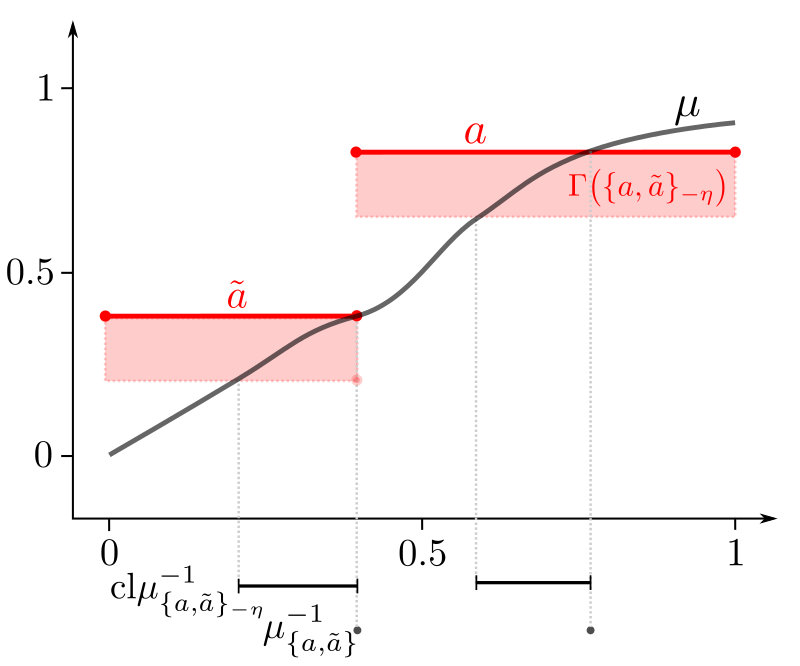}
		\end{center}
		\caption{Illustration of the assumption that $\Gamma(\partial\cC)$ needs to be closed
				for Hausdorff convergence in Lemma \ref{lem:HausdorffCompactConvergence}. \textit{Left:} 
			    Since $a$ is only right continuous at $s_0\approx 0.4$ and $\mu(s_0) = \lim_{s\rightarrow -s_0}a(s)$
			    it holds that ${\rm cl}\mu^{-1}_{a_{-\eta}}$ does not converge in Hausdorff distance
			    to $\mu^{-1}_a$.
			    \textit{Right:}
			    If $\cA=\{a\}$ is extended to include the function $\tilde a$
			    which satisfies $\tilde a(s) = a(s)$ for all $s\in S\setminus\{s_0\}$
			    and $\tilde a(s_0) =\lim_{s\rightarrow -s_0}a(s)$, then ${\rm cl}\mu^{-1}_{\{a,\tilde a\}_{-\eta}}$
			    does converge to $\mu^{-1}_{\{a,\tilde a\}}$ in Hausdorff-distance.
			    }\label{fig:GlosedGammaC}
	\end{figure}
\end{remark}
%-------------------------------------------------------------------------------------------------------

\FloatBarrier
\newpage

\section{The Difference between $">"$ or $"\geq"$ Excursion Sets}\label{scn:ChoiceInclusion}
%-------------------------------------------------------------------------------------------------------
To simplify notations let us define $B^\complement = S \setminus B$ for all $B\subseteq S$.
This section explains why the inclusions
$
	\hat{\cU}_{a+q\tau_N\sigma} \subseteq \cU_{a}
$
and
$
	\hat{\cL}_{b-q\tau_N\sigma} \subseteq \cL_{b}
$
are more natural than
$
	\hat{\cL}_{a+q\tau_N\sigma}^\complement \subseteq \cL_{a}^\complement
$
and
$
	\hat{\cU}_{b-q\tau_N\sigma}^\complement \subseteq \cU_{b}^\complement
$
which are used in the literature. For example,
\cite{Sommerfeld:2018CoPE, Bowring:2019, Bowring:2021,Maullin:2022} assume
$a(s) = b(s) = c \in \mathbb{R}$ for all $s\in S$ and study the probability of the inclusions
$
	\hat{\cL}_{c+q\tau_N\sigma}^\complement
	\subseteq \cL_c^\complement
	\subseteq \hat{\cL}_{c-q\tau_N\sigma}^\complement
$
.
The main issue with excursion sets using "$\geq$" instead of "$>$" is
that an "open ball" or "non-tangentiality" condition is required
to prove sharp upper bounds.
This condition is restrictive since it means that $\mu$ cannot
be flat on $\Gamma\big( \partial \cC \big)$
as illustrated in Fig. \ref{fig:Cope_Def_Problem}.
Because this section only serves an illustrative purpose, we simplify the
proof by assuming $\mathfrak{u}^{-1}_\cC = \mu^{-1}_\cC$.
This means that $\mu^{-1}_\cC \subseteq S$ is closed.
Furthermore, we require the following assumptions:
% Assumptions for old CoPE sets
%-------------------------------------------------------------------------------------------------------
\begin{itemize}[leftmargin=1.4cm]
	\item[\textbf{(A2')}] There exist an $\tilde{\eta}>0$ such that the
						   restriction of $G_N$  to $ \mu^{-1}_{\cC_{\tilde{\eta}}} $ has
						   almost surely continuous sample paths.
	\item[\textbf{(A5)}] Assume that for all open
	$\mathcal{O}_s \subset S$ containing
	$ s \in \mu^{-1}_{\partial\cC^-} \cap \mu^{-1}_{\cC^-} $
	there is a $c^- \in\cC^-$ such that $\mathcal{O}_s \cap \cU_{c^-} \neq \emptyset$
	and for all  open $\mathcal{O}_s \subset S$ containing
	$ s \in \mu^{-1}_{\partial\cC^+} \cap \mu^{-1}_{\cC^+} $
	there is a $c^+ \in\cC^+$ such that $\mathcal{O}_s \cap \cL_{c^+} \neq \emptyset$.
\end{itemize}
%-------------------------------------------------------------------------------------------------------
\begin{remark}
	Condition \textbf{(A2')} is stronger than Condition
	\textbf{(A2)}, since the latter only required continuity from outside
	$\mu^{-1}_\cC$ towards $\mu^{-1}_{\partial\cC}$ of $G_N$
	while \textbf{(A2')} requires additionally the continuity from the inside
	$\mu^{-1}_\cC$ towards $\mu^{-1}_{\partial\cC}$.
\end{remark}
%-------------------------------------------------------------------------------------------------------
\begin{remark}
	If $ \cC = \{ c \} $ for $ c \in \mathcal{F}(S) $ then
	$\mu^{-1}_{\partial\cC} = \mu^{-1}_{\cC} = \big\{ s\in S:~ \mu(s) = c(s) \big\}$
	and \textbf{(A5)} is equivalent to the open ball condition of Assumption 2.1.(a)
	from \cite{Sommerfeld:2018CoPE}.  
\end{remark}
%-------------------------------------------------------------------------------------------------------
\begin{figure}[b!]
			\includegraphics[width=0.49\textwidth]{\figurepath 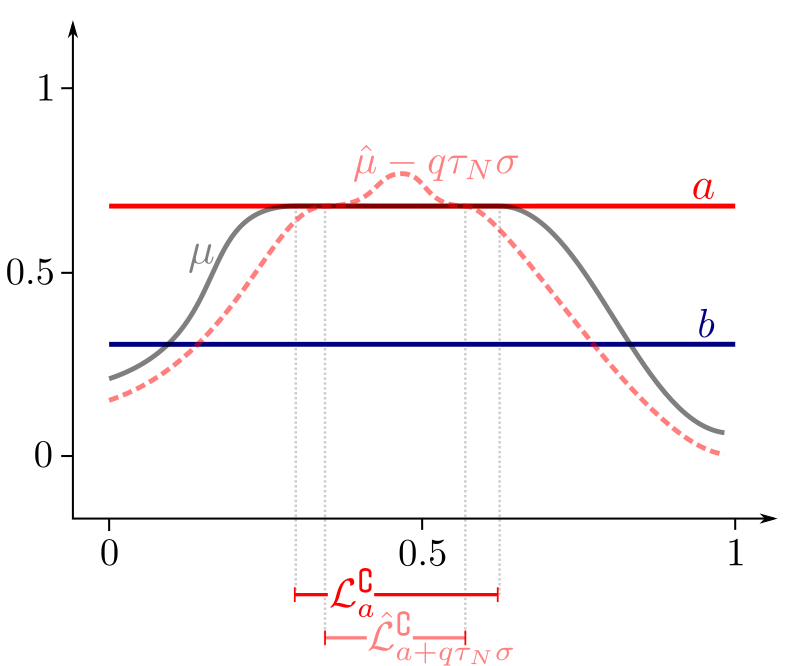}
			\includegraphics[width=0.49\textwidth]{\figurepath 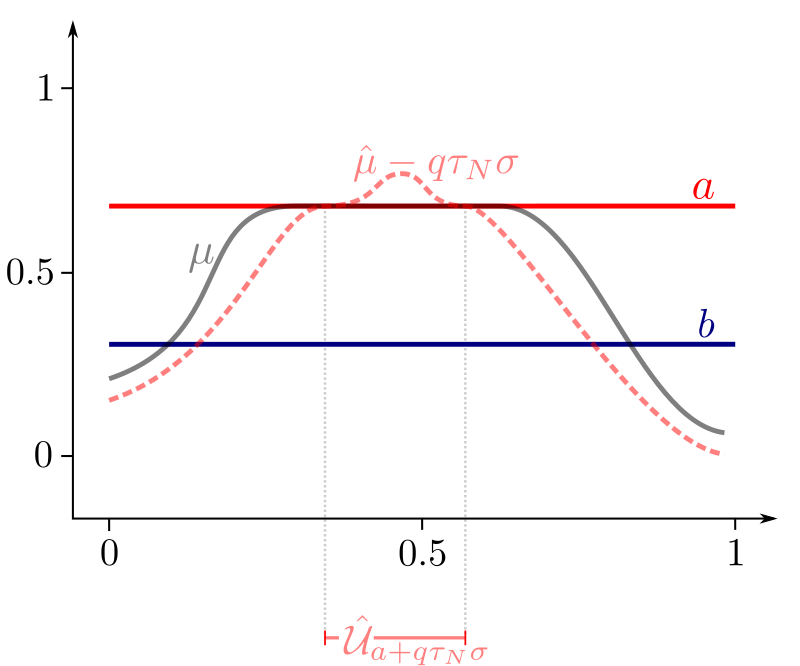}
		\caption{Illustration that
				 $
				 S \setminus \hat{\cL}_{a+q\tau_N\sigma} \subseteq S \setminus \cL_{a}
				 $ requires more assumptions to obtain a sharp upper bound than
				 $
				 \hat{\cU}_{a+q\tau_N\sigma} \subseteq \cU_{a}
				 $.
				 \textit{Left:} the inclusion $
				 S \setminus \hat{\cL}_{a+q\tau_N\sigma} \subseteq S \setminus \cL_{a}
				 $ does hold true although $\hat\mu_N - q\tau_N\sigma > \mu$ near $s\approx 0.5$.
				 \textit{Right:} the problem does not appear because $\cU_a=\emptyset $ and
				 therefore $\hat{\cU}_{a + q \tau_N \sigma} \not\subseteq \cU_{a}$.
				 }\label{fig:Cope_Def_Problem}
\end{figure}
%-------------------------------------------------------------------------------------------------------
\begin{theorem}
\label{thm:SimCoPEsetOld}
Let $ \cA,\cB \subseteq \mathcal{F}(S) $ and assume {\rm\textbf{(A1)}}-{\rm\textbf{(A3)}}.
Then
\begin{equation*}
	\liminf_{N \rightarrow \infty}
		\Prb_* \left[\, \forall a \in \cA\,\forall b \in \cB:~
		\hat{\cL}_{a + q\tau_N\sigma}^\complement \subseteq \cL_{a}^\complement
		~\wedge~
		\hat{\cU}_{b - q\tau_N\sigma}^\complement \subseteq \cU_{b}^\complement
	 \,\right]
		 \geq \Prb\left[\,  \mathfrak{T}_{\cA, \cB}(G) < q \,\right]\,. 
\end{equation*}
If also either $ \mu^{-1}_{\partial\cA} \cap \mu^{-1}_\cA = \emptyset $ and
$ \mu^{-1}_{\partial\cB} \cap \mu^{-1}_\cB = \emptyset $ or else \textbf{(A2')}
and {\rm\textbf{(A5)}}, then
\begin{equation*}
	\limsup_{N \rightarrow \infty}
		\Prb_* \left[\, \forall a \in \cA\,\forall b \in \cB:~
		\hat{\cL}_{a + q\tau_N\sigma}^\complement \subseteq \cL_{a}^\complement
		~\wedge~
		\hat{\cU}_{b - q\tau_N\sigma}^\complement \subseteq \cU_{b}^\complement
	 \,\right]
		 \leq \Prb\left[\,  \mathfrak{T}_{\cA, \cB}(G) \leq q \,\right]\,.
\end{equation*}
\end{theorem}
%-------------------------------------------------------------------------------------------------------
\begin{proof}
	The proof of the lower bound is the same as the corresponding proof in
	Theorem \ref{thm:MainSCoPES} and therefore omitted.
	The proof of the upper bound is similar to the the proof of the upper
	bound in the SCoRE Set Metatheorem \ref{thm:SCoPESMetatheorem}.
	The main difference is that we need to prove 
	\begin{equation*}
	\begin{aligned}
	&\textbf{(E1)}~\forall a \in \cA\,\forall b \in \cB:~
	\big\{ \hat{\mu}_{N}(s) + q\tau_{N} \sigma(s) \leq b(s) \big\}
						\subseteq  \big\{ \mu(s) \leq b(s) \big\}\\
	&\quad\quad\quad\quad\quad\quad\quad\quad\quad\quad\quad\wedge \big\{ \hat{\mu}_{N}(s)-a\tau_{N} \sigma(s)\geq a(s) \big\}
	 					\subseteq \big\{ \mu(s) \geq a(s) \big\}\\
	\Longrightarrow ~ ~ ~
	&\textbf{(E2)}~\forall s \in  \mu^{-1}_{\cA}\,\forall s' \in  \mu^{-1}_{\cB}:~\hat{\mu}_{N}(s)-q\tau_{N} \sigma(s)
				\leq \mu(s)\\
				 &\quad\quad\quad\quad\quad\quad\quad\quad\quad\quad\quad\quad\quad\quad\wedge~ \mu(s') \leq \hat{\mu}_{N}(s')+q\tau_{N} \sigma(s')
	\end{aligned}
	\end{equation*}
	for the subset of $\Omega$ where $\hat\mu_N$ has continuous sample paths on
	$\mu^{-1}_{\cC_{\tilde{\eta}}}$.
	
	Assume $ \textbf{(E1)} $ holds, but not $ \textbf{(E2)} $.
	W.l.o.g. assume there exist an $ s^*\in \mu^{-1}_{\cA} $
	such that
	$ \hat{\mu}_{N}(s^*) - q \tau_{N} \sigma(s^*) \geq \mu(s^*) $, i.e., $G_N(s^*) \geq q$.
	First we assume $ s^* \in \mu^{-1}_{\cA} \setminus \mu^{-1}_{\partial\cA}$.
	Since $ \mu^{-1}_{\cA} \setminus \mu^{-1}_{\partial\cA} $ is
	open, there exists $ a \in \cA $ such that $ a(s^*) \in \big(
	\mu(s^*), \hat{\mu}_{N}(s^*)-q\tau_{N} \sigma(s^*) \big) $.
	Thus, $ s^* \notin \cL_{a}^\complement $, but
	$ \hat{\mu}_{N}(s^* )-q\tau_{N} \sigma(s^* )\geq a(s^*)$,
	which contradicts $ \textbf{(E2)} $.	
	On the other hand, for $ s^* \in \mu^{-1}_{\partial\cA} \cap \mu^{-1}_{\cA} $,
	choose $\mathcal{O}_{s^*}$ small enough such that $G_N(s) \geq q$ for all $s \in \mathcal{O}_{s^*}$.
	This is possible by \textbf{(A2')}. Then \textbf{(A5)} guarantees the existence of $a$
	such that 
	$
 		s' \in \mathcal{O}_{s^*} \cap \cL_{a}
 	$. Thus, $s' \notin \cL_{a}^\complement$, but $G_N(s') \geq q$. A contradiction.	
	Similarly, the case $ s^*\in \mu^{-1}_{\cB} $ can be treated. Hence as in the proof of
	Theorem \ref{thm:SCoPESMetatheorem} an application of the Portmanteau
	Theorem finishes the proof.
%-------------------------------------------------------------------------------------------------------
\end{proof}

\FloatBarrier
\newpage
%-------------------------------------------------------------------------------------------------------
%-------------------------------------------------------------------------------------------------------
\section{Proofs of the Results in Section \ref{scn:SimCoPE}}
%-------------------------------------------------------------------------------------------------------
%-------------------------------------------------------------------------------------------------------

%-------------------------------------------------------------------------------------------------------
\subsection{Proof of Theorem \ref{thm:MainSCoPES}}
%-------------------------------------------------------------------------------------------------------
\begin{proof}
	We want to apply the SCoRE Set Metatheorem \ref{thm:SCoPESMetatheorem} in the case that $q^\pm\equiv q$.
	By {\rm\textbf{(A1)}}-{\rm\textbf{(A3)}} we obtain from Lemma \ref{lem:Tweak}
	\begin{equation*}
	\begin{split}
		T^{q,q}_{\cA_{-\eta_N}, \cB_{+\eta_N}}( G_N )
		 	\rightsquigarrow&
		\mathfrak{T}_{\cA,\cB}( G ) - q\,,\quad
		T^{q,q}_{\cA, \cB}( G_N )
		 	\rightsquigarrow
		T^{q,q}_{\mathcal{\cA}, \cB}( G )\\
		&T^{q,q}_{\cA_{-\tilde\eta},\cB_{+\tilde\eta}}( G_N )
		 	\rightsquigarrow
		T^{q,q}_{\cA_{-\tilde\eta},\cB_{+\tilde\eta}}( G )\,.
	\end{split}
	\end{equation*}
	Therefore {\rm\textbf{(M1)}} is satisfied.
	Condition {\rm\textbf{(M2)}} holds since the random variable
	$T^{q,q}_{\cA_{-\tilde\eta}, \cB_{+\tilde\eta}}( G )$ on $\mathbb{R}$ is tight which
	is equivalent to $T^{q,q}_{\cA_{-\tilde\eta}, \cB_{+\tilde\eta}}( G_N )$
	being asymptotically tight \cite[Lemma 1.3.8 ]{Vaart:1996weak}. Assumption {\rm\textbf{(M3)}} is
	part of Assumption {\rm\textbf{(A1)}}.
\end{proof}
%-------------------------------------------------------------------------------------------------------

%-------------------------------------------------------------------------------------------------------
\subsection{Proof of Theorem \ref{thm:SCest}}
%-------------------------------------------------------------------------------------------------------
We first establish that the estimate $\hat{\mathfrak{u}}_\cC^{\pm}$ is with inner
probability tending to one inside the set ${\rm cl}\,\mu^{-1}_{\cC_{\pm\tilde{\eta}}}$.
\begin{lemma}\label{lem:InclusionEstimator}
	Assume $\hat\mu_N$ satisfies Definition \ref{def:SuffWeakConv}(ii) for some $\tilde{\eta}>0$.
	Let $(k_N)_{N\in\mathbb{N}}$ be a positive sequence such that $\lim_{N\rightarrow\infty}k_N\tau_N = 0$.
	Then
	\begin{equation*}
		\liminf_{N\rightarrow\infty} \Prb_*\Big[\,
						\hat{\mathfrak{u}}_{\cC}^{\pm}
							\subseteq
						{\rm cl}\,\mu^{-1}_{\cC_{\pm\tilde{\eta}}}
					\,\Big] = 1\,.
	\end{equation*}
\end{lemma}
%-------------------------------------------------------------------------------------------------------
\begin{proof}
	Define
	$
		{\rm sgn}_{c,f}(s) = {\rm sgn}\big( f(s) - c(s) \big)
	$
	for $ f \in \mathcal{F}(S) $ and the set
	\begin{equation}\label{eq:estSC_approx}
		\hat{\mu}_\cC^{\pm}
			= \big\{\, s\in S~ \big\vert~ \exists c \in \cC:~
					0 \leq \mp \big( \hat\mu_N(s) - c(s) \big) \leq k_N\tau_N\sigma(s) \,\big\}\,.
	\end{equation}	
	 Note that ${\rm cl}\hat{\mu}_\cC^{\pm} \subseteq  \hat{\mathfrak{u}}_\cC^{\pm} $.
	 Assume that
	 $
	 s \in \hat{\mu}_\cC^{\pm} \cap
	 \big( S \setminus \mu^{-1}_{\cC_{\pm\tilde{\eta}}} \big)
	 $.
	 From Definition \ref{def:SuffWeakConv}(ii) and \eqref{eq:estSC_approx} we obtain	 
	\begin{equation*}
	\begin{split}
		\mathfrak{o}\big(\, K + Z_N(s) \,\big)
		\leq {\rm sgn}_{c,\mu}(s)\big( \hat\mu_N(s) - c(s) \big)
		\leq k_N\tau_N\sigma(s)\,,
	\end{split}
	\end{equation*}
	which implies
	$
		Z_N(s) \leq \tau_Nk_N\mathfrak{O}\mathfrak{o}^{-1} -  K
	$.
	Since $\inf_{s\in S} \tau_N^{-1}Z_N(s) $ is asymptotically
	tight and $\lim \tau_Nk_N = 0$ it holds that
	\begin{equation*}
		\limsup_{N\rightarrow\infty}\Prb^* \left[\,
					Z_N(s) \leq \tau_Nk_N\mathfrak{O}\mathfrak{o}^{-1} -  K
				\,\right]
%		 = \limsup_{N\rightarrow\infty}\Prb^* \Bigg[\, \frac{Z_N(s)}{\tau_N} \leq \frac{ \tau_Nk_N\mathfrak{O}\mathfrak{o}^{-1} -  K}{\tau_N} \,\Bigg]
		 = 0
	\end{equation*}
	combining this with Remark \ref{rmk:SignInequal} shows that
	\begin{equation*}
	\begin{split}
		\limsup_{N\rightarrow\infty}\Prb^*\Big[\,
						\hat{\mu}_\cC^{\pm} 
							\cap
						\big( S \setminus \mu^{-1}_{\cC_{\pm\tilde{\eta}}} \big)
						\neq \emptyset
					\,\Big]
				= 0
		&~~\Leftrightarrow ~~
		\liminf_{N\rightarrow\infty}\Prb_*\Big[\, \hat{\mu}_\cC^{\pm} \subseteq \mu^{-1}_{\cC_{\pm\tilde{\eta}}} \,\Big]
				= 1\\
		&~~\Leftrightarrow ~~
		\liminf_{N\rightarrow\infty}\Prb_*\Big[\, \hat{\mathfrak{u}}_\cC^{\pm} \subseteq {\rm cl}\,\mu^{-1}_{\cC_{\pm\tilde{\eta}}} \,\Big]
				= 1
	\end{split}
	\end{equation*}
	which is the claim.
\end{proof}

Using the above Lemma we can prove Theorem \ref{thm:SCest}. The idea of the proof is inspired by
the proof of Theorem 3.6 from \cite{Dette:2020functional}.
\begin{proof}
	We begin with proving the Hausdorff convergence of the estimator
	$\hat{\mathfrak{u}}^{\pm}_\cC$ of $\mathfrak{u}^{\pm}_\cC$.
	To begin with let $\eta_N = \rho\mathfrak{o}k_N\tau_N$ for some $ \rho \in (0,1) $
	and recall the definition of
	$\hat{\mu}_\cC^{\pm}$ from \eqref{eq:estSC_approx}. Assuming $N$ large enough such that
	$\eta_N < \tilde\eta$ we obtain
	\begin{equation*}
	\begin{split}
		\Prb_*\Big[\,
					\mu^{-1}_{\cC_{\pm\eta_N} } \subseteq \hat{\mu}_\cC^{\pm}
					\,\Big]
			&= \Prb_*\Bigg[\,
					\sup_{s \in \mu^{-1}_{\cC_{\pm\eta_N}}} \inf_{c \in \cC}
						\max\left\{0,\,
						\mp \frac{\hat\mu_N(s) - c(s) }{\sigma(s)} \right\}\leq k_N\tau_N
						\,\Bigg] \\
			&= \Prb_*\Bigg[\,
						\sup_{s \in \mu^{-1}_{\cC_{\pm\eta_N}}} \inf_{c \in \cC}
						\max\left\{0,\,
						\mp \frac{ \hat\mu_N(s) - \mu(s) }{\sigma(s)} \mp   \frac{\mu(s) - c(s)}{\sigma(s)}
						 \right\} \leq k_N\tau_N
						\,\Bigg] \\
			&\geq \Prb_*\Bigg[\, \sup_{s \in \mu^{-1}_{\cC_{\pm\eta_N}}} \inf_{c \in \cC}
			\max\left\{0,\,
					\mp \frac{\hat\mu_N(s) - \mu(s) }{\sigma(s)} + \rho k_N\tau_N
					\right\} \leq k_N\tau_N \,\Bigg] \\
			&= \Prb_*\Bigg[\, \sup_{s \in \mu^{-1}_{\cC_{\pm\eta_N}}} \vert G_N(s) \vert  \leq k_N (1-\rho) \,\Bigg] \\
			&\geq \Prb_*\Bigg[\, \sup_{s \in {\rm cl}\,\mu^{-1}_{ \cC_{\pm\tilde{\eta}} } } \vert G_N(s) \vert \leq k_N (1-\rho) \,\Bigg]\,.
	\end{split}
	\end{equation*}
	The latter converges to one since
	$\sup_{s \in {\rm cl}\,\mu^{-1}_{ \cC_{\pm\tilde{\eta}} } } \vert G_N(s) \vert $
	is asymptotically tight. Therefore we have
	$\liminf_{N\rightarrow\infty}\Prb_*\big[\, {\rm cl}\,\mu^{-1}_{\cC_{\pm\eta_N}} \subseteq \hat{\mathfrak{u}}_\cC^{\pm} \,\big]= 1$ and thus
	\begin{equation*}
	\begin{split}
	\limsup_{N\rightarrow\infty}
		\Prb^*\Bigg[ \sup_{s \in \mathfrak{u}^{\pm}_\cC} \inf_{s' \in \hat{\mathfrak{u}}_\cC^{\pm}} \vert s - s' \vert > \epsilon \Bigg]
	\leq \limsup_{N\rightarrow\infty}
		\Prb^*\Bigg[ \sup_{s \in {\rm cl}\,\mu^{-1}_{\cC_{\pm\eta_N}}} \inf_{s' \in \hat{\mathfrak{u}}_\cC^{\pm}} \vert s - s' \vert > \epsilon \Bigg]
		= 0\,.
	\end{split}
	\end{equation*}
%	Essentially the identical argument proves
%	\begin{equation}
%		\liminf_{N\rightarrow\infty}\Prb_*\Big[ \mathfrak{u}^{\pm}_{\cC}
%		\subseteq \hat{\mathfrak{u}}_{\cC}^{\pm} \Big] = 1\,.
%	\end{equation}	

	It remains to prove that
	\begin{equation}\label{eq:222a}
	\begin{split}
	\limsup_{N\rightarrow\infty}
		\Prb^*\Bigg[ \sup_{s \in \hat{\mathfrak{u}}_\cC^{\pm}} \inf_{s' \in \mathfrak{u}^{\pm}_\cC} \vert s - s' \vert > \epsilon \Bigg]	= 0\,.
	\end{split}
	\end{equation}
	
	To see this note that whenever
	$
	\hat{\mathfrak{u}}_\cC^{\pm} \subseteq {\rm cl}\,\mu^{-1}_{\cC_{\pm\tilde{\eta}}}
	$
	it holds that
	\begin{equation*}
	\begin{split}
		\sup_{s\in \hat{\mu}_\cC^{\pm}} \inf_{c \in \cC} \max\left\{0, \mp \big( \mu(s) - c(s) \big) \right\} 
			&\leq \sup_{s\in \hat{\mu}_\cC^{\pm}} \inf_{c \in \cC} \big\vert \mu(s) - \hat\mu_N(s) \big\vert \mp \big(  \hat\mu_N(s) - c(s) \big)  \\
			&\leq \sup_{s\in \hat{\mu}_\cC^{\pm}}\vert \mu(s) - \hat\mu_N(s) \vert + k_N\tau_N\mathfrak{O}  \\
			&\leq \sup_{s\in {\rm cl}\,\mu^{-1}_{\cC_{\pm\tilde{\eta}}}}\vert \mu(s) - \hat\mu_N(s) \vert + k_N\tau_N\mathfrak{O}
	\end{split}
	\end{equation*}
	which implies by Lemma \ref{lem:InclusionEstimator} that
	\begin{equation*}
	\begin{split}
		\liminf_{N\rightarrow\infty} \Prb_*\Bigg[
					&\sup_{s\in \hat{\mu}_\cC^{\pm}}
					\inf_{c \in \cC} 
					\max\left\{0, \mp \big(\mu(s) - c(s) \big)\right\} \leq 2k_N\tau_N\mathfrak{O}
				\Bigg]\\
			&\geq \liminf_{N\rightarrow\infty}\Prb_*\Bigg[ \sup_{s\in {\rm cl}\,\mu^{-1}_{\cC_{\pm\tilde{\eta}}}}\vert \mu(s) - \hat\mu_N(s) \vert \leq k_N\tau_N\mathfrak{O} \Bigg] \\
			&\geq \liminf_{N\rightarrow\infty}\Prb_*\Bigg[ \sup_{s\in \mu^{-1}_{\cC_{\pm\tilde{\eta}}}}\vert G_N(s) \vert \leq k_N \Bigg] = 1\,.
	\end{split}
	\end{equation*}
	Thus,
	$
	\sup_{s\in \hat{\mu}_\cC^{-1}} \inf_{c \in \cC} \max\left\{0, \mp \big(\mu(s) - c(s) \big) \right\}
	$
	converges to zero outer almost surely. It remains to show that
	this implies that
	$
	\sup_{s' \in \hat{\mathfrak{u}}_\cC^{\pm}} \inf_{s \in \mathfrak{u}^{\pm}_\cC} \vert s - s' \vert \rightarrow 0
	$
	outer almost surely as the the result of the Theorem then follows from \cite{Vaart:1996weak}.
	
	Let us assume that $
	\sup_{s\in \hat{\mu}_\cC^{\pm}} \inf_{c \in \cC} \vert \mu(s) - c(s) \vert
	$
	converges to zero outer almost surely. By Egorov's Theorem
	(Lemma 1.9.2(iii) from \cite{Vaart:1996weak})
	this is equivalent to
	\begin{equation}\label{eq:E1SCest}
	\begin{split}
		\forall \delta > 0\, \exists A\subseteq \Omega\,\forall\varepsilon>0\,
		&\forall\omega\in A\,\exists N'\, \forall
		N > N':\\
		&\Prb(A) \geq 1-\delta ~\wedge~ \sup_{s\in \hat{\mu}_\cC^{\pm}(\omega,N)} \inf_{c \in \cC} \vert \mu(s) - c(s) \vert < \epsilon
	\end{split}
	\end{equation}
	and assume that $
	\sup_{s' \in \hat{\mathfrak{u}}_\cC^{\pm}} \inf_{s \in \mathfrak{u}^{\pm}_\cC} \vert s - s' \vert \rightarrow 0
	$
	is not converging outer almost surely, i.e.,
	\begin{equation}\label{eq:E2SCest}
	\begin{split}
		\exists \delta > 0\, \forall A \subseteq \Omega\, \exists\varepsilon_0>0\,
		&\exists\omega_0\in A\,\forall N'\, \exists
		N > N':\\
		&\Prb(A) \geq 1-\delta ~\wedge~
		\sup_{s' \in \hat{\mathfrak{u}}_\cC^{\pm}(\omega_0,N)}
		\inf_{s \in \mathfrak{u}^{\pm}_\cC} \vert s - s' \vert \geq \epsilon_0\,.
	\end{split}
	\end{equation}	
	For clarity we explicitly added the dependence on $\omega$, $\omega_0$ and $N$ here.	
	
	Let $\delta$ from \eqref{eq:E2SCest}. For this $\delta$	 we choose the corresponding
	$A$ from \eqref{eq:E1SCest}. Hence by  \eqref{eq:E2SCest} there is a $\varepsilon_0$
	and $\omega_0\in A$ such that for all $N'>0$  there is an $N>N'$ satisfying
	\begin{equation}\label{eq:SCest3}
		\sup_{s' \in \hat{\mathfrak{u}}_\cC^{\pm}(\omega_0, N)}\inf_{s \in \mathfrak{u}^{\pm}_\cC} \vert s - s' \vert \geq \epsilon_0\,.
	\end{equation}
	
	We will now construct a contradiction between \eqref{eq:E1SCest} and \eqref{eq:E2SCest}.	
	From \eqref{eq:SCest3} we obtain a sequence
	$ (s_{N'})_{{N'}\in\mathbb{N}} \subset S $ satisfying, for some $N>N'$,
	\begin{equation*}
		s_{N'} \in \hat{\mathfrak{u}}_\cC^{\pm}(\omega_0, N)\,,\quad \quad \inf_{s \in \mathfrak{u}^{\pm}_\cC} \vert s - s_{N'} \vert \geq \epsilon_0\,.
	\end{equation*}	
	Let us define the set
	$
	\tilde\cN = \big\{ N \in \mathbb{N}~\vert~\exists N'\in\mathbb{N}:\,s_{N'} \in \hat{\mathfrak{u}}_\cC^{\pm}(\omega_0, N) \big\}
	$. As $S$ is compact, we can by replacing $(s_N)_{N\in\mathbb{N}}$ by a convergent subsequence
	w.l.o.g. assume that $s_N \rightarrow s^*$ for $N\rightarrow \infty$ with
	\begin{equation}\label{eq:SCest4}
		s^* \in \bigcap_{N \in \tilde\cN} {\rm cl} \hat{\mu}_\cC^{\pm}(\omega_0,N)
		   \quad\text{and}\quad
		\inf_{s \in \mathfrak{u}^{\pm}_\cC} \vert s - s^* \vert \geq \epsilon_0\,.
	\end{equation}
	
	If $s^* \in \bigcap_{N \in \tilde\cN} \hat{\mu}_\cC^{\pm}(\omega_0,N)$,
	then  \eqref{eq:SCest4} immediately contradicts \eqref{eq:E1SCest}. More general,
	if $s^* \in \bigcap_{N \in \tilde\cN} {\rm cl}\hat{\mu}^{\pm}_\cC(\omega_0,N)$
	we find a sequence 	$(t_N)_{N\in\mathbb{N}}$ such that $t_N \in \hat{\mu}^{\pm}_\cC(\omega_0,N)$
	converging to $s^*$. As $\mathfrak{u}^{\pm}_\cC$ is a compact set this implies that
	$\inf_{s \in \mathfrak{u}^{\pm}_\cC} \vert s - t_N \vert \geq \epsilon_1>0$ for all $N$ large enough.
	The latter contradicts \eqref{eq:E1SCest}, which finishes the proof.

\end{proof}

\FloatBarrier
\newpage

%-------------------------------------------------------------------------------------------------------
%-------------------------------------------------------------------------------------------------------
\section{Proofs of the Results in Section \ref{scn:testing}}
%-------------------------------------------------------------------------------------------------------
%-------------------------------------------------------------------------------------------------------
\FloatBarrier

%-------------------------------------------------------------------------------------------------------
\subsection{Proof of Proposition \ref{prop:CoPEasSCB}}
%-------------------------------------------------------------------------------------------------------
\begin{proof}
	It is enough to show that the equivalence of
	\begin{equation*}
	\begin{aligned}
		&\textbf{(E1)}~\forall s \in  S :~  \hat l_N(s)  \leq \mu(s) \leq \hat{u}_N(s)\\
		&\textbf{(E2)}~\forall c \in \mathcal{F}( S ):~
			\hat{\cL}_{c} \subseteq \cL_c
			~\wedge~
			\hat{\cU}_{c} \subseteq \cU_c\,.
	\end{aligned}
	\end{equation*}	
	\emph{Case }\textbf{(E1)}$\Rightarrow$\textbf{(E2)}:
	Assume that $s^* \in \hat{\cU}_{c}$, i.e., $c(s^*) < l_N(s^*)$.
	By \textbf{(E1)} it holds that $ l_N(s^*) \leq \mu(s^*) $ which implies $c(s^*) < \mu(s^*) $,
	i.e., $s^* \in \hat{\cU}_{c}$.
	Similar, $s^* \in \cL_{c}$ is proven.\\
	\emph{Case }\textbf{(E2)}$\Rightarrow$\textbf{(E1)}: By \textbf{(E2)} it holds for $c = \mu$ that
	$\hat{\cU}_{\mu} \subseteq \cU_\mu = \emptyset$ and
	$S = S\setminus \cL_\mu \subseteq S \setminus \hat{\cL}_{\mu}$ which is equivalent to \textbf{(E1)}.
\end{proof}

%-------------------------------------------------------------------------------------------------------------
\subsection{Proof of Theorem \ref{thm:Dette_Generalized}}
%-------------------------------------------------------------------------------------------------------
\begin{proof}
	Note that $
		\cL_{b+\delta} = \cU_{a+\delta} = \emptyset
	$.
	
	If $\delta \leq 0$, then $\mathbf{H}_{0}$ is true.
	%-------------------------------------------------------------------------------------------------------
	Moreover, by the definition of $\delta$ we have that
	\begin{equation}\label{eq:hypGlobalRel}
	\begin{split}
		&\mathbb{P}^*\big[ \mathbf{H}_{0} \text{ is rejected } \big]\\
		&= \mathbb{P}^*\big[
					\hat{\cU}_{a + q_{\alpha}\tau_N\sigma} \neq \emptyset
					~ \vee ~
					\hat{\cL}_{b + q_{\alpha}\tau_N\sigma} \neq \emptyset
				 \big]\\
		&= \mathbb{P}^*\big[
					\hat{\cU}_{a + q_{\alpha}\tau_N\sigma}  \not\subseteq \cU_{a + \delta}
					~ \vee ~
					\hat{\cL}_{b + q_{\alpha}\tau_N\sigma} \not\subseteq \cL_{b + \delta}
				 \big]\\
		&= 1 - \mathbb{P}_*\big[
					\hat{\cU}_{a + q_{\alpha}\tau_N\sigma}  \subseteq \cU_{a + \delta}
					~ \wedge ~
					\hat{\cL}_{b + q_{\alpha}\tau_N\sigma} \subseteq \cL_{b + \delta}
				 \big]\,.
	\end{split}
	\end{equation}
	We first prove statement $(a)$. As $\delta = 0 $ the claim follows directly from \eqref{eq:hypGlobalRel}
	and Corollary \ref{cor:Extraction} for $\cA = \{ a \}$ and $\cB = \{ b \}$.
	
	Statement $(b)$ follows from \eqref{eq:hypGlobalRel} and Lemma \ref{lem:CoPEContainLemma}
	as $\delta<0$. Statement $(c)$ follows from Lemma \ref{lem:CoPEContainLemma} as well, because for $\delta >0$ and we have
		\begin{equation*}
	\begin{split}
		\mathbb{P}_*\big[ \mathbf{H}_{0} \text{ is rejected } \big]
		&= \mathbb{P}_*\big[
					\hat{\cU}_{a + q_{\alpha}\tau_N\sigma} \neq \emptyset
					~ \vee ~
					\hat{\cL}_{b + q_{\alpha}\tau_N\sigma} \neq \emptyset
				 \big]\\
		&= 1 - \mathbb{P}^*\big[
					\hat{\cU}_{a + q_{\alpha}\tau_N\sigma}  \subseteq \cU_{a + \delta}
					~ \wedge ~
					\hat{\cL}_{b + q_{\alpha}\tau_N\sigma} \subseteq \cL_{b + \delta}
				 \big]\,.
	\end{split}
	\end{equation*}
\end{proof}
%-------------------------------------------------------------------------------------------------------

%-------------------------------------------------------------------------------------------------------
\subsection{Proof of Theorem \ref{thm:RelTubeTest}}
%-------------------------------------------------------------------------------------------------------
\begin{proof}
	%-------------------------------------------------------------------------------------------------------
	It is helpful to remember the definitions of the sets
	$$
	\mathcal{H}_0 = S \setminus (\, \cU_{a} \cup \cL_{b} \,)\quad\text{ and }\quad
	\mathcal{H}_1 = \cU_{a} \cup \cL_{b}
	$$
	of true null hypotheses and true alternative hypotheses respectively.
	
	%-------------------------------------------------------------------------------------------------------
	We first prove $(a)$.
	As the definition of $\delta$ implies
	$$
	\cU_{a} = \cU_{a - \delta}\,, ~ ~ ~ ~
	\cL_{b} = \cL_{b + \delta}\,,
	$$
	it follows that
	\begin{equation*}
	\begin{split}
		&\Prb^*\big[\, \exists s \in \mathcal{H}_0:~ s \in \hat{\mathcal{H}}_1 \,\big]\\
		&= \Prb^*\Big[\, \exists s \in \mathcal{H}_0:~
					s \in
					\hat{\cU}_{a+q_\alpha\tau_N \sigma}
					~\cup~
					\hat{\cL}_{b-q_\alpha\tau_N \sigma}
				\,\Big]\\
		&= 1 -  \Prb_*\Big[\, \forall s \in \mathcal{H}_0:~
					s \in
					S \setminus \big(\, \hat{\cU}_{a+q_\alpha\tau_N \sigma}
					~\cup~
					\hat{\cL}_{b-q_\alpha\tau_N \sigma}
					 \,\big)
				\,\Big]\\
		&\leq 1 -  \Prb_*\Big[\, \forall s \in \mathcal{H}_0:~
					s \in
					S \setminus 
					\big(\,
					\hat{\cU}_{a-\delta+q_\alpha\tau_N \sigma}
					~\cup~
					\hat{\cL}_{b+\delta-q_\alpha\tau_N \sigma}
					 \,\big)
				\,\Big]\\
		&= 1 -  \Prb_*\Big[\,
					\mathcal{H}_0
					\subseteq
					S \setminus \big(\,
						\hat{\cU}_{a-\delta+q_\alpha\tau_N \sigma}
									~\cup~
						  \hat{\cL}_{b+\delta-q_\alpha\tau_N \sigma}
					\,\big)
				\,\Big]\\
		&\leq 1 -  \Prb_*\Big[\,
						\hat{\cU}_{a-\delta+q_\alpha\tau_N \sigma} \subseteq \cU_{a-\delta}
					~\wedge~
						\hat{\cL}_{b+\delta-q_\alpha\tau_N \sigma} \subseteq \cL_{b+\delta}
					 \,\Big]\,.
	\end{split}
	\end{equation*}
	Thus, applying the limes superior to both sides, Corollary \ref{cor:Extraction} with
	$\cA = \big\{ a - \delta \big\}$ and $\cB = \big\{ b + \delta \big\}$ yields
	\begin{equation*}
	\begin{split}
		\limsup_{N\rightarrow\infty}
		\Prb^*\big[\, \exists s \in \mathcal{H}_0:~ s \in \hat{\mathcal{H}}_1 \,\big]
			\leq 1 -  \Prb\Big[\, \mathfrak{T}_{a-\delta, b+\delta}(G) < q_\alpha \,\Big]
			 \leq \alpha\,.
	\end{split}
	\end{equation*}
	The second inequality is a consequence of the first and the observation that
	\begin{equation*}
	\begin{split}
		\Prb^*\big[\, \exists s \in \mathcal{H}_0:~ s \in \hat{\mathcal{H}}_1 \,\big]
		= 1 - \Prb_*\big[\, \mathcal{H}_0 \subseteq \hat{\mathcal{H}}_0  \,\big]
		= 1 - \Prb_*\big[\, \hat{\mathcal{H}}_1 \subseteq \mathcal{H}_1 \,\big]\,.
	\end{split}		
	\end{equation*}
	This finishes the proof of $(a)$.
	
	%-------------------------------------------------------------------------------------------------------
	We now prove $(b)$.	Note that $\delta = 0$.
	%-------------------------------------------------------------------------------------------------------
	The case $\mathcal{H}_0 = S$ implies
	$
	\cU_{a} = \emptyset
	$
	and
	$
	\cL_{b} = \emptyset
	$.
	Again Corollary \ref{cor:Extraction} with
	$\cA = \big\{ a \big\}$ and $\cB = \big\{ b \big\}$ yields
	\begin{equation*}
	\begin{split}
		\liminf_{N\rightarrow\infty}\Prb_*\big[\, \exists s \in \mathcal{H}_0:~ s \in \hat{\mathcal{H}}_1 \,\big]
			&= 1 - \limsup_{N\rightarrow\infty}
				\Prb^*\Big[\,
					\hat{\cU}_{a+q_\alpha\tau_N \sigma} = \emptyset	
					~\wedge~
					\hat{\cL}_{b-q_\alpha\tau_N \sigma} = \emptyset
				\,\Big]\\
			&\geq 1 -  \Prb\Big[\, \mathfrak{T}_{a-\delta, b+\delta} (G) \leq q_\alpha \,\Big]\,.
	\end{split}
	\end{equation*}
	
	%-------------------------------------------------------------------------------------------------------
	To prove the case $\inf_{s\in S} a(s) - b(s) \geq M > 0$ we observe that
	\begin{equation*}
	\begin{split}
		&\Prb_*\Big[\,
			\hat{\cU}_{a+q_\alpha\tau_N \sigma} \subseteq \cU_{a}
			~\wedge~
			\hat{\cL}_{b-q_\alpha\tau_N \sigma} \subseteq \cL_{b}
			\,\Big]\\
		&\geq \Prb_*\Big[\,
					\mathcal{H}_0
					\subseteq
					S \setminus \big(\,
						\hat{\cU}_{a+q_\alpha\tau_N \sigma}
									~\cup~
						\hat{\cL}_{b-q_\alpha\tau_N \sigma}
					\,\big)
					\,\Big]\\
		&~~~+ \Prb_*\Big[\, \hat{\cU}_{a+q_\alpha\tau_N \sigma} \subseteq \cU_{b} \,\Big]
			+ \Prb_*\Big[\, \hat{\cL}_{b-q_\alpha\tau_N \sigma} \subseteq \cL_{a}  \,\Big]
			- 2 \,,
	\end{split}
	\end{equation*}
 	since $\mathcal{H}_0 = S \setminus (\, \cU_{a} \cup \cL_{b} \,)$.
 	Using this yields
	\begin{equation*}
	\begin{split}
		\Prb^*\big[\, \exists s \in \mathcal{H}_0:~ s \in \hat{\mathcal{H}}_1 \,\big]
		&= 1 -  \Prb_*\Big[\,
					\mathcal{H}_0
					\subseteq
					S \setminus \big(\,
						\hat{\cU}_{a+q_\alpha\tau_N \sigma}
									~\cup~
						 \hat{\cL}_{b-q_\alpha\tau_N \sigma}
					\,\big)
				\,\Big]\\
		&\geq - 1
			- \Prb_*\Big[\,
				\hat{\cU}_{a+q_\alpha\tau_N \sigma} \subseteq \cU_{a}
					~\wedge~
				\hat{\cL}_{b-q_\alpha\tau_N \sigma} \subseteq \cL_{b}
			\,\Big]\\
		&~~~ + \Prb_*\Big[\, \hat{\cU}_{a+q_\alpha\tau_N \sigma} \subseteq \cU_{b} \,\Big]
			 + \Prb_*\Big[\,
					\hat{\cL}_{b-q_\alpha\tau_N \sigma} \subseteq \cL_{a}
				\,\Big]\,.
	\end{split}
	\end{equation*}
	Applying the limes superior part of Corollary \ref{cor:Extraction} with
	$\cA = \big\{ a \big\}$ and $\cB = \big\{ b \big\}$
	and Lemma \ref{lem:CoPEContainLemma}, which
	shows that the last two probabilities on the r.h.s. converge to $1$,
	finishes the proof of $(b)$.

	%-------------------------------------------------------------------------------------------------------
	To prove the second claim of $(b)$. Recall the definitions
	\begin{equation}\label{eq:DeltapmN}
		\Delta^a_N(s) = \tfrac{\mu(s) - a(s)}{\tau_N \sigma(s)}\,\quad\text{ and }\quad
		\Delta^b_N(s) = \tfrac{\mu(s) - b(s)}{\tau_N \sigma(s)}\,.
	\end{equation}
	Hence $ \mathcal{H}_1^{\eta_N} = \cU_{a+\eta_N} \cup \cL_{b-\eta_N} $
	implies that
	\begin{equation*}
		\inf_{s \in \cU_{a+\eta_N}} \Delta^a_N(s) \geq \tfrac{\eta_N}{\mathfrak{O}\tau_N}\,,~~~
		\sup_{s \in \cL_{b-\eta_N}}	\Delta^b_N(s) \leq -\tfrac{\eta_N}{\mathfrak{O}\tau_N}\,.
	\end{equation*}
	Thus,
	\begin{equation*}
	\begin{split}
		&\Prb_*\big[\, \forall s \in \mathcal{H}_1^{\eta_N}:~s \in \hat{\mathcal{H}}_1 \,\big]\\
		&= \Prb_*\Big[\,
					\cU_{a+\eta_N} \cup\,\cL_{b-\eta_N} 
						\subseteq
					 \hat{\cU}_{a+q_\alpha\tau_N \sigma} \cup\hat{\cL}_{b-q_\alpha\tau_N \sigma} 
					\,\Big]\\
%--------------------------------------------------------------------------------------------------------------------------------
		&\geq \Prb_*\Big[\,
				 \cU_{a+\eta_N} \subseteq \hat{\cU}_{a+q_\alpha\tau_N \sigma} 
				 	~\wedge~
				 \cL_{b-\eta_N} \subseteq \hat{\cL}_{b-q_\alpha\tau_N \sigma}
				\,\Big]\\
%--------------------------------------------------------------------------------------------------------------------------------
		&= \Prb_*\Bigg[\,
				\inf_{ s \in \cU_{a+\eta_N}} G_N(s) + \Delta^+_N(s) > q_\alpha
					~\wedge~
				\sup_{ s \in \cL_{b-\eta_N}} G_N(s) + \Delta^-_N(s) < -q_\alpha
				\,\Bigg]\\
%--------------------------------------------------------------------------------------------------------------------------------
		&\geq \Prb_*\Bigg[\,
				\inf_{ s \in \cU_{a}} G_N(s) + \tfrac{\eta_N}{\mathfrak{O}\tau_N} > q_\alpha
					~\wedge~
				\sup_{ s \in \cL_{b}} G_N(s) -\tfrac{\eta_N}{\mathfrak{O}\tau_N} < -q_\alpha
				\,\Bigg]\\
%--------------------------------------------------------------------------------------------------------------------------------
		&\geq \Prb_*\Bigg[\,
			\inf_{ s \in \cU_{a} \cap {\rm cl}\,\mu^{-1}_{\{a\}_{\tilde{\eta}}}}
					\hspace{-0.3cm} G_N(s) > q_\alpha - \tfrac{\eta_N}{\mathfrak{O}\tau_N}
				~\wedge~\hspace{-0.3cm}
			\inf_{ s \in \cU_{a} \cap S \setminus {\rm cl}\,\mu^{-1}_{\{a\}_{\tilde{\eta}}} }
					 \hspace{-0.4cm}\frac{K + Z_N(s)}{\tau_N} > q_\alpha\\
	&~~~~~~~~~~\wedge~
			\sup_{ s \in \cL_{b} \cap {\rm cl}\,\mu^{-1}_{\{b\}_{\tilde{\eta}}}}\hspace{-0.3cm}
					 G_N(s)  < -q_\alpha + \tfrac{\eta_N}{\mathfrak{O}\tau_N}
				~\wedge~\hspace{-0.3cm}
			\sup_{ s \in  \cL_{b} \cap S \setminus {\rm cl}\,\mu^{-1}_{\{b\}_{\tilde{\eta}}}}
					 \hspace{-0.4cm}\frac{K + Z_N(s)}{\tau_N} < -q_\alpha
			\,\Bigg]\\%--------------------------------------------------------------------------------------------------------------------------------
		&\geq \Prb_*\Bigg[\,
				\inf_{ s \in \cU_{a} \cap {\rm cl}\,\mu^{-1}_{\{a\}_{\tilde{\eta}}}} G_N(s) > q_\alpha - \tfrac{\eta_N}{\mathfrak{O}\tau_N} \,\Bigg]
			+ 2\Prb_*\Bigg[\, \inf_{ s \in S} \frac{K + Z_N(s)}{\tau_N} > q_\alpha \,\Bigg]\\
	&~~~~~~+ \Prb_*\Bigg[\,
		\sup_{ s \in  \cL_{b} \cap {\rm cl}\,\mu^{-1}_{\{b\}_{\tilde{\eta}}}} G_N(s) < -q_\alpha + \tfrac{\eta_N}{\mathfrak{O}\tau_N} 
				\,\Bigg]
		  - 3\\
	&\xrightarrow{N\rightarrow\infty} 1\,,
	\end{split}
	\end{equation*}
	since all four inner probabilities on the last line converge to one, which follows from the fact that
	$$
		\inf_{s \in S} \tau_N^{-1} Z_N(s)\,, ~~~
		\inf_{ s \in \cU_{a} \cap {\rm cl}\,\mu^{-1}_{\{a\}_{\tilde{\eta}}}}
				\hspace{-0.3cm} G_N(s)\,, ~~~
		\inf_{ s \in \cL_{b} \cap {\rm cl}\,\mu^{-1}_{\{b^-\}_{\tilde{\eta}}}}
				\hspace{-0.3cm} -G_N(s)		
	$$
	are all asymptotically tight. For the latter two this follows from
	the weak convergence $G_N \rightsquigarrow G$ on ${\rm cl}\,\mu^{-1}_{\{a\}_{\tilde{\eta}}}$
	and Lemma \ref{lem:Tweak}.
	This means the proof of $(c)$ is complete.

	%-------------------------------------------------------------------------------------------------------
	Finally, we prove $(d)$.
	Since $\Delta > 0$ it holds that
	\begin{equation*}
		\inf_{s \in \cU_{a}} \Delta^a_N(s) \geq\tfrac{\Delta}{\mathfrak{O} \tau_N}\,,~ ~ ~
		\sup_{s \in \cL_{b}}	\Delta^b_N(s) \leq -\tfrac{\Delta}{\mathfrak{O}\tau_N}
	\end{equation*}
	The proof for $\mathcal{H}_1$ therefore is essentially identical to the proof
	of the last statement in part $(b)$.
    Moreover, we have that
	\begin{equation*}
	\begin{split}
		\Prb_*\big[\, \forall s \in \mathcal{H}_0:~ s \in \hat{\mathcal{H}}_0 \,\big]
		&= \Prb_*\Big[\,
				\mathcal{H}_0
				\subseteq
				S \setminus
				\big(\, \hat{\cU}_{a+q_\alpha\tau_N \sigma}
							 ~\cup~
						\hat{\cL}_{b-q_\alpha\tau_N \sigma} \,\big)
				 \,\Big]\\
		&\geq \Prb_*\Big[\,
				\hat{\cU}_{a+q_\alpha\tau_N \sigma} \subseteq \cU_{a-\delta}
					~\wedge~
				\hat{\cL}_{b-q_\alpha\tau_N \sigma}\subseteq \cL_{b+\delta}
				\,\Big]\\
		&\geq \Prb_*\Big[\,
					\hat{\cU}_{a+q_\alpha\tau_N \sigma} \subseteq \cU_{a-\delta}
					\,\Big]
				+ \Prb_*\Big[\,
					\hat{\cL}_{b-q_\alpha\tau_N \sigma} \subseteq \cL_{b+\delta} 
					\,\Big] - 1\,.
	\end{split}
	\end{equation*}
	Since $\inf_{s\in S} \big( a(s) - \delta  +a(s) \big) > \delta $ and  $\sup_{s\in S} \big( b(s) - b(s) - \delta  \big) < - \delta $ the claim follows from Lemma \ref{lem:CoPEContainLemma}.
\end{proof}

%-------------------------------------------------------------------------------------------------------------
\subsection{Proof of Theorem \ref{thm:Equiv_2sides}}
%-------------------------------------------------------------------------------------------------------
\begin{proof}
	If $\delta \geq 0$, then $\mathbf{H}_{0}$ is true, i.e.,
	$
		\cL_{b} = \cU_{a} = \emptyset
	$.
	%-------------------------------------------------------------------------------------------------------
	Moreover, by the definition of $\delta$ we have that
	\begin{equation}\label{eq:hypGlobalEquiv}
	\begin{split}
		\mathbb{P}^*\big[ \mathbf{H}_{0} \text{ is rejected } \big]
		&= \mathbb{P}^*\big[
					\hat{\cU}_{a + q_{\alpha}\tau_N\sigma} = \emptyset
					~ \wedge ~
					\hat{\cL}_{b + q_{\alpha}\tau_N\sigma} = \emptyset
				 \big]\\
		&= \mathbb{P}^*\big[
					\hat{\cU}_{a + q_{\alpha}\tau_N\sigma}  \subseteq \cU_{a + \delta}
					~ \wedge ~
					\hat{\cL}_{b + q_{\alpha}\tau_N\sigma} \subseteq \cL_{b + \delta}
				 \big]\,.
	\end{split}
	\end{equation}
	We first prove statement $(a)$. As $\delta = 0 $ the claim follows directly from \eqref{eq:hypGlobalEquiv}
	and Corollary \ref{cor:Extraction} for $\cA = \{ a \}$ and $\cB = \{ b \}$.
	
	Statement $(b)$ follows from \eqref{eq:hypGlobalEquiv} and Lemma \ref{lem:CoPEContainLemma}
	as $\delta>0$. Statement $(c)$ follows from Lemma \ref{lem:CoPEContainLemma} as well, because for $\delta <0$ and we have
		\begin{equation*}
	\begin{split}
		\mathbb{P}_*\big[ \mathbf{H}_{0} \text{ is rejected } \big]
		&= \mathbb{P}_*\big[
					\hat{\cU}_{a + q_{\alpha}\tau_N\sigma} = \emptyset
					~ \wedge ~
					\hat{\cL}_{b + q_{\alpha}\tau_N\sigma} = \emptyset
				 \big]\\
		&= \mathbb{P}_*\big[
					\hat{\cU}_{a + q_{\alpha}\tau_N\sigma}  \subseteq \cU_{a + \delta}
					~ \wedge ~
					\hat{\cL}_{b + q_{\alpha}\tau_N\sigma} \subseteq \cL_{b + \delta}
				 \big]\,.
	\end{split}
	\end{equation*}
\end{proof}
%-------------------------------------------------------------------------------------------------------

%-------------------------------------------------------------------------------------------------------------
\subsection{Proof of Theorem \ref{thm:lequivTest}}
%-------------------------------------------------------------------------------------------------------
\begin{proof}
	%-------------------------------------------------------------------------------------------------------
	Recall that
	$
		\mathcal{H}_0 = S \setminus \big( \cU_{b} \cap \cL_{a} \big)
	$
	and
	$
		\mathcal{H}_1 = \cU_{b} \cap \cL_{a} 
	$.
	
	%-------------------------------------------------------------------------------------------------------
	We first prove $(a)$.
	We obtain
	\begin{equation*}
	\begin{split}
		&\Prb^*\big[\, \exists s \in \mathcal{H}_0:~ s \in \hat{\mathcal{H}}_1 \,\big]\\
		&= 1 -  \Prb_*\big[\, \forall s \in \mathcal{H}_0:~ s \in \hat{\mathcal{H}}_0 \,\big]\\
		&= 1 -  \Prb_*\Big[\,
						S \setminus \big( \cU_{b} \cap \cL_{a} \big)
							\subseteq
						S \setminus \big(
									\hat{\cU}_{b+q_\alpha\tau_N\sigma}
								\cap
									\hat{\cL}_{a-q_\alpha\tau_N\sigma}
								\big)
							 \,\Big]\\
		&\leq 1 -  \Prb_*\Big[\,
						 \hat{\cU}_{b+q_\alpha\tau_N\sigma} \subseteq \cU_{b}
									 ~\wedge~
						\hat{\cL}_{a-q_\alpha\tau_N\sigma} \subseteq \cL_{a}
								  \,\Big]\\
		&\leq 1 -  \Prb_*\Big[\,
						 \hat{\cU}_{b-\delta+q_\alpha\tau_N\sigma} \subseteq \cU_{b-\delta}
									 ~\wedge~
						 \hat{\cL}_{a+\delta-q_\alpha\tau_N\sigma} \subseteq \cL_{a+\delta}
								  \,\Big]\,.
	\end{split}
	\end{equation*}
	Thus, applying the limes superior to both sides of the above inequality and Corollary \ref{cor:Extraction} with
	$\cA = \big\{ b-\delta \big\}$ and $\cB = \big\{ a + \delta \big\}$ yields the claim.
	The second inequality follows from the first and
	\begin{equation*}
	\begin{split}
		\Prb^*\big[\, \exists s \in \mathcal{H}_0:~ s \in \hat{\mathcal{H}}_1 \,\big]
		= 1 - \Prb_*\big[\, \mathcal{H}_0 \subseteq \hat{\mathcal{H}}_0  \,\big]
		= 1 - \Prb_*\big[\, \hat{\mathcal{H}}_1 \subseteq \mathcal{H}_1 \,\big]\,.
	\end{split}		
	\end{equation*}
	
%-------------------------------------------------------------------------------------------------------
	In order to proof $(b)$ we use our previous calculation we obtain
	\begin{equation*}
	\begin{split}
		&\Prb^*\big[\, \exists s \in \mathcal{H}_0:~ s \in \hat{\mathcal{H}}_1 \,\big]\\
				&= 1 -  \Prb_*\Big[\,
						S \setminus \big( \cU_{b} \cap \cL_{a} \big)
							\subseteq
						S \setminus \big(
									\hat{\cU}_{b+q_\alpha\tau_N\sigma}
								\cap
									\hat{\cL}_{a-q_\alpha\tau_N\sigma}
								\big)
							 \,\Big]\\
				&\geq 1 -  \Prb_*\Big[\,
							\big( 
							S \setminus\cU_{b} \subseteq S \setminus\hat{\cU}_{b+q_\alpha\tau_N\sigma}
									 		~\vee~
							S \setminus\cU_{b} \cap \hat{\cL}_{a-q_\alpha\tau_N\sigma}
							\neq \emptyset  
							 \big)
\\
				&~~~~~~~~~~~~~~~~~~\wedge~
							\big( S \setminus\cL_{a} \subseteq S \setminus\hat{\cL}_{a-q_\alpha\tau_N\sigma} 
									 		~\vee~
							S \setminus \cL_{a} \cap \hat{\cU}_{b+q_\alpha\tau_N\sigma}
								\neq \emptyset  
							\big)
				  \,\Big]\\
		&\geq 1 - \Prb_*\Big[\,
						\hat{\cU}_{b+q_\alpha\tau_N\sigma} \subseteq \cU_{b}
									 ~\wedge~
						 \hat{\cL}_{a-q_\alpha\tau_N\sigma} \subseteq \cL_{a}
					,\Big]
				- \Prb_*\big[\, C_N
								  \,\big]\,.
	\end{split}
	\end{equation*}
	Here
	\begin{equation*}
		C_N = \hat{\cU}_{b+q_\alpha\tau_N\sigma} \cap \big( S \setminus \cL_{a} \big)
								\neq \emptyset  
						 ~\vee~
				 \big( S \setminus \cU_{b} \big) \cap \hat{\cL}_{a-q_\alpha\tau_N\sigma}
							\neq \emptyset 
	\end{equation*}
	Applying the limes inferior to this inequality, Corollary \ref{cor:Extraction} with
	$\cA = \big\{ b-\delta \big\}$ and $\cB = \big\{ a + \delta \big\}$ on
	the first probability and recognizing that
	$ \Prb_*\big[\, C_N \,\big]$ converges to one by Lemma \ref{lem:IntersectionCope},
	since $\inf_{s \in S} \big( a(s) - b(s) \big) > 0$,
	yields the claim.

	For the second claim recall the notation introduced in \eqref{eq:DeltapmN}.
	Since $\mathcal{H}_1^{\eta_N} = \cU_{ b + \eta_N } \cap \cL_{a - \eta_N}  $
	implies that
	\begin{equation*}
		\inf_{s \in \cU_{b+\eta_N}} \Delta^b_N(s) \geq \tfrac{\eta_N}{\mathfrak{O}\tau_N}\,,~~~
		\sup_{s \in \cL_{a-\eta_N}}	\Delta^a_N(s) \leq -\tfrac{\eta_N}{\mathfrak{O}\tau_N}\,.
	\end{equation*}
	Thus,
	$
	\Prb_*[\, A \cap B \subseteq C \cap B \,]
		\geq \Prb_*[\, A \subseteq C ~\wedge~ B \subseteq C \,]
	$
	for all $A,B,C,D$ yields
	\begin{equation*}
	\begin{split}
		&\Prb_*\big[\, \forall s \in \mathcal{H}_1^{\eta_N}:~ s \in \mathcal{H}_1
						\,\big]\\
		&= \Prb_*\Big[\,
					\cU_{ b + \eta_N } \cap \cL_{a - \eta_N} 
					\subseteq
					\hat{\cU}_{b+q_\alpha\tau_N\sigma} \cap \hat{\cL}_{a-q_\alpha\tau_N\sigma}
					\,\Big]\\
		&\geq \Prb_*\Bigg[\,
				\inf_{ s \in \cU_{b+\eta_N}} G_N(s) + \Delta^b_N(s) > q_\alpha
								~\wedge~
				\sup_{ s \in \cL_{a-\eta_N} } G_N(s) + \Delta^a_N(s) < -q_\alpha
				\,\Bigg]
	\end{split}
	\end{equation*}	
	The rest of the proof is almost identical to the proof of Theorem \ref{thm:RelTubeTest}b).
	
%-------------------------------------------------------------------------------------------------------		
	\textbf{Proof of $(c)$:}
	The proof is similar to the proof of part (c) from Theorem \ref{thm:RelTubeTest}.
\end{proof}

\FloatBarrier
\newpage

%-------------------------------------------------------------------------------------------------------
%-------------------------------------------------------------------------------------------------------
\section{Proofs of the Results in Section \ref{scn:Scheffe}}
%-------------------------------------------------------------------------------------------------------
%-------------------------------------------------------------------------------------------------------
\begin{lemma}\label{lemma:maxScalarProduct}
	Let $ l\in \mathbb{R} $, $ \boldsymbol{a}, \boldsymbol{w} \in \mathbb{R}^K $
	and $  \tfrac{l^2}{\Vert a \Vert^2} \leq 1$.
	Let $ P = \big( \boldsymbol{v}_1^T,\ldots, \boldsymbol{v}_{k-1}^T \big) \in \mathbb{R}^{K-1 \times K}$,
	where $\boldsymbol{v}_1,\ldots, \boldsymbol{v}_{K-1}$ form a basis
	of the hyperplane $E = \{ \boldsymbol{x}\in \mathbb{R}^K:~\boldsymbol{a}^T\boldsymbol{x} = 0 \}$.
	Let $P_E^\bot = P(P^TP)^{-1}P^T$ the orthogonal projection onto $E$.
		\begin{equation*}
		\begin{split}
			a)&~~~ \max_{\substack{\Vert \boldsymbol{x} \Vert = 1\\ \boldsymbol{a}^T\boldsymbol{x} = l} } \pm \boldsymbol{x}^T\boldsymbol{w}
			= \pm l\frac{\boldsymbol{a}^T\boldsymbol{w} }{\Vert \boldsymbol{a} \Vert^2} + \sqrt{ 1 - \tfrac{l^2}{\Vert \boldsymbol{a} \Vert^2} } \big\Vert P_E^\bot\boldsymbol{w} \big\Vert\\
			b)&~~~ \max_{\substack{\Vert \boldsymbol{x} \Vert = 1\\ \boldsymbol{a}^T\boldsymbol{x} = l} }
				\vert \boldsymbol{x}^T\boldsymbol{w} \vert 
			= \left\vert l \frac{ \boldsymbol{a}^T\boldsymbol{w}  }{\Vert \boldsymbol{a} \Vert^2}  \right\vert +  \sqrt{ 1 - \tfrac{l^2}{\Vert \boldsymbol{a} \Vert^2} } \big\Vert P_E^\bot\boldsymbol{w} \big\Vert
		\end{split}
		\end{equation*}
\end{lemma}
\begin{proof}
 For $ \boldsymbol{x} \in \mathbb{R}^K $ we define $\hat{\boldsymbol{x}} = \boldsymbol{x} / \Vert \boldsymbol{x} \Vert$.
 Every $ \boldsymbol{x} \in \{ \boldsymbol{x} \in \mathbb{R}^K:~\boldsymbol{a}^T\boldsymbol{x} = l \}$
 can be decomposed into
 \begin{equation*}
 	\boldsymbol{x} = l\frac{\boldsymbol{a}}{\Vert \boldsymbol{a} \Vert^2} + \boldsymbol{x}_E\,,~~\text{with }\boldsymbol{x}_E \in E\,.
 \end{equation*} 
 Using this decomposition, the fact that $\hat{\boldsymbol{x}}_E^T\boldsymbol{w}$ does not depend on $\Vert \boldsymbol{x}_E \Vert$
 and that the smallest angle between the vector $\pm\boldsymbol{w}$ and the linear subspace $E$ is the angle between $\boldsymbol{w}$ and its orthogonal projection $P_E^\bot(\pm\boldsymbol{w})$ onto $E$, we obtain
 \begin{equation*}
 \begin{split}
 	\max_{\substack{\Vert \boldsymbol{x} \Vert = 1\\ \boldsymbol{a}^T\boldsymbol{x} = l} } \boldsymbol{x}^T(\pm\boldsymbol{w})
 	&= \max_{ \Vert \boldsymbol{x}_E \Vert^2 = 1 - \tfrac{l^2}{\Vert \boldsymbol{a} \Vert^2} } \pm l\frac{\boldsymbol{a}^T\boldsymbol{w}}{\Vert \boldsymbol{a} \Vert^2} + \boldsymbol{x}_E^T(\pm\boldsymbol{w})\\
 	&= \pm l\frac{\boldsymbol{a}^T\boldsymbol{w}}{\Vert \boldsymbol{a} \Vert^2}
 	  + \max_{ \Vert \boldsymbol{x}_E \Vert^2 = 1 - \tfrac{l^2}{\Vert \boldsymbol{a} \Vert^2} }  \Vert \boldsymbol{x}_E \Vert \hat{\boldsymbol{x}}_E^T(\pm\boldsymbol{w})\\
 	&= \pm l\frac{\boldsymbol{a}^T\boldsymbol{w}}{\Vert \boldsymbol{a} \Vert^2}
 	  + \sqrt{1 - \tfrac{l^2}{\Vert \boldsymbol{a} \Vert^2}} \max_{ \Vert \hat{\boldsymbol{x}}_E \Vert^2 = 1  } \hat{\boldsymbol{x}}_E^T(\pm\boldsymbol{w})\\
 	&= \pm l\frac{\boldsymbol{a}^T\boldsymbol{w}}{\Vert \boldsymbol{a} \Vert^2}
 	  + \sqrt{1 - \tfrac{l^2}{\Vert \boldsymbol{a} \Vert^2}} \frac{ \boldsymbol{w}^T P_E^\bot  }{ \big\Vert P_E^\bot \boldsymbol{w} \big\Vert} \boldsymbol{w}\\
 	&= \pm l\frac{\boldsymbol{a}^T\boldsymbol{w}}{\Vert \boldsymbol{a} \Vert^2}
 	  + \sqrt{1 - \tfrac{l^2}{\Vert \boldsymbol{a} \Vert^2}} \big\Vert P_E^\bot \boldsymbol{w} \big\Vert\,.
 \end{split}
 \end{equation*}
 This proves equation a). Finally, $b)$ is an immediate consequence of $a)$. 
\end{proof}
%-------------------------------------------------------------------------------------------------------
\subsection{Proof of Corollary \ref{thm:ScheffeSSSCoPE_zero}}
%-------------------------------------------------------------------------------------------------------
\begin{proof}
	Assumptions \textbf{(A1)}-\textbf{(A4)} are trivially satisfied in this setup, since $ \mathbb{S}^{K-1}$ is compact, $\mu$ and $c$ continuous
	and $G_N$ defined in \eqref{eq:ContrastULT} has continuous sample paths.
	Let $ \boldsymbol{\varepsilon} \sim \mathcal{N}(0, I_{K \times K})$.  Applying
	Theorem \ref{thm:MainSCoPES} and Lemma \ref{lemma:maxScalarProduct} yields
	\begin{equation*}
	\begin{split}
	\lim_{N \rightarrow \infty}
			\Prb \Big[\, \hat{\cU}_{\tau_N q \sigma } \subseteq \cU_0 ~\wedge~ \hat{\cL}_{-\tau_N q \sigma } \subseteq  \cL_0 \,\Big]
		 	&= \Prb\Bigg[~ \max_{ \substack{\Vert \boldsymbol{a} \Vert = 1\\ \boldsymbol{a}^T\boldsymbol{\beta} = 0 } }
				\Bigg\vert \frac{ \boldsymbol{a}^T\mathfrak{X}^{1/2}\boldsymbol{\varepsilon} }{ \sqrt{\boldsymbol{a}^T\mathfrak{X}\boldsymbol{a}}}  \Bigg\vert   \leq q ~\Bigg]\\
		 	&= \Prb\Big[~ 
				\left\Vert P(P^TP)^{-1}P^T\boldsymbol{\varepsilon} \right\Vert \leq q ~\Big]\,.
	\end{split}
	\end{equation*}
	If $\boldsymbol{\beta} = 0$ then
	$P = I_{ K \times K}$. If $\boldsymbol{\beta} \neq 0$, then
	$P\in\mathbb{R}^{K \times (K-1)}$ is a matrix
	having as columns a basis of the hyperplane
	$ \big\{ \boldsymbol{a} \in \mathbb{R}^K:~
	(\mathfrak{X}^{-1/2}\boldsymbol{\beta})^T\boldsymbol{a} = 0 \big\}
	$.
	The claim follows since $ P(P^TP)^{-1}P^T $ is an
	orthogonal projection and therefore
	$$P(P^TP)^{-1}P^T\boldsymbol{\varepsilon} \sim \mathcal{N}\big( 0, I_{(K-1 + \mathds{1}_{\boldsymbol{\beta}=0}) \times (K - 1 + \mathds{1}_{\boldsymbol{\beta}=0})} \big)\,.$$
\end{proof}

\FloatBarrier

\end{document}